\documentclass[12pt,letter]{amsart}

\usepackage[centering,text={16cm,23cm}]{geometry}
\usepackage{amsmath,amscd,amssymb,latexsym}
\usepackage{hyperref}
\usepackage[dvipsnames,hyperref]{xcolor}
\usepackage{amsfonts}
\usepackage{amsthm}
\usepackage{bm}
\usepackage{bbm}
\usepackage{caption}
\usepackage{bold-extra}
\usepackage{enumerate}
\usepackage[mathscr]{eucal}
\usepackage{mathrsfs}
\usepackage{mathtools}
\usepackage{csquotes}
\usepackage{stmaryrd}
\usepackage{braket}

\usepackage{tikz}
\usetikzlibrary{shapes.misc,arrows,automata}
\usepackage{tikz-cd}
\usepackage{graphicx}
\tikzset{cross/.style={cross out, draw=black, fill=none, minimum size=2*(#1-\pgflinewidth), inner sep=0pt, outer sep=0pt}, cross/.default={2pt}}

\hypersetup{
	colorlinks=false,
	linkcolor=Blue,
	citecolor=Periwinkle,
	filecolor=Magenta,      
	urlcolor=Periwinkle,
}

\makeatletter

\theoremstyle{definition}
\newtheorem*{defintro}{Definition}
\newtheorem{thmintro}{Theorem}

\newtheorem{theorem}{Theorem}[section]
\newtheorem{definition}[theorem]{Definition}
\newtheorem{proposition}[theorem]{Proposition}
\newtheorem{example}[theorem]{Example}

\newtheorem{lemma}[theorem]{Lemma}
\newtheorem{corollary}[theorem]{Corollary}
\newtheorem{conjecture}[theorem]{Conjecture}
\newtheorem{construction}[theorem]{Construction}

\theoremstyle{remark}
\newtheorem{remark}[theorem]{Remark}

\newcommand{\A}{\mathbb{A}}
\newcommand{\C}{\mathbb{C}}
\newcommand{\D}{\Delta}
\newcommand{\Dp}{\Delta^\star}
\newcommand{\F}{\mathbb{F}}
\renewcommand{\L}{\mathcal{L}}

\renewcommand{\O}{\mathcal{O}}
\renewcommand\P{\mathbb{P}}
\newcommand{\R}{\mathbb{R}}

\newcommand{\BP}{(B,\mathcal{P})}
\newcommand{\DIC}{(B,\mathcal{P},\varphi)}
\newcommand{\IC}{(\check{B},\check{\mathcal{P}},\check{\varphi})}

\newcommand{\PD}{\mathbb{P}_\Delta}

\newcommand{\XcDegen}{\check{\mathcal{X}}}

\newcommand{\YD}{\mathbb{P}_\Delta}

\def\const{{\sf Const}}

\begin{document}

	\title[GW invariants and mirror symmetry for non-Fano toric varieties]{Gromov--Witten invariants and mirror symmetry for non-Fano varieties via tropical disks}

\author{Per Berglund}
\address{University of New Hampshire \\ Department of Physics and Astronomy \\ 105 Main St, Durham, NH 03824 \\ USA} 
\email{per.berglund@unh.edu}

\author{Tim Gr\"afnitz}
\address{Leibniz-Universit\"at Hannover \\ Institut f\"ur Algebraische Geometrie \\ Welfengarten 1, 30167 Hannover \\ Germany} 
\email{graefnitz@math.uni-hannover.de}

\author{Michael Lathwood}
\address{University of New Hampshire \\ Department of Physics and Astronomy \\ 105 Main St, Durham, NH 03824 \\ USA} 
\email{michael.lathwood@unh.edu}


\begin{abstract}
Under mirror symmetry a non-Fano variety $X$ corresponds to an instanton corrected Hori-Vafa potential $W$. The classical period of $W$ equals the regularized quantum period of $X$, which is a generating function for descendant Gromov-Witten invariants. These periods define closed mirror maps relating complex with symplectic parameters and open mirror maps relating coordinates on the mirror curves. 

We interpret the corrections to $W$ by broken lines in a scattering diagram, so that $W$ is the primitive theta function $\vartheta_1$. We show that, after wall crossing to infinity and application of the closed mirror map, $W=\vartheta_1$ is equal to the open mirror map. By tropical correspondence, $\vartheta_1$ is a generating function for $2$-marked logarithmic Gromov-Witten invariants, which are algebraic analogues of counts of Maslov index $2$ disks. This generalizes the predictions of mirror symmetry to the non-Fano case.
\end{abstract}

\maketitle
\setcounter{tocdepth}{1}
\tableofcontents
\makeatletter

\section{Introduction} 

\subsection{Fano mirror symmetry} 

Under mirror symmetry, Fano varieties $X$ correspond to Landau-Ginzburg models $W : \check{X} \rightarrow \mathbb{C}$. If $X$ admits a $\mathbb{Q}$-Gorenstein degeneration to a toric variety $X_\Sigma$, then $W$ can be taken to be the Hori-Vafa potential $W_\Sigma$, which is a Laurent polynomial given by the sum of toric monomials labelled by the rays of the fan $\Sigma$. There are infinitely many such degenerations and the corresponding potentials $W_\Sigma$ are related by an operation called mutation \cite{ACGK}. The classical period $\pi_W$ of the Laurent polynomial $W$ equals the regularized quantum period $G_X$, which is defined as a generating function for descendant Gromov-Witten invariants of $X$ \cite{MSfano}. This period can also be obtained as the holomorphic part of the logarithmic solution to a Picard-Fuchs type differential equation \cite{LM}. 

The period defines closed and open mirror maps. Closed mirror maps relate the complex parameters $z$ to the symplectic parameters $Q$. The open mirror map relates the coordinate $x$ on the mirror curve, which is the moduli space of A-branes in the local Calabi-Yau variety $K_X$, to the symplectic area parameter $U$ of the dual B-branes. It is expected from mirror symmetry and proved in many fashions \cite{CCLT}\cite{GRZ}\cite{LLW}\cite{LLL}\cite{You} that the open mirror map, after insertion of the closed mirror maps, is a generating function for counts of holomorphic disks or, in the algebraic language, $2$-marked logarithmic Gromov-Witten invariants. A $q$-refined (``quantum'') version of this statement is worked out in \cite{GRZZ}. 

Another point of view \cite{CPS} is that the potential $W$ is equal to the primitive theta function $\vartheta_1$ in the central chamber of the scattering diagram defined by $X$. This is defined as a sum over primitive broken lines ending in the central chamber. In the Fano case, all such broken lines are actually straight lines, and there is one for each ray of $\Sigma$ (if $X$ admits a $\mathbb{Q}$-Gorenstein degeneration to $X_\Sigma$), so one recovers the Hori-Vafa potential $W_\Sigma$. After wall crossing to infinity (i.e. to an unbounded chamber), $\vartheta_1$ becomes a generating function for $2$-marked logarithmic Gromov-Witten invariants and, as shown in \cite{GRZ}, equal to the open mirror map. This also confirms the intrinsic mirror symmetry perspective \cite{GSintrinsic} that $\vartheta_1$ can be considered as a canonical variable on the mirror to $X$.

\begin{example}
For $X=\mathbb{P}^2$ we have $W=x+y+zx^{-1}y^{-1}$ with a complex parameter $z$, hence
\[ \pi_W(z) = \sum_{k>0} \const_{x,y}(W^k) = \sum_{k>0} \frac{(3k)!}{(k!)^3}z^k = 6z+90z^2+1680z^3+\ldots, \]
where $\const_{x,y}$ means taking the constant term with respect to the variables $x$ and $y$. Consider the formal logarithmic anti-derivative
\[ a_W(z) = \int \pi_W(z) \frac{dz}{z} = \sum_{k>0} \frac{1}{k}\const_{x,y}(W^k) = -\const(\log(1-W)). \]
The closed mirror map relates $z$ to a symplectic parameter $Q$ by
\[ Q = \exp(\log z + 3a_W(-z)) = ze^{3a_W(z)} = z - 6z^2 + 63z^3 - 866z^4 + \ldots, \]
where the factor $3$ is the intersection multiplicity of the line class $L$ with the excpetional class $-K_{\mathbb{P}^2}=3L$. The open mirror map relates the coordinate $x$ on the mirror curve to the symplectic area $U$ of the dual brane,
\[ U = \exp(\log x - a_W(-z)) = xe^{-a_W(z)} = x(1+2z-13z^2+158z^3+\ldots). \]
Inserting the inverse closed mirror map $z(Q)$ into the inverse open mirror map $x(U)$ we get
\[ x = UM_W(Q), \quad M_W(Q) = e^{a_W(-z(Q))} = 1-2Q+5Q^2-32Q^3+\ldots \]
The coefficients of $M_W(Q)$ are open Gromov-Witten invariants of $K_{\mathbb{P}^2}$ with winding one relative to a moment map fiber or an outer Aganagic-Vafa brane $N(K_{\mathbb{P}^2},1L)=-2$, $N(K_{\mathbb{P}^2},2L)=5$, $N(K_{\mathbb{P}^2},3L)=-32$. They are related to $2$-marked logarithmic Gromov-Witten invariants $R_{1,3d-1}(\mathbb{P}^2,dL)$ counting rational curves of degree $d$ intersecting an elliptic curve $E$ in a fixed point with multiplicity $1$ and in a non-fixed point with multiplicity $3d-1$, via $N(K_{\mathbb{P}^2},dL)=\frac{(-1)^d}{3d-1}R_{1,3d-1}(\mathbb{P}^2,dL)$.
\end{example}

\subsection{Non-Fano mirror symmetry} 

The aim of this paper is to extend the above framework to the case of non-Fano varieties $X$. Let $D$ be a smooth anticanonical divisor of $X$ and consider the smooth log Calabi-Yau pair $(X,D)$.

\begin{defintro}
A function $W\in\mathbb{C}[z_1,\ldots,z_r][x_1^{-1},\ldots,x_n^{-1}]\llbracket x_1,\ldots,x_n\rrbracket$ is called \emph{mirror dual} to $(X,D)$ if the classical period
\[ \pi_W(z) = \sum_{k>0}\const(W^k) \in \mathbb{C}[z_1,\ldots,z_r] \]
equals the regularized quantum period $G_X(z)$ (Definition \ref{defi:G}) of $X$.
\end{defintro}

\begin{thmintro}[Theorem \ref{thm:theta}]
\label{thm:main0}
The primitive theta function $\vartheta_1(X,D)$ is a mirror potential of $X$ in the sense that the classical period $\pi_{\vartheta_1(X,D)}(z)$ is equal to the regularized quantum period $G_X(z)$. This is true in any chamber of the scattering diagram.
\end{thmintro}

We will see that for non-Fano varieties, $\vartheta_1$ is not equal to the Hori-Vafa potential, but receives corrections coming from internal rays in the scattering diagram or, in the symplectic language, certain special Maslov index $0$ disks (instanton corrections). This has been observed in \cite{Aur09} and discussed for some examples in \cite{CPS}. See Figure \ref{fig:intro} for the example of $\mathbb{F}_3$.

\begin{figure}[h!]
\centering
\begin{tikzpicture}[scale=2]
\draw (1,0) -- (0,1) -- (-1,0) -- (-3,-1) -- cycle;
\coordinate[fill,cross,inner sep=2pt,rotate=45] (0) at (.5,.5);
\coordinate[fill,cross,inner sep=2pt,rotate=45] (1) at (-.5,.5);
\coordinate[fill,cross,inner sep=2pt,rotate=18.43] (2) at (-2,-.5);
\coordinate[fill,cross,inner sep=2pt,rotate=11.31] (3) at (-1,-.5);
\draw[dashed] (.5,2) -- (0) -- (2,.5);
\draw[dashed] (-.5,2) -- (1) -- (-4,.5);
\draw[dashed] (-4,-.5) -- (2) -- (-4,-.5-2/3);
\draw[dashed] (2,-.5) -- (3) -- (-4,-1.5);
\draw (-.5,1.5) -- (0) -- (1.5,-.5);
\draw (.5,1.5) -- (1) -- (-1.75,-.75);
\draw (.5,.75) -- (2) -- (-4,-1-1/2);
\draw (2,.25) -- (3) -- (-4,-1-1/4);
\draw (-1,0) -- (2,0);
\foreach \i in {3,5,...,49}
{
\draw (0,1) -- ({1/\i},2);
\draw (0,1) -- ({-1/\i},2);
}
\fill (0,1) -- (1/50,2) -- (-1/50,2) -- (0,1);
\foreach \i in {9,15,...,147}
{
\draw (-3,-1) -- (-4,{-1-1/3+1/\i});
\draw (-3,-1) -- (-4,{-1-1/3-1/\i});
}
\fill (-3,-1) -- (-4,-1-1/3+1/150) -- (-4,-1-1/3-1/150) -- (-3,-1);
\foreach \i in {1,2,...,20}
{
\draw (1,0) -- (2,{1/8+1/(5+5*\i)});
\draw (1,0) -- ({2-1/(5*\i)},-1/2);
}
\fill (1,0) -- (2,1/7) -- (2,-1/2) -- (1,0);
\fill[red] (-.95,-.42) circle (1pt);
\draw[red] (-.95,-.42) -- (-4,-.42);
\draw[red] (-4,-.5);
\draw[red] (-.95,-.42) -- (2,-.42);
\draw[red] (-.95,-.42) -- (-.95,.05) -- (-4,.05);
\draw[red] (-.95,-.42) -- (-4,-1.43667);
\draw[red] (-4,-1.45);
\draw[red] (-.95,-.42) -- (-1.42,-.42) -- (-4,-1.28);
\draw[red] (-4,-1.25);
\draw[red] (-.95,-.42) -- (-1.185,-.185) -- (-1.76,-.38) -- (-4,-.38);
\draw[red] (-4,-.3);
\end{tikzpicture}
\caption{The scattering diagram for the Hirzebruch surface $\mathbb{F}_3$ and broken lines (red) defining the corrected potential $W=\vartheta_1$.}
\label{fig:intro}
\end{figure}

The theta function $\vartheta_1(X,D)$ is defined in any chamber of the scattering diagram, and the different chambers are related by mutations (\S\ref{S:mutation}). We write $\vartheta_1(X,D)_\infty$ for the theta function at infinity, meaning in an unbounded chamber infinitely far away from the central chamber. $\vartheta_1(X,D)_\infty$ is the same in all unbounded chambers. It has infinitely many terms, but only depends on one variable $y$, corresponding to the unbounded direction of the scattering diagram. We extend the tropical correspondence of \cite{Gra2} and the results of \cite{GRZ} to the non-Fano case and show that the open mirror map defined by the potential $\vartheta_1$ is a generating function for $2$-marked logarithmic Gromov-Witten invariants.

\begin{thmintro}[Corollary \ref{cor:theta}, Theorem \ref{thm:mirmap}]
\label{thm:main}
Define $a_W(z) = -\const_{x,y}(\log(1+\vartheta_1(X,D)))$, the closed mirror maps $Q_i = ze^{d_ia_W(z)}$ and the open mirror map $M_W(Q) = e^{a_W(z(Q))}$. Then
\[ M_{t^{-1}\vartheta_1}(Q) = y^{-1}\vartheta_1(y)_\infty = 1+\sum_{\beta\in NE(X)} \frac{1}{\beta\cdot D-1} R_{0,(1,\beta\cdot D-1)}((X,D),\beta)Q^\beta, \]
where $y$ is related to $Q$ by the change of variables $Q_i=z_i(t/y)^{d_i}$, and we defined $d_i=\beta_i\cdot D$, where $\beta_i$ is the curve class corresponding to the parameter $z_i$. The numbers $R_{0,(1,\beta\cdot D-1)}((X,D),\beta)$ are $2$-marked logarithmic Gromov-Witten invariants of the log Calabi-Yau pair $(X,D)$, where $D$ is a smooth anticanonical divisor, see \S\ref{S:2marked}.
\end{thmintro}

\begin{remark}
Similar results have been previously obtained in the Fano and semi-Fano case, where $2$-marked logarithmic invariants are equal to open Gromov-Witten invariants and the periods are solutions to Picard-Fuchs type differential equations \cite{Lau14}\cite{GRZ}\cite{You}. These methods don't work in the non-semi-Fano case. Non-Fano examples have been considered in \cite{CPS}, but not with respect to an anticanonical divisor. If there exists a $\mathbb{Q}$-Gorenstein degeneration to a Fano toric variety, the theta functions are related via mutations (see \S\ref{S:mutation}), and we can use the techniques of the Fano case.
\end{remark}

\subsection*{Acknowledgements} 
PB and ML are supported in part by the Department of Energy grant DE-SC0020220.
During the Concluding Conference of Simons Collaboration on Homological Mirror Symmetry at SCGP, Helge Ruddat suggested that ML should reach out to TG, which lead to this collaboration.
We thank Fenglong You for helpful discussions.
We thank the anonymous referees for their careful feedback.

\part{Geometric Setup} 

\section{Log Calabi-Yau pairs} 
\label{S:setup}

\begin{definition}
A \emph{smooth log Calabi--Yau pair} $(X,D)$ consists of a smooth projective variety $X$ and a smooth anticanonical divisor $D$. 
\end{definition}

\begin{remark}
There might be less restrictive definitions, e.g. in \cite{GS22canonical} the authors consider the case where $D+K_X$ is numerically equivalent to an effective $\mathbb{Q}$-divisor supported on $D$. The ``maximal boundary'' case with nodal singular divisor has been considered e.g. in \cite{GHK} and \cite{BBG}. If $X$ is a toric variety, then $D$ has the same class as the toric boundary $\partial X=\sum_\rho D_\rho$.
\end{remark}

\begin{example}
\label{eg:Hirz}
The Hirzebruch surfaces $\F_m=\P(\O_{\P^1}(m)\oplus\O_{\P^1})$ are an infinite sequence of toric surfaces which can be thought of as twisted $\P^1$ bundles over $\P^1$, with the parameter $m$ controlling the twisting of the fibers. $\mathbb{F}_m$ is a smooth toric variety. If $F$, $E$ and $S$ are the classes of a fiber, the negative section, and a positive section, respectively, then we have
\[ F^2=0, \ \ E^2=-m, \ \ S^2=m, \ \ F\cdot E = 1, \  \ F\cdot S=1, \ \ E\cdot S=0 \]
in the Chow ring $A_*(\F_m)$. The monoid of effective curce classes $NE(\mathbb{F}_m)$ is generated by $F$ and $E$. The anticanonical divisor is given by
\[ -K_{\F_m}=2E+(2+m)F, \]
so that the degree of $\F_m$ is $(-K_{\F_m})^2=4m+4(2-m)=8$, independent of $m$.

Let $D$ be a smooth anticanonical divisor. Then $(\mathbb{F}_m,D)$ is a smooth log Calabi-Yau pair. $\mathbb{F}_m$ is Fano for $m=0,1$ and non-Fano for $m\geq 2$. It is semi-Fano for $m=2$.
\end{example}

\subsection{$\mathbb{Q}$-Gorenstein degenerations} 
\label{S:toricmodel}

\begin{definition}
\label{def:toricmodel}
Let $\mathcal{X}\rightarrow B$ be a flat family of projective varieties over the spectrum of a discrete valuation ring, such that the relative canonical divisor $K_{\mathcal{X}/B}$ is $\mathbb{Q}$-Cartier. Let $X=X_\eta$ be the general fiber and $X_0$ the special fiber. We say $X$ is a \emph{$\mathbb{Q}$-Gorenstein deformation} of $X_0$ and $X_0$ is a \emph{$\mathbb{Q}$-Gorenstein degeneration} of $X$.
\end{definition}

\begin{example}
Let $X$ be the blow up of $\mathbb{P}^2$ in $k$ points in general position. A $\mathbb{Q}$-Gorenstein degeneration $X_0$ of $X$ is given by the (successive) blow up of $\mathbb{P}^2$ in $k$ points in non-general position. If in each step we blow up a torus-fixed point, then $X_0$ is a toric variety with Gorenstein singularities. Figure \ref{fig:cubicfan} shows the fan of a toric blow up of $\mathbb{P}^2$ in $6$ points, which is a $\mathbb{Q}$-Gorenstein degeneration of a smooth cubic surface. Further examples are considered in \S\ref{S:Bl}.
\end{example}

\begin{figure}[h!]
\centering
\begin{tikzpicture}[scale=1]
\draw[->] (0,0) -- (1,0);
\draw[->] (0,0) -- (0,1);
\draw[->] (0,0) -- (-1,2);
\draw[->] (0,0) -- (-1,1);
\draw[->] (0,0) -- (-1,0);
\draw[->] (0,0) -- (-1,-1);
\draw[->] (0,0) -- (0,-1);
\draw[->] (0,0) -- (1,-1);
\draw[->] (0,0) -- (2,-1);
\end{tikzpicture}
\caption{The fan of a $\mathbb{Q}$-Gorenstein degeneration of a smooth cubic.}
\label{fig:cubicfan}
\end{figure}

\begin{example}
\label{eg:HirzModel}
The Hirzebruch surface $\F_m$ is non-Fano for $m\geq 2$. It admits a $\mathbb{Q}$-Gorenstein degeneration to the toric Fano variety $\mathbb{F}_0=\mathbb{P}^1\times\mathbb{P}^1$ if $m$ is even and to $\mathbb{F}_1=\text{Bl }\mathbb{P}^2$ if $m$ is odd. This is given by the family
\[ \mathcal{X} = \{x_0^my_1-x_1^my_0+tx_0^{m-k}x_1^ky_2=0\} \subset \mathbb{P}^1_{x_0,x_1}\times\mathbb{P}^2_{y_0,y_1,y_2}\times\mathbb{A}^1_t \rightarrow \mathbb{A}^1_t. \]
The fiber over $t=0$ is $\mathbb{F}_m$ and the general fiber is $\mathbb{F}_{m-2k}$. The family leaves the anticanonical polarization unchanged, hence is $\mathbb{Q}$-Gorenstein, and acts on curve classes by
\begin{align*}
\mathbb{F}_m & \rightarrow  \mathbb{F}_{m-2k}, \\
F & \mapsto   F, \\
E & \mapsto  E-kF, \\
S & \mapsto  E+kF.
\end{align*}
Here $F$, $E$ and $S$ are the classes of a fiber, the negative section, and a positive section, respectively, such that $F^2=0$, $E^2=-m$, $S^2=m$ and $F\cdot E = F\cdot S=1$, $E\cdot S=0$. The monoid of effective curve classes $NE(\mathbb{F}_m)$ is generated by $F$ and $E$.
\end{example}

\begin{remark}
The toric varieties to which a Fano variety $X$ admits a $\mathbb{Q}$-Gorenstein degeneration are related via mutations of their spanning polytopes \cite{Ilten}. In \S\ref{S:mutation} we will study a similar relation for non-Fano varieties $X$ and mutations of their mirror potentials $W$.
\end{remark}

\subsection{The fan and (quasi-)polytope picture} 
\label{S:fanpolytope}

A normal toric variety $X$ corresponds (up to isomorphism) to a polyhedral fan $\Sigma$. A polarization on $X$ is a choice of ample line bundle. This induces an embedding into projective space, so $X$ is projective in this case. A toric variety $X$ together with a polarization corresponds (up to isomorphism) to a lattice polytope $\Delta$, and we write $X=\YD$ in this case.

\begin{remark}
When thinking of $\YD$ as a symplectic manifold, the Newton polytope is the moment polytope, namely it is the image of the moment map $\mu:\YD\longrightarrow\mathfrak{g}^*$ where $\mathfrak{g}^*\cong \R^n$ is the dual Lie algebra of the dense algebraic torus $G=T:=(\C^\times)^n$ in $\YD$. 
When $\D$ is reflexive it is also known as a Delzant polytope.
\end{remark}

We explain how this framework extends to non-ample line bundles. A line bundle (resp. Cartier divisor $D$) on $X$ corresponds to an integral (with respect to the lattice) continuous function $\varphi$ that is linear on the cones of $\Sigma$. It is determined by the values on the vertex $v$ of $\Sigma$ and the primitive ray generators $m_\rho$. If $\varphi(v)=0$ and $\varphi(m_\rho)=d_\rho\in\mathbb{Z}$, then $\varphi$ corresponds to the Cartier divisor $D=\sum_\rho d_\rho D_\rho$, where $D_\rho$ is the toric divisor corresponding to $\rho$. The divisor is semi-ample resp. ample if and only if $\varphi$ is convex resp. strictly convex (strictly convex means that $\varphi$ is not linear in a neighborhood of any ray, so that its domains of linearity are exactly the maximal cones of $\Sigma$).

\begin{definition}
A \emph{(lattice) quasi-polytope} $\Delta$ in $N_{\mathbb{R}}$ is an abstract polytope (this is the partially ordered set describing the combinatorial structure underlying a polytope, without an embedding) together with a map from the set of $0$-cells to $N_{\mathbb{R}}$ (resp. to the lattice $N$), specifying the positions of the vertices. 
\end{definition}

\begin{example}
A $2$-dimensional quasi-polytope is simply a cyclic ordering $(v_1,\ldots,v_n)$ of vertices $v_i \in N_{\mathbb{R}}$ (resp. $N$) that are connected by edges. The edges bound a $2$-cell, but this $2$-cell does not need to map to a polytope in $N_{\mathbb{R}}$, since the edges may intersect.
\end{example}

\begin{example}
A fan $\Sigma$ induces an abstract polytope $\Delta$. The vertex of $\Sigma$ gives the maximal cell of $\Delta$ and the maximal cells of $\Sigma$ give the $0$-cells of $\Delta$.
\end{example}

\begin{definition}
Let $\Sigma$ be a fan and let $\varphi$ be a continuous function that is linear on the cones of $\Sigma$. The \emph{dual polytope} $\Delta$ of $(\Sigma,\varphi)$ is the quasi-polytope whose abstract polytope is dual to the abstract polytope induced by $\Sigma$, and whose $0$-cells are mapped to the lattice elements given by the coefficients of $\varphi$ on the corresponding maximal cone of $\Sigma$.
\end{definition}

\begin{example}
Consider the Hirzebruch surface $\mathbb{F}_m$. Its fan $\Sigma$ has ray generators $m_1=(1,0)$, $m_2=(0,1)$, $m_3=(-1,0)$ and $m_4=(-m,-1)$. The corresponding toric divisors are $D_1=S$, $D_2=F$, $D_3=E$ and $D_4=F$, see Example \ref{eg:HirzModel} for the notation. The anticanonical divisor $D$ corresponds to the piecewise linear function $\varphi$ given by $x+y$ on the maximal cone $\braket{m_1,m_2}$, by $-x+y$ on $\braket{m_2,m_3}$, by $-x+(m-1)y$ on $\braket{m_3,m_4}$ and by $x-(m+1)y$ on $\braket{m_4,m_1}$. The dual polytope $\Delta$ has vertices $v_1=(1,1)$, $v_2=(-1,1)$, $v_3=(-1,m-1)$ and $v_4=(1,-(m+1))$, in this order. It is a polytope only if $m\leq 2$, i.e., if $\mathbb{F}_m$ is semi-Fano. See Figure \ref{fig:F3fanpoly} for the case $m=3$
\end{example}

\begin{figure}[h!]
\centering
\begin{tikzpicture}[scale=.5]
\draw[->] (0,0) -- (3,0);
\draw[->] (0,0) -- (0,3);
\draw[->] (0,0) -- (-9,0);
\draw[->] (0,0) -- (-9,-3);
\draw (1.5,1) node{$x+y$};
\draw (-2,1) node{$-x+y$};
\draw (-6,-1) node{$-x+2y$};
\draw (0,-1) node{$x-4y$};
\draw[<->] (5,0) -- (7,0);
\draw (0,-5);
\end{tikzpicture}
\begin{tikzpicture}[scale=.8]
\draw (1,1) -- (-1,1) -- (-1,2) -- (1,-4) -- cycle;
\fill (1,1) circle (2pt) node[above]{$(1,1)$};
\fill (-1,1) circle (2pt) node[left]{$(-1,1)$};
\fill (-1,2) circle (2pt) node[left]{$(-1,2)$};
\fill (1,-4) circle (2pt) node[below]{$(1,-4)$};
\end{tikzpicture}
\caption{The fan $\Sigma$ with piecewise linear function $\varphi$ (left) and the dual quasi-polytope $\Delta$ (right) for $\mathbb{F}_3$ with anticanonical divisor.}
\label{fig:F3fanpoly}
\end{figure}

\begin{remark}
If $\Delta$ is the dual quasi-polytope of $(\Sigma,\varphi)$, then the (outer) normal fan of $\Delta^\star$ is given by $\Sigma$. For this one has to choose an orientation of the edges that is compatible with one of the maximal cells of the image.
\end{remark}

\begin{definition}
The \emph{spanning polytope} $\Delta^\star$ of $\Sigma$ is the (possibly non-convex) polytope whose vertices are the ray generators of $\Sigma$. See Figure \label{fig:HirzSurfPlots} for the first Hirzebruch surfaces.
\end{definition}

\begin{remark}
The quasi-polytope $\Delta$ can be seen as the spanning polytope of a multi-fan $\Sigma^\star$, and it is also called a multi-polytope in this case, see \cite{HattoriMasuda}. The duality between $(\Sigma,\varphi)$ and $\Delta$ is equivalent to a duality between $\Delta^\star$ and $(\Sigma^\star,\varphi^\star)$, for some piecewise linear function $\varphi^\star$. The duality can also be interpreted as a duality between the (non-convex) polytope $\Delta^\star$ and the (quasi-)polytope $\Delta$. If $\Delta^\star$ is a reflexive polytope (i.e., if $X$ is a Gorenstein toric Fano variety), then $\Delta$ is the polar dual of $\Delta^\star$.
\end{remark}

When $X$ is non-toric but admits a $\mathbb{Q}$-Gorenstein deformation to a toric variety, we will still talk about the fan and (quasi-)polytope of $X$.
 
\section{Gromov-Witten invariants} 
\label{S:GW}

Gromov-Witten invariants are one way to rigorously define the number of curves on a projective variety $X$ that meet certain prescribed properties. They are defined via intersection theory on the moduli space of stable maps, which is a compactification of the moduli space of smooth curves with marked points. Some stable maps, for instance multiple covers, have non-trivial automorphism group $\text{Aut}$, and they are counted with weight $\tfrac{1}{|\text{Aut}|}$. As a consequence, Gromov-Witten invariants are rational numbers in general. We will work with several different Gromov-Witten invariants that we introduce in the following, partly to fix notation.

\subsection{Relative Gromov-Witten invariants} 

Relative Gromov-Witten invariants of a smooth log Calabi-Yau pair $(X,D)$ count curves on $X$ of a given genus $g$ (for us $g=0$) and curve class $\beta$ that intersect the divisor $D$ in a single unspecified point, necessarily with mutliplicity $\beta\cdot D$. They can formally be defined as a special case of logarithmic Gromov-Witten invariants \cite{LogGW}. The class $\beta$ and the tangency condition define a class of stable log maps, and there exists a moduli space $\mathcal{M}_{0,(\beta\cdot D)}((X,D),\beta)$ of basic stable log maps of this class. It is a proper algbraic stack and admits a virtual fundamental class $[\mathcal{M}_{0,(\beta\cdot D)}((X,D),\beta)]^{virt}$. The relative Gromov-Witten invariants are defined by integration, that is, proper pushforward to a point,
\[ R_{0,(\beta\cdot D)}((X,D),\beta) = \int_{[\mathcal{M}_{0,(\beta\cdot D)}((X,D),\beta)]^{virt}} 1. \]
For example, $R_{0,(3)}((\mathbb{P}^2,E),L)=9$, the nine flex lines of the elliptic curve $E$, and $R_{0,(6)}((\mathbb{P}^2,E),2L)=\tfrac{135}{4}=27+9\cdot\tfrac{3}{4}$. A double cover of a line contributes $\tfrac{3}{4}$ by \cite{GPS}, Proposition 6.1.

\subsection{$2$-marked Gromov-Witten invariants} 
\label{S:2marked}

A natural variation of $R_{0,(\beta\cdot D)}((X,D),\beta)$ is to weaken the tangency condition and add a fixed point condition. We consider $2$-marked logarithmic Gromov-Witten invariants which intersect $D$ in a specified point with multiplicity $p$ and and unspecified point with multiplicity $(\beta\cdot D)-p$. Again, the curve class and tangency conditions define a class of stable log maps. The fixed point condition is accounted for by integrating over the pullback of the point class along the evaluation map of the first point $\text{ev} : \mathcal{M}_{0,(p,\beta\cdot D-p)}((X,D),\beta) \rightarrow D$,
\[ R_{0,(p,\beta\cdot D-p)}((X,D),\beta) = \int_{[\mathcal{M}_{0,(p,\beta\cdot D-p)}((X,D),\beta)]^{virt}} \textup{ev}^\star [\textup{pt}]. \]
For example, $R_{0,(2,1)}((\mathbb{P}^2,E),L)=1$, the unique tangent line to the elliptic curve $E$ at the specified point, and $R_{0,(1,2)}((\mathbb{P}^2,E),L)=4$.

\begin{proposition}[\cite{CC} and \cite{GRZ}, Theorem 5.4]
For all $p,q \in \mathbb{N}$ with $p+q=\beta\cdot D$ we have
\[ p^2R_{0,(p,q)}((X,D),\beta) = q^2R_{0,(q,p)}((X,D),\beta). \]
\end{proposition}

\subsection{Descendant Gromov-Witten invariants} 
\label{S:descGW}

Descendant Gromov-Witten invariants are another possibility to specify tangency conditions. Consider the moduli space of $1$-marked stable log maps $\mathcal{M}_{0,1}((X,D),\beta)$ (without a condition on the intersection multiplicity, we work with stable log maps only to be able to treat singular varieties). There is a natural line bundle $\mathcal{L}$ on $\mathcal{M}_{0,1}((X,D),\beta)$ whose fiber over an element $\mathcal{M}_{0,1}((X,D),\beta)$ is the cotangent line of the corresponding curve at the marked point. The first Chern class $\psi=c_1(\mathbb{L})$ is called the psi class. The descendant Gromov-Witten invariants are defined by fixing the marked point (using the evaluation map as above) and inserting a maximal power of the psi class, such that the virtual dimension becomes $0$, i.e., we get a finite number,
\[ D_{0,1}(X,\beta) = \int_{[\mathcal{M}_{0,1}((X,D),\beta)]^{virt}} \psi^{\beta\cdot D-2}\textup{ev}^\star [\textup{pt}]. \]
For example, $D_{0,1}((\mathbb{P}^2,E),L)=1$, the unique (co-)tangent line to the elliptic curve $E$ at the specified point. In general, $D_{0,1}((\mathbb{P}^2,E),dL)=\tfrac{1}{(d!)^3}$, as follows from the string equation for psi classes.

\begin{definition}
\label{defi:G}
The regularized quantum period of $X$ is
\[ G_X(z) = \sum_\beta (\beta\cdot(-K_X))! D_{0,1}(X,\beta)z^\beta. \]
\end{definition}

\begin{proposition}[\cite{Givental}]
\label{prop:Giv}
For a toric variety $X$ we have
\[ D_{0,1}(X,\beta) = \frac{1}{\sum_\rho (\beta\cdot D_\rho)!}, \]
such that
\[ G_X(z) = \sum_\beta \frac{(\beta\cdot(-K_X))!}{\sum_\rho (\beta\cdot D_\rho)!}z^\beta. \]
Here $D_\rho$ are the toric divisors of $X$, labelled by the rays $\rho$ in the fan $\Sigma$, and $-K_X=\sum_\rho D_\rho$ is the anticanonical divisor.
\end{proposition}

\subsection{Local Gromov-Witten invariants} 
\label{S:GWlocal}

Given a projective surface $X$, the total space of the canonical bundle $K_X$ is a Calabi-Yau threefold (it has trivial canonical bundle) by the adjunction formula. One can define Gromov-Witten invariants without any conditions, since the virtual dimension of the moduli space of stable (log) maps is zero. These are called local Gromov-Witten invariants of $K_X$,
\[ N_0(K_X,\beta) = \int_{[\mathcal{M}_{0,0}(K_X,\beta)]^{virt}} 1. \]
For example, $N_0(K_{\mathbb{P}^2},L)=3$ and $N_0(K_{\mathbb{P}^2},2L)=-\tfrac{45}{8}=-6+3\cdot\tfrac{1}{8}$.

\subsubsection{Log-local correspondence} %
\label{S:logopen}
By the log-local correspondence \cite{GGR}, for a Fano variety $X$ we have 
\[ N_0(K_X,\beta)=\tfrac{(-1)^{\beta\cdot D-1}}{\beta\cdot D}R_{0,(\beta\cdot D)}((X,D),\beta). \]
This is no longer true for non-Fano varieties, because curves can move away from the zero section of $K_X$, which is non-negative in this case.

\subsection{Open Gromov-Witten invariants} 
\label{S:GWopen}

Consider the Calabi-Yau threefold $K_X$ as a symplectic manifold. Let $\omega$ be the K\"ahler form of a Ricci flat metric and let $\Omega$ be a nowhere vanishing holomorphic volume form. Let $L$ be a special Lagrangian submanifold $L$. This is a real $3$-dimensional manifold such that $\omega|_L=0$ and $\text{Im }\Omega|L=0$. The open Gromov-Witten invariants $O_0((K_X,L),\beta)$ count holomorphic disks with boundary mapping to $L$. They were rigorously defined in \cite{FOOO}. There are different natural choices for $L$. 

\subsubsection{Moment map fibers} %
One can consider $K_X\rightarrow B$ as a special Lagrangian torus fibration by the moment map. The base $B$ has two different structures of a real $3$-dimensional affine manifold with singularities, given by the intersection complex and its dual, respectively, see \S\ref{S:toricdeg}. If $L$ is a \emph{moment map fiber} of this fibration, it is differomorphic to $S^1\times S^1\times \mathbb{C}$. In this case, the invariants $O_0((K_X,L),\beta)$ have been studied e.g. in \cite{CCLT}\cite{LLW}\cite{You}.

\subsubsection{Aganagic-Vafa branes} %
If $X$ and hecne $K_X$ is toric, one can consider $K_X$ as a symplectic GIT quotient. An \emph{Aganagic-Vafa brane} is a special Lagrangian submanifold $L$ defined by extending the defining equations of $K_X$, see \cite{FL}, \S2. The manifold $L$ is differomorphic to $S^1\times\mathbb{C}$, depends on a \emph{framing} parameter $f$, and is called an \emph{inner} or \emph{outer} brane depending on its intersection with a torus orbit closure. In this case, the invariants $O_0((K_X,L),\beta)$ have been studied e.g. in \cite{AKV02}. 

\subsubsection{Open-closed correspondence} %
It was shown in \cite{Chan} and \cite{FL} that if $L$ is a moment map fiber or an outer Aganagic-Vafa brane, then $O_0((K_X,L),\beta)$ is equal to $N_0(K_{\hat{X}},\pi^\star\beta-C)$, where $\pi : \hat{X} \rightarrow X$ is the blow up of $X$ at a point. One way to see this is as follows. By capping the disk $L$ to a holomorphic curve, $O_0((K_X,L),\beta)$ is equal to the Gromov-Witten invariant of a (partial) compactification of $K_X$ with a specified point. Together with the log-local correspondence and a blow up formula (\cite{GRZ}, Theorem 6.5) this shows that, in the Fano case,
\[ R_{0,(1,\beta\cdot D-1)}((X,D),\beta)=(-1)^{\beta\cdot D-1}(\beta\cdot D-1)O_0((K_X,L),\beta). \]
Hence, $2$-marked Gromov-Witten invariants can be seen as an algebraic analogue of counts of Maslov index $2$ disks.

\section{Mirror symmetry} 

\subsection{Landau-Ginzburg models} 

Under mirror symmetry, a log Calabi-Yau pair $(X,D)$ corresponds to a Landau-Ginzburg model, which is a (possibly non-compact) toric variety $\check{X}$ together with a potential function $W : \check{X} \rightarrow \mathbb{C}$ whose critial locus is compact. If $X$ is Calabi-Yau, the divisor $D$ and the potential $W$ are trivial and the mirror $\check{X}$ is a compact Calabi-Yau variety. If $X$ is projective, as in our case, one can always take $\check{X}$ to be an algebraic torus $(\mathbb{C}^\star)^n$, where $n$ is the dimension of $X$. Then the potential is a Laurent polynomial $W\in\mathbb{C}[z][x_1^{\pm1},\ldots,x_n^{\pm1}]$, with complex parameters $z=(z_1,\ldots,z_r)$ corresponding to a basis $(\beta_1,\ldots,\beta_r)$ of effective curve classes of $X$. We define what we mean by mirror symmetry in this case.

\begin{definition}
\label{defi:HV}
Let $X_\Sigma$ be a toric variety with fan $\Sigma$. The \emph{Hori-Vafa potential} of $X_\Sigma$ is
\[ W_\Sigma = \sum_\rho z^\rho x^{m_\rho}. \]
The sum is over the rays $\rho$ of $\Sigma$, $m_\rho=(m_1,\ldots,m_n)$ is the primitive ray generator, $x^{m_\rho}=x_1^{m_1}\cdots x_n^{m_n}$ and $z^\rho=z_1^{l_1}\cdots z_r^{l_r}$ such that if $(\rho_1,\ldots,\rho_k)$ is a collection of rays with $\sum_{i=1}^k m_{\rho_i} = 0$, then $\sum_{i=1}^k z^\rho x^{m_\rho} = z^\beta := z_1^{d_1}\cdots z_r^{d_r}$, where $\beta=d_1\beta_1+\ldots+d_r\beta_r$ is the effective curve class whose intersections with toric divisors are given by $(\rho_1,\ldots,\rho_k)$.
\end{definition}

The gauged linear sigma model perspective is studied in \cite{BH} and the B-model periods in \cite{BL1}\cite{BL2}.

\subsection{Mirror potentials} 
\label{S:MS}

\begin{definition}
\label{defi:pi}
The \emph{classical period} of a Laurent polynomial $W\in\mathbb{C}[z][x_1^{\pm1},\ldots,x_n^{\pm1}]$ is
\[ \pi_W(z) = \sum_{k>0} \const(W^k), \]
where $\const$ means taking the constant term with respect to the variables $x_1,\ldots,x_n$.
\end{definition}

\begin{definition}
\label{defi:MS}
We say $(X,D)$ is mirror dual to $W\in\mathbb{C}[z][x_1^{\pm1},\ldots,x_n^{\pm1}]$ if 
\[ \pi_W(z)=G_X(z). \]
Here $G_X(z)$ is the regularized quantum period from Definition \ref{defi:G}.
\end{definition}

\begin{proposition}
If $X$ is a Fano variety admitting a $\mathbb{Q}$-Gorenstein degeneration to a toric Fano variety $X_\Sigma$, then $X$ is mirror dual to $W_\Sigma$.
\end{proposition}

\begin{remark}
We will see in \S\ref{S:mutation} that we can drop the condition on $X$ being Fano, but still need $X_\Sigma$ Fano. Theorem \ref{thm:main0} states that in general (also in the non-Fano case), the potential $W$ is equal to the primitive theta function $\vartheta_1$ defined by broken lines in the scattering diagram of $X$ (see \S\ref{sec:scatteringbroken}).
\end{remark}

\subsection{Mirror maps} 
\label{S:mirmap}

The \emph{closed mirror maps} relate the complex parameters $z_i$ to symplectic parameters $Q_i$ that will be variables in a generating function for Gromov-Witten invariants. For an interpretation of the \emph{open mirror map}, we consider the Calabi-Yau threefold $K_X$ given by the total space of the canonical bundle of $X$. It is mirror dual to the hypersurface in $\mathbb{C}^2\times(\mathbb{C}^\star)^2$ defined by $uv=W(x,y)$. The Aganagic-Vafa A-branes on $K_X$ (see \S\ref{S:GWopen}) are mirror dual to B-branes supported on $\{v=0\}$ and their moduli space is the mirror curve $C=\{W(x,y)=0\}\subset(\mathbb{C}^\star)^2$. By solving for $y$ in terms of $x$, we can use $x$ as a local parameter for the moduli space $C$. The open mirror map relates $x$ with a complex modulus $U$ encoding the symplectic area of the disk together with the $U(1)$ holonomy (\emph{winding}) of its boundary circle.

\begin{definition}
\label{defi:mirmap}
For a potential function $W$, consider the formal logarithmic antiderivative
\[ a_W(z) = \int \pi_W(z) \frac{dz}{z} = \sum_{k>0} \frac{1}{k}\const_{x,y}(W^k) = -\const(\log(1-W)). \]
Define the \emph{closed mirror maps}
\[ Q_i = z_ie^{d_ia_W(-z)}, \]
and the \emph{open mirror map}
\[ U = xe^{-a_W(-z)}. \]
Here $d_i=\beta_i\cdot D$ is the intersection multiplicity of the curve class $\beta_i$ corresponding to $z_i$ with the anticanonical divisor $D$. Moreover, define
\[ M_W(Q) = e^{a_W(z(Q))}. \]
This is the inverse of the open mirror map after inserting the inverses of the closed mirror maps. By abuse of notation we will also call it the open mirror map.
\end{definition}

\begin{theorem}
\label{thm:corr}
Let $(X,D)$ be a smooth log Calabi-Yau pair (not necessarily Fano) with mirror dual potential $W$. The open mirror map is a generating function for $2$-marked logarithmic Gromov-Witten invariants,
\[ M_W(Q) = 1+\sum_\beta \frac{(-1)^{\beta\cdot D-1}}{\beta\cdot D-1}R_{0,(1,\beta\cdot D)}((X,D),\beta)Q^\beta. \]
\end{theorem}

By \S\ref{S:logopen}, in the Fano case the coefficients are exactly the open Gromov-Witten invariants of $K_X$. We will prove Theorem \ref{thm:corr} in \S\ref{S:ms} by using generalizations of the tropical correspondence theorems from \cite{Gra22}\cite{Gra2} and the combinatorial identities from \cite{GRZZ}.

\part{Methods} 

\section{Toric degenerations} 
\label{S:toricdeg}

\begin{definition}
A \emph{toric degeneration} of $(X,D)$ is a flat family $(\mathcal{X},\mathcal{D}) \rightarrow T$ over a base $T$ (in our case $T=\text{Spec }\mathbb{C}[\text{NE}(X)]\llbracket t\rrbracket$) such that the general fiber is isomorphic to $(X,D)$ and the special fiber $X_0$ is a union of toric varieties glued along toric divisors, $D_0$ is a union of toric divisors not involved in the gluing, and such that the family is strictly semistable (i.e. locally of the form $x_1\cdots x_k=t^l$) away from a codimension $2$ subset $Z\subset X_0$. The latter condition is equivalent to $\mathcal{X}\rightarrow T$ being log smooth with respect to the divisorial log structures defined by $X_0\subset \mathcal{X}$ and $\{0\}\subset T$.
\end{definition}

\begin{example}[$\XcDegen$ for $(\P^2,3L)$]
	The following algebraic family serves as a toric degeneration for the smooth log Calabi--Yau pair $(\P^2,E)$, with elliptic curve $E$.
	\begin{align*}
		\XcDegen=\{x_1x_2x_3=t(y_1+f_3)\}
	\end{align*}
	for $(x_1,x_2,x_3,y_1)\in\P(1,1,1,3)$, $t\in\A^1$, and $f_3$ is a cubic polynomial that deforms $\XcDegen$. 
    The zero locus of $y_1$ defines a family of divisors $\check{\mathfrak{D}}\subset\XcDegen$. 
    The case $f_3=0$ corresponds to the toric boundary divisor $\partial\P^2$, whereas a nontrivial cubic gives the anticannonical family.
    The central fiber is given by a union of three weighted projective spaces
    \begin{align*}
        \check{\P}_0=\P^2(1,1,3)\coprod\P^2(1,1,3)\coprod\P^2(1,1,3)
    \end{align*}
    which are glued along toric divisors as prescribed by the combinatorics of $\Delta$.
\end{example}

In the Fano case the toric degeneration can be constructed globally as a projective variety (in the so called cone or polytope picture), while for the non-Fano case we have to work more locally (in the fan picture).

\subsection{The toric Fano case} 
\label{S:toricdeg1}

For a toric Fano variety $\mathbb{P}_\Delta$, one can construct a toric degeneration of $\mathbb{P}_\Delta$ from its polytope $\Delta$. This is a Fano polytope, i.e., it has a unique interior lattice point. The central subdivision $\mathcal{P}$ of $\Delta$ is the collection of polytopes, one for each edge of $\Delta$, given by the convex hull of the interior lattice point and the edge. Consider the convex piecewise linear function $\check{\varphi} : \Delta \rightarrow \mathbb{R}$ whose domains of linearity are the components of the central subdivision, and that takes values $0$ at the interior lattice point and $1$ along the boundary of $\Delta$. The upper convex hull of $\check{\varphi}$,
\[ \Delta_{\check{\varphi}} = \{(m,h) \in \Delta \times \mathbb{R} \ | \ h \geq \check{\varphi}(m)\}, \]
is an unbounded polytope. It defines a (non-projective) toric variety $\mathcal{X}_\varphi := X_{\Delta_{\check{\varphi}}}$. Projection to the last factor (which is the unique unbounded direction) defines a map $\mathcal{X}_{\check{\varphi}} \rightarrow \mathbb{A}^1$. This is a flat family of projective toric varieties. The general fiber is isomorphic to $\mathbb{P}_\Delta$ and the central fiber is a union of toric varieties, corresponding to the components of the central subdivision of $\Delta$. One can define a family of divisors $\mathcal{D} \rightarrow \mathbb{A}^1$ by the vanishing of the coordinate corresponding to the interior lattice point. The general fiber gives the toric boundary of $\mathbb{P}_\Delta$. Hence, $\mathcal{X}_{\check{\varphi}} \rightarrow \mathbb{A}^1$ is a toric degeneration of the log Calabi-Yau pair $(\mathbb{P}_\Delta,\partial\mathbb{P}_\Delta)$. The polytope $\check{B}=\Delta$ together with its central subdivison $\check{\mathcal{P}}$ and the piecewise linear function $\check{\varphi}$ is a triple $\IC$ called the \emph{intersection complex} of $(\YD,X)$

\subsection{The smooth Fano case} 

The toric degeneration $\mathcal{X}_{\check{\varphi}}$ constructed above is a subvariety of $\mathbb{P}^{N-1} \times \mathbb{A}^1$, where $N$ is the number of lattice points of $\Delta$. Its defining equations correspond to relations among the lattice points of $\Delta$, and in general it is not a complete intersection. Now consider a sufficiently general deformation of the defining equations without changing the central fiber. This yields another toric degeneration $\mathcal{X} \rightarrow \mathbb{A}^1$. The general fiber is a pair $(X,D)$, where $X$ is a $\mathbb{Q}$-Gorenstein deformation of $\mathbb{P}_\Delta$ and $D$ is a deformation of $\partial\mathbb{P}_\Delta$, see \S\ref{S:toricmodel}.

The deformation leads to a modification of the intersection complex. Before, $\Delta$ was a polytope, and in particular an affine manifold (a topological manifold with affine linear transition maps). The deformation introduces affine singularities, one on each interior edge of the central subdivision $\mathcal{P}$, such that now $\check{B}=\Delta$ is an affine manifold with singularities.
We have to introduce these singularities to ensure that locally at each vertex of $\Delta$ the family $\mathcal{X}\rightarrow\mathbb{A}^1$ can still be described by the toric construction $X_{\Delta_{\check{\varphi}}}$ above.

\subsection{The fan picture} 
\label{S:fan}

Dual to the polytope picture, toric varieties can be described by fans, and a union of toric varieties can be described by gluing together such fans. If $X$ is toric, the total space $\mathcal{X}$ of the toric degeneration is then simply given by the toric variety obtained by taking the fan over the dual intersection complex. Again, smoothing the general fiber corresponds to changing the affine structure, such that the dual intersection complex becomes an affine manifold with singularities. In this picture, the polarization by the divisor $D$ is given by a multi-valued piecewise linear function $\varphi : B \rightarrow \mathbb{Z}^\rho$, where $\rho$ is the Picard rank of $X$.

The \emph{dual intersection complex} $\DIC$ can be constructed as the discrete Legendre transform of $\IC$.

\begin{construction}[Dual intersection complex $(B,\mathcal{P},\varphi)$]
\label{dicConstruction}
Let $\Sigma$ be the fan of $X$, which is the normal fan of $\Delta$.
Let $\Delta^\star$ be the spanning fan of $\Sigma$. This is the convex hull of the ray generators. If $X$ is non-Fano, then $\Delta^\star$ is non-convex.
Let $\mathcal{P}$ be the collection of (possibly unbounded) polytopes given by $\Delta^\star$ and, for each facet $f$ of $\Delta^\star$, the unbounded polytope with facet $f$ and unbounded directions given by the rays corresponding to the vertices of $\Delta^\star$ contained in $f$. 
We change the affine structure as follows. On each maximal-dimensional cell there is a canonical affine structure by embedding the polytope into $\mathbb{R}^n$. Locally around each vertex there is a canonical affine structure described by the fan locally at the vertex. We demand that the transition maps are such that the unbounded edges become all parallel.
This leads to codimension $2$ affine singularities supported on the facets of $\Delta^\star$.
If $X$ is a surface, the singular locus consists of one interior point on each of the edges of $\Delta^\star$.
Lastly, $\varphi$ is a piecewise linear function that is 0 on $\Dp$ and has slope 1 along the 1-dimensional cells of $\mathcal{P}$.
\end{construction}

\begin{example}[$\P^2$]
Following Construction \ref{dicConstruction}, we plot the dual intersection complex for the log Calabi-Yau pair $(\P^2,3L)$ in Figure \ref{fig:P2BP}.
	\begin{figure}[h!]
	\begin{tikzpicture}[scale=2]
 			\draw[opacity=0.2,fill=blue] (0,1)coordinate(B) -- (1,0)coordinate(C) -- (-1,-1)coordinate(D) --cycle;
 			
 			\draw[opacity=0.2,fill=green] ([shift={(-1/2,1/2)}]C) -- ([shift={(1/2,1/2)}]B) --([shift={(-1/2,1/2)}]B) -- ([shift={(-1/2,-1)}]B) -- (B);
            \draw[opacity=0.2,fill=green] (D) -- ([shift={(1/2,1)}]D) -- ([shift={(-0.3,0.1)}]D) --([shift={(-.3,-.3)}]D) -- cycle;

 			\draw[opacity=0.2,fill=green] (D) -- ([shift={(1,1/2)}]D) -- ([shift={(0,-0.3)}]D) -- ([shift={(-.3,-.3)}]D) -- cycle;
            \draw[opacity=0.2,fill=green] (0,-1/2) -- (1.5,-1/2) -- (1.5,0) -- (1,0) -- cycle;
            \draw[opacity=0.2,fill=green] (1,0) -- (1.5,0) -- (1.5,.5) -- (.5,.5) -- cycle;

 			\draw (1,0) -- (3/2,0);
 			\draw (0,1) -- (0,3/2);
 			\draw (-1,-1) -- (-1.3,-1.3);
 			
 			\draw[shift={(1/2,1/2)}, dashed] (0,0) -- (0,1);
 			\draw[shift={(1/2,1/2)}, dashed] (0,0) -- (1,0);
 			\draw[opacity=0.2,fill=red] (1/2,1/2) -- (1/2,3/2) -- (3/2,3/2) -- (3/2,1/2);
 			\draw[->] (1,1/2) node[anchor=south west] {$A_1$} arc (0:90:1/2) ;
            \draw[->] (-1/2,1/2) node[anchor=south east] {$A_2$} arc (90:226:1/2);
            \draw[->] (-.4,-.8) arc (232:350:1/2) ;
            \node at (0.2,-1) {$A_3$};
 			
 			\node at (1/2,1/2)  {\tiny{$\color{red}{\times}$}};
            \node at (-1/2,0)  {\tiny{$\color{red}{\times}$}};
            \node at (0,-1/2)  {\tiny{$\color{red}{\times}$}};
            
 			\draw[dashed] (-.5,0) -- (-.5,1.5);
 			\draw[dashed] (-.5,0) -- (-1.3,-0.9);
 			
 			\draw[opacity=0.2,fill=red] (-.5,0) -- (-.5,1.5) -- (-1.3,1.5) -- (-1.3,-0.9);

 			\draw[dashed] (0,-1/2) -- (1.5,-1/2);
 			\draw[dashed] (0,-1/2) -- (-1,-1.3);
 			\draw[opacity=0.2,fill=red] (0,-1/2) -- (1.5,-1/2) -- (1.5,-1.3) -- (-1,-1.3) -- cycle;
 			
 		\end{tikzpicture}
\caption{The dual intersection complex $\BP$ for $(\P^2,3L)$.}
\label{fig:P2BP}
\end{figure}
The transition maps $A_i$ are given by
\begin{align*}
 	A_1=\begin{bmatrix}
 		0 &-1\\
 		1 & 2
 	\end{bmatrix}
 	\,\, , \,\,  	
 	A_2=\begin{bmatrix}
 		3 & -1\\
 		4 & -1
 	\end{bmatrix}
 	\,\, , \,\,
 	A_3=\begin{bmatrix}
 		3 & -4\\
 		1 & -1
 	\end{bmatrix}
 \end{align*}
and the monodromy $M=A_3A_2A_1$ is
\begin{align*}
M=
    \begin{bmatrix}
        1 & 9 \\
        0 & 1
    \end{bmatrix}
    =
    \begin{bmatrix}
        1 & (-K_{\P^2})^2 \\
        0 & 1
    \end{bmatrix}.
\end{align*}
For surfaces, the intersection number $(-K_{\PD})^2$ is also the volume of the Newton polytope. That is,
	\begin{align*}
		\text{Vol}(\PD)=\int_{\PD} (c_1)^2=[-K_{\PD}]^2=\text{Vol}(\D)
	\end{align*}
	by the Duistermaat-Heckman theorem.
\end{example}

\begin{example}
Figure \ref{fig:F1intcplx} shows the intersection complex of $\mathbb{F}_1$, and the dual intersection complex in two different charts. Note that $\mathbb{F}_1$ is smooth, so the deformation only changes the divisor from the toric boundary to a smooth anticanonical divisor.
\end{example}

\begin{figure}[h!]
\centering
\begin{tikzpicture}[scale=1]
\draw (-1,-1) -- (-1,2) -- (1,0) -- (1,-1) -- (-1,-1);
\draw (0,0) -- (-1/2,-1/2);
\coordinate[fill,cross,inner sep=2pt,rotate=45] (0) at (-1/2,-1/2);
\draw[dashed] (-1/2,-1/2) -- (-1,-1);
\draw (0,0) -- (-1/2,2/2);
\coordinate[fill,cross,inner sep=2pt,rotate=26.57] (0) at (-1/2,2/2);
\draw[dashed] (-1/2,2/2) -- (-1,2);
\draw (0,0) -- (1/2,0/2);
\coordinate[fill,cross,inner sep=2pt,rotate=0] (0) at (1/2,0/2);
\draw[dashed] (1/2,0/2) -- (1,0);
\draw (0,0) -- (1/2,-1/2);
\coordinate[fill,cross,inner sep=2pt,rotate=45] (0) at (1/2,-1/2);
\draw[dashed] (1/2,-1/2) -- (1,-1);
\draw[<->] (1.5,0) -- (2.5,0);
\end{tikzpicture}
\begin{tikzpicture}[scale=.8]
\draw (1,0) -- (0,1) -- (-1,0) -- (-1,-1) -- (1,0);
\draw (1,0) -- (2,0);
\draw (0,1) -- (0,2);
\draw (-1,0) -- (-2,0);
\draw (-1,-1) -- (-2,-2);
\coordinate[fill,cross,inner sep=2pt,rotate=45] (0) at (1/2,1/2);
\draw[dashed] (2,1/2) -- (1/2,1/2) -- (1/2,2);
\coordinate[fill,cross,inner sep=2pt,rotate=45] (0) at (-1/2,1/2);
\draw[dashed] (-2,1/2) -- (-1/2,1/2) -- (-1/2,2);
\coordinate[fill,cross,inner sep=2pt,rotate=0] (0) at (-1,-1/2);
\draw[dashed] (-2,-1/2) -- (-1,-1/2) -- (-2,-1.5);
\coordinate[fill,cross,inner sep=2pt,rotate=26.57] (0) at (0,-1/2);
\draw[dashed] (-1.5,-2) -- (0,-1/2) -- (2,-1/2);
\draw (2,0) node[right]{\large$=$};
\end{tikzpicture}
\begin{tikzpicture}[scale=.8]
\draw (-1-2/3,-3) -- (-1,-1) -- (0,0) -- (1,0) -- (2,-2) -- (2+1/5,-3);
\draw (-1,-1) -- (-1,1);
\draw (0,0) -- (0,1);
\draw (1,0) -- (1,1);
\draw (2,-2) -- (2,1);
\coordinate[fill,cross,inner sep=2pt,rotate=-18.43] (1) at (-1.5,-2.5);
\coordinate[fill,cross,inner sep=2pt,rotate=45] (2) at (-1/2,-1/2);
\coordinate[fill,cross,inner sep=2pt,rotate=0] (3) at (1/2,0);
\coordinate[fill,cross,inner sep=2pt,rotate=26.57] (4) at (3/2,-1);
\draw[dashed] (-1.55,-3) -- (1) -- (2) -- (3) -- (4) -- (2,-3);
\end{tikzpicture}
\caption{The intersection complex of $\mathbb{F}_1$ (left) and the dual intersection complex of $\mathbb{F}_1$ in two different charts (middle and right).}
\label{fig:F1intcplx}
\end{figure}

\subsection{The non-Fano case} 
\label{S:toricdeg5}

In the non-Fano case, we cannot work in the polytope picture, because polytopes correspond to toric varieties with an ample polarization. This means we cannot construct a toric degeneration globally as a projective variety. However, we can still work in the fan picture and construct the toric degeneration by gluing together affine pieces described by the gluing of fans. As before the asymptotic rays are the rays of a toric model of $X$. But now the spanning polytope $\Delta^\star$ is non-convex. In order for all fans to be complete, we need to subdivide $\Delta^\star$. Geometrically this corresponds to a blow up of a torus fixed point. Since the point lies on the central fiber, this does not change the general fiber. The blow up introduces a new component of the central fiber, corresponding to the new vertex. On the other components it acts like to blow up at a point. While $\Delta^\star$ was non-convex, the cells of the new subdivision are all convex.

Note that $\mathcal{X}_\varphi$ is not defined as a subvariety of $\mathbb{P}^{N-1}\times\mathbb{A}^1$, so we cannot simply deform its defining equations. But $\mathcal{X}_\varphi$ is defined by gluing toric varieties along toric divisors, and we can perturb this gluing. This is just the same as in the Fano case, and again we introduce affine singularities on the interior edges of the dual intersection complex. So while we don't have an intersection complex for $\mathcal{X}\rightarrow\mathbb{A}^1$, we indeed have a dual intersection complex.

\begin{example}
Figure \ref{fig:F3intcplx} shows the intersection complex of $\mathbb{F}_3$, and the dual intersection complex in two different charts. The spanning polytope $\Delta^\star$ is non-convex and subdivided.
\end{example}

\begin{figure}[h!]
\centering
\begin{tikzpicture}[scale=1]
\draw (1,0) -- (0,1) -- (-1,0) -- (-3,-1) -- (1,0);
\draw (0,0) -- (2,0);
\draw (0,0) -- (0,2);
\draw (0,0) -- (-4,0);
\draw (0,0) -- (-4,-4/3);
\coordinate[fill,cross,inner sep=2pt,rotate=45] (0) at (1/2,1/2);
\draw[dashed] (2,1/2) -- (1/2,1/2) -- (1/2,2);
\coordinate[fill,cross,inner sep=2pt,rotate=45] (0) at (-1/2,1/2);
\draw[dashed] (-4,1/2) -- (-1/2,1/2) -- (-1/2,2);
\coordinate[fill,cross,inner sep=2pt,rotate=26.57] (0) at (-2,-1/2);
\draw[dashed] (-4,-1/2) -- (-2,-1/2) -- (-4,-1/2-2/3);
\coordinate[fill,cross,inner sep=2pt,rotate=14.04] (0) at (-1,-1/2);
\draw[dashed] (-4,-1.5) -- (-1,-1/2) -- (2,-1/2);
\draw (2,0) node[right]{\large$=$};
\end{tikzpicture}
\begin{tikzpicture}[scale=1]
\draw (-2-1/4,-3/2) -- (-2,0) -- (-1,1) -- (0,0) -- (1,0) -- (1+3/4,-3/2);
\draw (-2,-1) -- (-2,2);
\draw (-1,0) -- (-1,2);
\draw (0,-1) -- (0,2);
\draw (1,-1) -- (1,2);
\coordinate[fill,cross,inner sep=2pt,rotate=45] (0) at (-1.5,.5);
\coordinate[fill,cross,inner sep=2pt,rotate=45] (0) at (-.5,.5);
\coordinate[fill,cross,inner sep=2pt,rotate=0] (0) at (.5,0);
\coordinate[fill,cross,inner sep=2pt,rotate=26.57] (0) at (1.5,-1);
\draw[dashed] (-2.125,-1.5) -- (-2,-1) -- (-1.5,.5) -- (-1,0) -- (-.5,.5) -- (0,-1) -- (.5,0) -- (1,-1) -- (1.5,-1) -- (1.6,-1.5);
\end{tikzpicture}
\caption{The  dual intersection complex of $\mathbb{F}_3$ in two different charts.}
\label{fig:F3intcplx}
\end{figure}

\begin{remark}
One could think of the intersection complex for a non-Fano variety as the central subdivision of a quasi-polytope (see \S\ref{S:fanpolytope}).
\end{remark}

\subsection{Comparison with the straight boundary model} 
\label{S:toricdeg6}

In \cite{CPS}, \S8.2, the authors discuss \emph{straight boundary models} of Hirzebruch surfaces by considering $\mathbb{F}_m$ relative to the ample divisor $D = E+S+F+mF = 2E+(2m+1)F$ instead of the toric anticanonical divisor $-K_{\mathbb{F}_3}=E+S+F+F=E+(m+2)F$. Since $D$ is ample, one can work in the polytope picture as in \S\ref{S:toricdeg1}. Dualizing to the fan picture, one obtains the same dual intersection complex $B$ as in Figure \ref{fig:F3intcplx} but with a different piecewise linear function $\varphi$. It takes value $0$ only on a part of the non-convex polytope $\Delta^\star$, and it has slope $m-1$ along  the bottom left ray, which corresponds to a toric divisor of class $F$. The authors of \cite{CPS} chose this polarization to have a polytope picture. We have explained in \S\ref{S:toricdeg5} that the anticanonical polarization leads to a quasi-polytope. Apart from this, the anticanonical polarization is more natrual, e.g. when considering relative Gromov-Witten invariants, and we will use the anticanonical polarization in all of our paper.

\begin{figure}[h!]
\centering
\begin{tikzpicture}[scale=1/2]
\draw (-1,-1) -- (-1,6) -- (1,0) -- (1,-1) -- (-1,-1);
\draw (0,0) -- (-1,-1);
\draw (0,0) -- (1,-1);
\draw (0,0) -- (0,2);
\draw (0,2) -- (-1,6);
\draw (0,2) -- (1,0);
\draw (2,2) node[right]{$\longleftrightarrow$};
\draw (3,0);
\end{tikzpicture}
\begin{tikzpicture}[scale=1]
\draw (1,0) -- (0,1) -- (-1,0) -- (-3,-1) -- (1,0);
\draw (0,0) -- (2,0);
\draw (0,0) -- (0,2);
\draw (0,0) -- (-4,0);
\draw (0,0) -- (-4,-4/3);
\coordinate[fill,cross,inner sep=2pt,rotate=45] (0) at (1/2,1/2);
\draw[dashed] (2,1/2) -- (1/2,1/2) -- (1/2,2);
\coordinate[fill,cross,inner sep=2pt,rotate=45] (0) at (-1/2,1/2);
\draw[dashed] (-4,1/2) -- (-1/2,1/2) -- (-1/2,2);
\coordinate[fill,cross,inner sep=2pt,rotate=26.57] (0) at (-2,-1/2);
\draw[dashed] (-4,-1/2) -- (-2,-1/2) -- (-4,-1/2-2/3);
\coordinate[fill,cross,inner sep=2pt,rotate=14.04] (0) at (-1,-1/2);
\draw[dashed] (-4,-1.5) -- (-1,-1/2) -- (2,-1/2);
\end{tikzpicture}
\caption{Polytope and fan picture for a straight boundary model of $\mathbb{F}_3$.}
\label{fig:F3straight}
\end{figure}

\section{Scattering and broken lines} 
\label{sec:scatteringbroken}
In this section we introduce the concepts of scattering diagrams, broken lines and tropical curves, and show that they are all related to each other. Then we show that they can be used to compute relative and $2$-marked Gromov-Witten invariants. This essentially follows from a degeneration formula for logarithmic Gromov-Witten invariants.

\subsection{Scattering diagrams} 
\label{S:scattering}

Let $B$ be an affine manifold with singularities and let $\varphi$ be a piecewise affine function on $B$. In general we will assume $\varphi$ to be multi-valued to distinguish curve classes, but taking the total degree (intersection with the divisor $D$) we obtain a single valued function. Let $\iota : B^\circ \hookrightarrow B$ be the complement of the singular locus and let $\Lambda_B=\iota_\star\Lambda_{B^\circ}$ be the pushforward of the sheaf of integral tangent vectors on $B^\circ$. For simplicity we restrict to the $2$-dimensional case.

\begin{definition}
A \emph{ray} $\mathfrak{d}$ on $B$ consists of a base $b_{\mathfrak{d}} \in B$, a direction $m_{\mathfrak{d}}\in\Lambda_{B,b_{\mathfrak{d}}}$ and a function $f_{\mathfrak{d}}\in \mathbb{C}[z^{m_{\mathfrak{d}}}]\llbracket t\rrbracket$. Via parallel transport this defines a section of the sheaf $\mathcal{R}$ whose stalk at a point $x\in B$ is given by $R_x=\varprojlim \mathbb{C}[P_x]/(t^k)$, where
\[ P_x = \{p=(m,h)\in\Lambda_{B,x}\oplus\mathbb{Z} \ | \ h \leq \varphi_x(m)\}. \]
Here $\varphi_x$ is the linear part of $\varphi$ locally at $x$ and $t=z^{(0,1)}$, such that the $t$-order is given by $\varphi_x(m)-h$. Note that the $t$-order of a ray can increase at it propagates.

We demand the following properties:
\begin{itemize}
\item If $b_{\mathfrak{d}}\in\text{Sing}(B)$, then $f_{\mathfrak{d}}=1+z^m_{\mathfrak{d}}$. (This is sometimes called a \emph{slab}.) Note that in this case $\Lambda_{B,x}$ is only $1$-dimensional, so $m_{\mathfrak{d}}$ is defined up to sign.
\item If $b_{\mathfrak{d}}\notin\text{Sing}(B)$, then $f_{\mathfrak{d}} \equiv 1 \text{ mod } (t)$. (This is sometimes called a \emph{wall}.)
\end{itemize}
Note that we use different sign conventions than many other treatments of this topic. This is to avoid negative signs. For example, \cite{GPS} uses the convention $f_{\mathfrak{d}}\in \mathbb{C}[z^{-m_{\mathfrak{d}}}]\llbracket t\rrbracket$ for (``outgoing'') rays.
\end{definition}

\begin{definition}
A \emph{scattering diagram} $\mathscr{S}$ on $B$ is a collection of rays such that for each $k \geq 0$ there are only finitely many rays with $f_{\mathfrak{d}} \not\equiv 1 \text{ mod } (t^k)$. The \emph{support} $|\mathscr{S}|$ of $\mathscr{S}$ is the union of its rays, considered as subsets of $B$. A \emph{chamber} $\mathfrak{u}$ of $\mathscr{S}$ is a connected component of $B\setminus|\mathscr{S}|$. 
\end{definition}

\begin{definition}
An affine manifold with singularities $B$ defines a scattering diagram $\mathscr{S}_0(B)$ by taking all possible slabs ($2$ for each affine singularity).
\end{definition}

\begin{definition}
A point $x\in B$ where two or more rays intersect is called a \emph{vertex}. (This is sometimes called a \emph{joint}.) Locally at a vertex $x$ we produce new rays with base $x$ by the following \emph{scattering procedure}. For each ray $\mathfrak{d}$ containing $x$ consider the complement $\mathfrak{d}\setminus\{x\}$. This consists of one component (if $x=b_{\mathfrak{d}}$) or two components (otherwise). We cyclically order all such components (with respect to a simple loop around $x$ that is disjoint from $\text{Sing}(B)$) to obtain a sequence $\mathfrak{d}_1,\ldots,\mathfrak{d}_s$. Each $\mathfrak{d}_i$ defines a $\mathbb{C}\llbracket t\rrbracket$-automorphism $\theta_{\mathfrak{d}_i}$ of the localized ring $(R_x)_{\prod f_{\mathfrak{d}_i}}$ by
\[ \theta_{\mathfrak{d}_i}(z^m) = z^mf_{\mathfrak{d}_i}^{\langle n_i,m \rangle}, \]
where $n_i$ is the primitive normal vector to $m_{\mathfrak{d}_i}$, positive with respect to the chosen ordering (i.e. the orientation coming from the loop around $x$). Now, for $k \geq 0$, define
\[ \theta^{(k)} = \theta_{\mathfrak{d}_s} \circ \ldots \circ \theta_{\mathfrak{d}_1} \text{ mod } (t^{k+1}). \]
It turns out that this acts on $z^m$ by multiplication with a Laurent polynomial. Hence, to make $\theta^k$ the identity, we have to add finitely many rays. Doing this iteratively for all $k\geq 0$, we obtain scattering diagrams $\mathscr{S}_0,\mathscr{S}_1,\ldots$ that are \emph{consistent to order $k$}. Taking the formal limit we get a consistent scattering diagram $\mathscr{S}_\infty$.
\end{definition}

\begin{remark}
The idea of the Gross-Siebert program is that the dual intersection complex of $X$ after scattering, i.e. $\mathscr{S}_\infty(B)$, is the intersection complex of its mirror $\check{X}$. One can construct $\check{X}$ from $B$ by gluing thickenings of affine models \cite{GS12}. The consistency of $\mathscr{S}_\infty(B)$ ensures that the gluing does not depend on the affine chart.
\end{remark}

\subsection{Broken lines} 

\begin{definition}
\label{defi:broken}
A \textit{broken line} for a wall structure $\mathscr{S}$ on $B$ is a proper continuous map $\mathfrak{b} : (-\infty,0] \rightarrow B_0$ with image disjoint from any vertices of $\mathscr{S}$, along with a sequence $-\infty=t_0<t_1<\cdots<t_r=0$ for some $r\geq 1$ with $\mathfrak{b}(t_i)\in|\mathscr{S}|$ for $i\leq r-1$, and for each $i=1,\ldots,r$ an expression $a_iz^{m_i}$ with $a_i\in \mathbb{C}\setminus\{0\}$, $m_i\in\Lambda_{\mathfrak{b}(t)}$ for any $t\in(t_{i-1},t_i)$, defined at all points of $\mathfrak{b}([t_{i-1},t_i])$, and subject to the following conditions:
\begin{itemize}
\item $\mathfrak{b}|_{(t_{i-1},t_i)}$ is a non-constant affine map with image contained in a unique chamber $\mathfrak{u}_i$ of $\mathscr{S}$, and $\mathfrak{b}'(t)=-m_i$ for all $t\in(t_{i-1},t_i)$.
\item For each $i=1,\ldots,r-1$ the expression $a_{i+1}z^{m_{i+1}}$ is a result of transport of $a_iz^{m_i}$ from $\mathfrak{u}_i$ to $\mathfrak{u}_{i+1}$, i.e., is a term in the expansion of $\theta_{\mathfrak{d}}(a_iz^{m_i})$, where $\mathfrak{d}$ is the ray that contains the intersection of closures $\bar{\mathfrak{u}_i}\cap\bar{\mathfrak{u}_{i+1}}$.
\item $a_1=1$ and $(m_1,h)$ has $t$-order zero at $\mathfrak{b}(t_1)$, i.e., $h=\varphi(m_1)$.
\end{itemize}
Write $a_{\mathfrak{b}}z^{m_{\mathfrak{b}}}$ for the \emph{ending monomial} $a_rz^{m_r}$.
\end{definition}

\begin{example}
Figure \ref{fig:F3broken} shows the scattering diagram and some broken lines for $\mathbb{F}_3$ in two different charts. It was produced by using a \texttt{SageMath} code which can be found on the second authors webpage.
\end{example}

\begin{figure}[h!]
\centering
\begin{tikzpicture}[scale=1.5]
\clip (-5,-2) rectangle(2,1.5);
\draw[->] (-18.0, -6.0) -- (-95.5, -31.0);
\draw[->] (-18.0, -6.0) -- (-96.125, -31.0);
\draw[->] (-18.0, -6.0) -- (-96.125, -31.0);
\draw[->] (-75.0, -24.0) -- (-97.0, -31.0);
\draw[->] (-18.0, -6.0) -- (-96.571, -31.0);
\draw[->] (-3.0, -1.0) -- (-97.0, -29.2);
\draw[->] (-3.0, -1.0) -- (-83.0, -31.0);
\draw[->] (-3.0, 0.0) -- (-97.0, -15.667);
\draw[->] (-1.667, -0.667) -- (-77.5, -31.0);
\draw[->] (-1.667, -0.667) -- (-77.5, -31.0);
\draw[-] (-1.0, -0.5) -- (-5.0, -1.5);
\draw[->] (-3.0, -0.5) -- (-3.0, 31.0);
\draw[->] (-3.0, -1.0) -- (-93.0, -31.0);
\draw[->] (-3.0, -1.0) -- (-93.0, -31.0);
\draw[->] (-1.667, -0.667) -- (-92.667, -31.0);
\draw[->] (-3.0, 0.0) -- (-96.0, 31.0);
\draw[-] (-2.0, -0.5) -- (-4.0, -1.5);
\draw[->] (0.0, -0.5) -- (28.0, 4.167);
\draw[->] (-1.667, -0.667) -- (-62.333, -31.0);
\draw[-] (-0.5, 0.5) -- (-1.75, -0.75);
\draw[->] (-0.75, -0.5) -- (28.0, 4.727);
\draw[->] (-3.0, 0.0) -- (-97.0, 0.0);
\draw[->] (-3.0, 0.0) -- (-97.0, 0.0);
\draw[-] (0.5, 0.5) -- (-0.5, 1.5);
\draw[->] (-1.5, 0.5) -- (-96.0, -31.0);
\draw[->] (0.0, 1.0) -- (-10.0, 31.0);
\draw[->] (-3.0, 24.0) -- (-4.0, 31.0);
\draw[->] (0.0, 6.0) -- (-3.571, 31.0);
\draw[->] (0.0, 6.0) -- (-3.125, 31.0);
\draw[->] (0.0, 6.0) -- (-3.125, 31.0);
\draw[->] (0.0, 6.0) -- (-2.5, 31.0);
\draw[->] (0.0, 1.0) -- (0.0, 31.0);
\draw[->] (0.333, 0.667) -- (0.333, 31.0);
\draw[->] (0.0, 1.0) -- (0.0, 31.0);
\draw[-] (0.5, 0.5) -- (1.5, -0.5);
\draw[->] (-8.5, -3.0) -- (-97.0, -30.947);
\draw[->] (1.0, -0.0) -- (28.0, -0.0);
\draw[->] (2.143, 0.286) -- (28.0, 0.286);
\draw[->] (1.286, -0.286) -- (28.0, -0.286);
\draw[->] (3.0, -0.0) -- (28.0, -0.0);
\draw[->] (-1.0, 0.0) -- (28.0, 0.0);
\draw[->] (2.0, -0.0) -- (28.0, -0.0);
\draw[->] (3.0, -0.0) -- (28.0, -0.0);
\draw[->] (3.0, -0.0) -- (28.0, -0.0);
\draw[->] (3.0, -0.0) -- (28.0, -0.0);
\draw[->] (3.0, -0.0) -- (28.0, -0.0);
\draw[->] (3.0, -0.0) -- (28.0, -0.0);
\draw[->] (3.0, -0.0) -- (28.0, -0.0);
\draw[->] (3.0, -0.0) -- (28.0, -0.0);
\draw[->] (3.0, -0.0) -- (28.0, -0.0);
\draw[->] (3.0, -0.0) -- (28.0, -0.0);
\draw[->] (3.0, -0.0) -- (28.0, -0.0);
\draw[->] (3.0, -0.0) -- (28.0, -0.0);
\draw[->] (3.0, -0.0) -- (28.0, -0.0);
\draw[->] (3.0, -0.0) -- (28.0, -0.0);
\draw[->] (3.0, -0.0) -- (28.0, -0.0);
\draw[->] (3.0, -0.0) -- (28.0, -0.0);
\draw[->] (3.0, -0.0) -- (28.0, -0.0);
\draw[->] (3.0, -0.0) -- (28.0, -0.0);
\draw[->] (3.0, -0.0) -- (28.0, -0.0);
\draw[->] (3.0, -0.0) -- (28.0, -0.0);
\draw[->] (3.0, -0.0) -- (28.0, -0.0);
\draw[->] (1.615, 0.154) -- (28.0, 0.154);
\draw[->] (2.0, -0.0) -- (28.0, -0.0);
\draw[->] (2.0, -0.0) -- (28.0, -0.0);
\draw[->] (2.0, -0.0) -- (28.0, -0.0);
\draw[->] (1.0, -0.0) -- (28.0, -0.0);
\draw[->] (1.0, -0.0) -- (28.0, -0.0);
\draw[->] (2.294, -0.118) -- (28.0, -0.118);
\draw[->] (1.286, -0.286) -- (28.0, -0.286);
\draw[->] (1.154, -0.154) -- (28.0, -0.154);
\draw[->] (2.647, 0.118) -- (28.0, 0.118);
\draw[->] (2.143, 0.286) -- (28.0, 0.286);
\draw[->] (2.379, 0.069) -- (28.0, 0.069);
\draw[->] (2.172, -0.069) -- (28.0, -0.069);
\draw[-] (-0.5, 0.5) -- (0.5, 1.5);
\draw[->] (1.5, 0.5) -- (28.0, -8.333);
\draw[->] (0.333, 0.667) -- (28.0, 28.333);
\draw[->] (0.333, 0.667) -- (15.5, 31.0);
\draw[->] (0.333, 0.667) -- (15.5, 31.0);
\draw[->] (0.0, 1.0) -- (10.0, 31.0);
\draw[->] (1.0, -0.0) -- (28.0, -13.5);
\draw[->] (1.0, -0.0) -- (28.0, -13.5);
\draw[-] (-2.0, -0.5) -- (0.5, 0.75);
\draw[->] (0.75, 0.5) -- (28.0, -10.4);
\draw[->] (2.0, -0.0) -- (28.0, -8.667);
\draw[->] (1.0, -0.0) -- (28.0, -9.0);
\draw[->] (3.0, -0.0) -- (28.0, -6.25);
\draw[->] (3.0, -0.0) -- (28.0, -6.25);
\draw[-] (-1.0, -0.5) -- (3.0, 0.5);
\draw[->] (0.5, 3.0) -- (-4.167, 31.0);
\draw[->] (3.0, -0.0) -- (28.0, -5.0);
\draw[->] (3.0, -0.0) -- (28.0, -5.0);
\draw[->] (2.0, -0.0) -- (28.0, -5.2);
\draw[->] (2.0, -0.0) -- (28.0, -5.2);
\draw[->] (1.0, -0.0) -- (28.0, 5.4);
\draw[->] (1.0, -0.0) -- (28.0, 5.4);
\draw[->] (3.0, -0.0) -- (28.0, -4.167);
\draw[->] (3.0, -0.0) -- (28.0, -4.167);
\draw[->] (3.0, -0.0) -- (28.0, -4.167);
\draw[->] (1.0, -0.0) -- (28.0, -4.5);
\draw[->] (2.714, -0.286) -- (28.0, -4.5);
\draw[->] (2.714, -0.286) -- (28.0, -4.5);
\draw[->] (2.0, -0.0) -- (28.0, 4.333);
\draw[->] (3.0, -0.0) -- (28.0, -3.571);
\draw[->] (3.0, -0.0) -- (28.0, -3.571);
\draw[->] (1.0, -0.0) -- (28.0, -3.857);
\draw[->] (3.75, -0.25) -- (28.0, -3.714);
\draw[->] (3.75, -0.25) -- (28.0, -3.714);
\draw[->] (3.0, -0.0) -- (28.0, 3.571);
\draw[->] (3.0, -0.0) -- (28.0, 3.571);
\draw[->] (3.0, -0.0) -- (28.0, -3.125);
\draw[->] (3.0, -0.0) -- (28.0, -3.125);
\draw[->] (2.0, -0.0) -- (28.0, -3.25);
\draw[->] (1.0, -0.0) -- (28.0, -3.375);
\draw[->] (1.286, -0.286) -- (28.0, -3.625);
\draw[->] (3.0, -0.0) -- (28.0, 3.125);
\draw[->] (2.0, -0.0) -- (28.0, 3.25);
\draw[->] (2.0, -0.0) -- (28.0, 3.25);
\draw[->] (1.615, 0.154) -- (28.0, -2.778);
\draw[->] (3.0, -0.0) -- (28.0, -2.778);
\draw[->] (3.0, -0.0) -- (28.0, -2.778);
\draw[->] (3.0, -0.0) -- (28.0, -2.778);
\draw[->] (1.615, 0.154) -- (28.0, -2.778);
\draw[->] (3.0, -0.0) -- (28.0, 2.778);
\draw[->] (3.0, -0.0) -- (28.0, 2.778);
\draw[->] (3.0, -0.0) -- (28.0, 2.778);
\draw[->] (1.0, -0.0) -- (28.0, 3.0);
\draw[->] (3.571, 0.286) -- (28.0, 3.0);
\draw[->] (3.571, 0.286) -- (28.0, 3.0);
\draw[->] (3.0, -0.0) -- (28.0, -2.5);
\draw[->] (3.0, -0.0) -- (28.0, -2.5);
\draw[->] (3.0, -0.0) -- (28.0, -2.5);
\draw[->] (12.0, -1.0) -- (28.0, -2.6);
\draw[->] (3.857, -0.286) -- (28.0, -2.7);
\draw[->] (3.857, -0.286) -- (28.0, -2.7);
\draw[->] (2.143, 0.286) -- (28.0, -2.3);
\draw[->] (3.0, -0.0) -- (28.0, 2.5);
\draw[->] (3.0, -0.0) -- (28.0, 2.5);
\draw[->] (1.0, -0.0) -- (28.0, 2.7);
\draw[->] (4.5, 0.25) -- (28.0, 2.6);
\draw[->] (4.5, 0.25) -- (28.0, 2.6);
\draw[->] (3.0, -0.0) -- (28.0, -4.545);
\draw[->] (2.0, -0.0) -- (28.0, -4.727);
\draw[->] (2.294, -0.118) -- (28.0, -2.455);
\draw[->] (3.0, -0.0) -- (28.0, -2.273);
\draw[->] (3.0, -0.0) -- (28.0, -2.273);
\draw[->] (2.294, -0.118) -- (28.0, -2.455);
\draw[->] (3.0, -0.0) -- (28.0, 2.273);
\draw[->] (3.0, -0.0) -- (28.0, 2.273);
\draw[->] (2.0, -0.0) -- (28.0, 2.364);
\draw[->] (1.0, -0.0) -- (28.0, 2.455);
\draw[->] (2.143, 0.286) -- (28.0, 2.636);
\draw[->] (3.0, -0.0) -- (28.0, -2.083);
\draw[->] (3.0, -0.0) -- (28.0, -2.083);
\draw[->] (1.154, -0.154) -- (28.0, 2.083);
\draw[->] (3.0, -0.0) -- (28.0, 2.083);
\draw[->] (3.0, -0.0) -- (28.0, 2.083);
\draw[->] (3.0, -0.0) -- (28.0, 2.083);
\draw[->] (1.154, -0.154) -- (28.0, 2.083);
\draw[->] (3.0, -0.0) -- (28.0, -3.846);
\draw[->] (3.0, -0.0) -- (28.0, -3.846);
\draw[->] (3.857, -0.286) -- (28.0, -4.0);
\draw[->] (3.0, -0.0) -- (28.0, -1.923);
\draw[->] (3.0, -0.0) -- (28.0, -1.923);
\draw[->] (3.0, -0.0) -- (28.0, -1.923);
\draw[->] (2.0, -0.0) -- (28.0, -2.0);
\draw[->] (3.0, -0.0) -- (28.0, 1.923);
\draw[->] (3.0, -0.0) -- (28.0, 1.923);
\draw[->] (1.286, -0.286) -- (28.0, 1.769);
\draw[->] (15.0, 1.0) -- (28.0, 2.0);
\draw[->] (4.714, 0.286) -- (28.0, 2.077);
\draw[->] (4.714, 0.286) -- (28.0, 2.077);
\draw[->] (1.615, 0.154) -- (28.0, -5.5);
\draw[->] (3.0, -0.0) -- (28.0, -1.786);
\draw[->] (3.0, -0.0) -- (28.0, -1.786);
\draw[->] (3.0, -0.0) -- (28.0, -1.786);
\draw[->] (3.0, -0.0) -- (28.0, -1.786);
\draw[->] (3.0, -0.0) -- (28.0, -1.786);
\draw[->] (2.647, 0.118) -- (28.0, 1.929);
\draw[->] (3.0, -0.0) -- (28.0, 1.786);
\draw[->] (3.0, -0.0) -- (28.0, 1.786);
\draw[->] (2.647, 0.118) -- (28.0, 1.929);
\draw[->] (2.0, -0.0) -- (28.0, -6.933);
\draw[->] (2.0, -0.0) -- (28.0, -3.467);
\draw[->] (3.0, -0.0) -- (28.0, -1.667);
\draw[->] (3.0, -0.0) -- (28.0, -1.667);
\draw[->] (3.0, -0.0) -- (28.0, 1.667);
\draw[->] (3.0, -0.0) -- (28.0, 1.667);
\draw[->] (3.0, -0.0) -- (28.0, -4.688);
\draw[->] (2.294, -0.118) -- (28.0, -4.938);
\draw[->] (3.0, -0.0) -- (28.0, -1.563);
\draw[->] (5.571, -0.286) -- (28.0, -1.688);
\draw[->] (5.571, -0.286) -- (28.0, -1.688);
\draw[->] (3.0, -0.0) -- (28.0, 1.563);
\draw[->] (3.0, -0.0) -- (28.0, 1.563);
\draw[->] (2.0, -0.0) -- (28.0, 1.625);
\draw[->] (2.714, -0.286) -- (28.0, -6.235);
\draw[->] (3.75, -0.25) -- (28.0, -4.529);
\draw[->] (3.0, -0.0) -- (28.0, -2.941);
\draw[->] (3.0, -0.0) -- (28.0, -2.941);
\draw[->] (3.0, -0.0) -- (28.0, -1.471);
\draw[->] (3.857, -0.286) -- (28.0, -1.706);
\draw[->] (2.172, -0.069) -- (28.0, -1.588);
\draw[->] (2.172, -0.069) -- (28.0, -1.588);
\draw[->] (3.0, -0.0) -- (28.0, 1.471);
\draw[->] (3.0, -0.0) -- (28.0, 1.471);
\draw[->] (3.0, -0.0) -- (28.0, 1.471);
\draw[->] (3.0, -0.0) -- (28.0, 1.471);
\draw[->] (3.0, -0.0) -- (28.0, 2.941);
\draw[->] (2.0, -0.0) -- (28.0, 3.059);
\draw[->] (3.0, -0.0) -- (28.0, -1.389);
\draw[->] (3.0, -0.0) -- (28.0, -1.389);
\draw[->] (3.0, -0.0) -- (28.0, -1.389);
\draw[->] (3.0, -0.0) -- (28.0, -1.389);
\draw[->] (4.143, -0.286) -- (28.0, -1.611);
\draw[->] (3.0, -0.0) -- (28.0, 1.389);
\draw[->] (2.714, -0.286) -- (28.0, -2.947);
\draw[->] (3.0, -0.0) -- (28.0, -1.316);
\draw[->] (3.0, -0.0) -- (28.0, -1.316);
\draw[->] (3.0, -0.0) -- (28.0, -1.316);
\draw[->] (3.0, -0.0) -- (28.0, -1.316);
\draw[->] (3.0, -0.0) -- (28.0, -1.316);
\draw[->] (3.0, -0.0) -- (28.0, 1.316);
\draw[->] (6.429, 0.286) -- (28.0, 1.421);
\draw[->] (6.429, 0.286) -- (28.0, 1.421);
\draw[->] (3.0, -0.0) -- (28.0, 2.632);
\draw[->] (3.0, -0.0) -- (28.0, 2.632);
\draw[->] (4.714, 0.286) -- (28.0, 2.737);
\draw[->] (3.0, -0.0) -- (28.0, -3.75);
\draw[->] (3.0, -0.0) -- (28.0, -1.25);
\draw[->] (3.0, -0.0) -- (28.0, 1.25);
\draw[->] (2.379, 0.069) -- (28.0, 1.35);
\draw[->] (2.379, 0.069) -- (28.0, 1.35);
\draw[->] (3.0, -0.0) -- (28.0, -2.381);
\draw[->] (2.0, -0.0) -- (28.0, -2.476);
\draw[->] (3.0, -0.0) -- (28.0, -1.19);
\draw[->] (3.0, -0.0) -- (28.0, 1.19);
\draw[->] (3.0, -0.0) -- (28.0, 1.19);
\draw[->] (3.0, -0.0) -- (28.0, 1.19);
\draw[->] (3.0, -0.0) -- (28.0, 1.19);
\draw[->] (5.0, 0.286) -- (28.0, 1.381);
\draw[->] (2.0, -0.0) -- (28.0, 2.476);
\draw[->] (2.172, -0.069) -- (28.0, -3.591);
\draw[->] (3.0, -0.0) -- (28.0, -1.136);
\draw[->] (3.0, -0.0) -- (28.0, 1.136);
\draw[->] (3.0, -0.0) -- (28.0, 1.136);
\draw[->] (3.0, -0.0) -- (28.0, 1.136);
\draw[->] (3.0, -0.0) -- (28.0, 1.136);
\draw[->] (3.0, -0.0) -- (28.0, 1.136);
\draw[->] (3.0, -0.0) -- (28.0, -2.174);
\draw[->] (3.0, -0.0) -- (28.0, -1.087);
\draw[->] (3.0, -0.0) -- (28.0, -1.087);
\draw[->] (3.0, -0.0) -- (28.0, -1.087);
\draw[->] (5.571, -0.286) -- (28.0, -1.261);
\draw[->] (3.0, -0.0) -- (28.0, 1.087);
\draw[->] (3.0, -0.0) -- (28.0, 2.174);
\draw[->] (3.0, -0.0) -- (28.0, 2.174);
\draw[->] (1.154, -0.154) -- (28.0, 3.348);
\draw[->] (3.0, -0.0) -- (28.0, -1.042);
\draw[->] (3.0, -0.0) -- (28.0, -1.042);
\draw[->] (3.0, -0.0) -- (28.0, -1.042);
\draw[->] (3.0, -0.0) -- (28.0, -1.042);
\draw[->] (3.0, -0.0) -- (28.0, -1.042);
\draw[->] (3.0, -0.0) -- (28.0, -1.042);
\draw[->] (3.0, -0.0) -- (28.0, 1.042);
\draw[->] (3.75, -0.25) -- (28.0, -3.16);
\draw[->] (5.571, -0.286) -- (28.0, -2.08);
\draw[->] (3.0, -0.0) -- (28.0, -1.0);
\draw[->] (3.0, -0.0) -- (28.0, -1.0);
\draw[->] (3.0, -0.0) -- (28.0, 1.0);
\draw[->] (3.571, 0.286) -- (28.0, 2.24);
\draw[->] (3.0, -0.0) -- (28.0, 3.0);
\draw[->] (2.647, 0.118) -- (28.0, 3.16);
\draw[->] (3.0, -0.0) -- (28.0, 0.962);
\draw[->] (3.0, -0.0) -- (28.0, 0.962);
\draw[->] (3.0, -0.0) -- (28.0, 0.962);
\draw[->] (4.5, 0.25) -- (28.0, 2.962);
\draw[->] (3.0, -0.0) -- (28.0, -1.852);
\draw[->] (3.0, -0.0) -- (28.0, -1.852);
\draw[->] (3.0, -0.0) -- (28.0, -1.852);
\draw[->] (3.0, -0.0) -- (28.0, -0.926);
\draw[->] (3.0, -0.0) -- (28.0, 0.926);
\draw[->] (3.0, -0.0) -- (28.0, 0.926);
\draw[->] (3.0, -0.0) -- (28.0, 0.926);
\draw[->] (3.0, -0.0) -- (28.0, 0.926);
\draw[->] (3.0, -0.0) -- (28.0, 0.926);
\draw[->] (3.0, -0.0) -- (28.0, 0.926);
\draw[->] (3.0, -0.0) -- (28.0, 1.852);
\draw[->] (2.0, -0.0) -- (28.0, 1.926);
\draw[->] (2.0, -0.0) -- (28.0, 3.852);
\draw[->] (3.0, -0.0) -- (28.0, -0.893);
\draw[->] (3.0, -0.0) -- (28.0, -0.893);
\draw[->] (3.0, -0.0) -- (28.0, 0.893);
\draw[->] (3.0, -0.0) -- (28.0, 0.893);
\draw[->] (3.0, -0.0) -- (28.0, -0.862);
\draw[->] (3.0, -0.0) -- (28.0, -0.862);
\draw[->] (3.0, -0.0) -- (28.0, 1.724);
\draw[->] (3.0, -0.0) -- (28.0, 2.586);
\draw[->] (3.571, 0.286) -- (28.0, 3.655);
\draw[->] (3.0, -0.0) -- (28.0, -0.833);
\draw[->] (3.0, -0.0) -- (28.0, -0.833);
\draw[->] (3.0, -0.0) -- (28.0, 0.833);
\draw[->] (3.0, -0.0) -- (28.0, -2.419);
\draw[->] (2.0, -0.0) -- (28.0, -1.677);
\draw[->] (3.0, -0.0) -- (28.0, 0.806);
\draw[->] (3.0, -0.0) -- (28.0, 0.806);
\draw[->] (6.429, 0.286) -- (28.0, 1.677);
\draw[->] (2.379, 0.069) -- (28.0, 2.548);
\draw[->] (3.0, -0.0) -- (28.0, -0.781);
\draw[->] (3.0, -0.0) -- (28.0, 0.781);
\draw[->] (3.0, -0.0) -- (28.0, 0.781);
\draw[->] (3.0, -0.0) -- (28.0, -1.515);
\draw[->] (3.0, -0.0) -- (28.0, -1.515);
\draw[->] (3.0, -0.0) -- (28.0, -0.758);
\draw[->] (3.0, -0.0) -- (28.0, -0.758);
\draw[->] (3.0, -0.0) -- (28.0, -0.758);
\draw[->] (3.0, -0.0) -- (28.0, 0.758);
\draw[->] (3.0, -0.0) -- (28.0, 0.758);
\draw[->] (3.0, -0.0) -- (28.0, 1.515);
\draw[->] (3.0, -0.0) -- (28.0, 1.515);
\draw[->] (3.0, -0.0) -- (28.0, 1.515);
\draw[->] (3.0, -0.0) -- (28.0, -0.735);
\draw[->] (3.0, -0.0) -- (28.0, -0.735);
\draw[->] (3.0, -0.0) -- (28.0, -0.735);
\draw[->] (4.5, 0.25) -- (28.0, 2.324);
\draw[->] (3.0, -0.0) -- (28.0, -0.714);
\draw[->] (3.0, -0.0) -- (28.0, 0.714);
\draw[->] (3.0, -0.0) -- (28.0, 0.694);
\draw[->] (3.0, -0.0) -- (28.0, 0.694);
\draw[->] (3.0, -0.0) -- (28.0, 0.694);
\draw[->] (3.0, -0.0) -- (28.0, -1.351);
\draw[->] (3.0, -0.0) -- (28.0, 0.676);
\draw[->] (3.0, -0.0) -- (28.0, 0.676);
\draw[->] (3.0, -0.0) -- (28.0, 0.676);
\draw[->] (2.0, -0.0) -- (28.0, 1.405);
\draw[->] (3.0, -0.0) -- (28.0, -0.658);
\draw[->] (3.0, -0.0) -- (28.0, -0.641);
\draw[->] (3.0, -0.0) -- (28.0, -0.641);
\draw[->] (3.0, -0.0) -- (28.0, -0.641);
\draw[->] (3.0, -0.0) -- (28.0, 1.282);
\draw[->] (3.0, -0.0) -- (28.0, 1.282);
\draw[->] (3.0, -0.0) -- (28.0, -0.625);
\draw[->] (3.0, -0.0) -- (28.0, 1.875);
\draw[->] (3.0, -0.0) -- (28.0, 0.61);
\draw[->] (3.0, -0.0) -- (28.0, 0.595);
\draw[->] (3.0, -0.0) -- (28.0, 0.595);
\draw[->] (3.0, -0.0) -- (28.0, 0.595);
\draw[->] (3.0, -0.0) -- (28.0, -1.163);
\draw[->] (3.0, -0.0) -- (28.0, -0.581);
\draw[->] (3.0, -0.0) -- (28.0, 0.581);
\draw[->] (3.0, -0.0) -- (28.0, 1.163);
\draw[->] (3.0, -0.0) -- (28.0, -0.568);
\draw[->] (3.0, -0.0) -- (28.0, -0.556);
\draw[->] (3.0, -0.0) -- (28.0, 0.543);
\draw[->] (3.0, -0.0) -- (28.0, 0.532);
\draw[->] (3.0, -0.0) -- (28.0, -0.521);
\draw[->] (3.0, -0.0) -- (28.0, 0.521);
\draw[->] (3.0, -0.0) -- (28.0, -0.51);
\draw[->] (3.0, -0.0) -- (28.0, 1.02);
\draw[->] (3.0, -0.0) -- (28.0, 0.49);
\draw[->] (3.0, -0.0) -- (28.0, 0.481);
\draw[->] (3.0, -0.0) -- (28.0, -0.463);
\draw[->] (3.0, -0.0) -- (28.0, 0.439);
\fill[white] (10,.5) -- (.5,.5) -- (.5,10);
\fill[white] (-10,.5) -- (-.5,.5) -- (-.5,10);
\fill[white] (-10,-.5) -- (-2,-.5) -- (-11,-3.5);
\fill[white] (-10,-3.5) -- (-1,-.5) -- (10,-.5);
\draw[dashed] (0.5, 0.5) -- (10, 0.5);
\draw[dashed] (0.5, 0.5) -- (0.5, 10);
\draw[dashed] (-0.5, 0.5) -- (-0.5, 10);
\draw[dashed] (-0.5, 0.5) -- (-10, 0.5);
\draw[dashed] (-2.0, -0.5) -- (-10, -0.5);
\draw[dashed] (-2.0, -0.5) -- (-11, -3.5);
\draw[dashed] (-1.0, -0.5) -- (-10, -3.5);
\draw[dashed] (-1.0, -0.5) -- (10, -0.5);
\draw[red] (-10,-.3) -- (-3,-.3) -- (-2.7,.1) -- (-3.1,.1);
\fill[red] (-3.1,.1) circle (1pt);
\end{tikzpicture}
\hspace{1cm}
\begin{tikzpicture}[scale=.6,rotate=90]
\draw[dashed] (0.0, 0.5) -- (-1.0, 1.0);
\draw[dashed] (0.5, 1.5) -- (-1.0, 1.0);
\draw[dashed] (0.5, 1.5) -- (0.0, 2.0);
\draw[dashed] (0.5, 2.5) -- (0.0, 2.0);
\draw[dashed] (0.5, 2.5) -- (-1.0, 3.0);
\draw[dashed] (-3.0, 3.5) -- (-1.0, 3.0);
\draw[dashed] (0.0, 0.5) -- (-1.0, 0.0);
\draw[dashed] (-1.0, -0.5) -- (-1.0, 0.0);
\draw[dashed] (-1.0, -0.5) -- (-3.0, -1.0);
\draw[dashed] (-5.5, -1.5) -- (-3.0, -1.0);
\draw[->] (-1.0, -0.5) -- (-3.0, -1.5);
\draw[-] (0.5, 1.5) -- (-0.333, 0.667);
\draw[->] (-0.333, 0.333) -- (-1.25, -1.5);
\draw[-] (0.0, 1.0) -- (-1, 1.0);
\draw[->] (0.5, 2.5) -- (-0.5, 3.5);
\draw[->] (0.0, 0.5) -- (0.0, -1.5);
\draw[-] (0.0, 0.5) -- (0.0, 1.333);
\draw[->] (0.333, 1.667) -- (1.25, 3.5);
\draw[->] (0.5, 2.5) -- (3.0, 0.0);
\draw[->] (-0.667, -0.333) -- (0.5, -1.5);
\draw[->] (2.0, 1.0) -- (3.0, 1.0);
\draw[->] (0.0, 0.0) -- (3.0, 0.0);
\draw[->] (0.0, -0.714) -- (3.0, -0.714);
\draw[->] (1.714, 2.714) -- (3.0, 2.714);
\draw[->] (1.0, 2.0) -- (3.0, 2.0);
\draw[->] (-2.0, -1.0) -- (3.0, -1.0);
\draw[->] (-0.0, 3.0) -- (3.0, 3.0);
\draw[->] (0.5, 1.5) -- (2.5, 3.5);
\draw[->] (-1.0, -0.5) -- (3.0, 1.5);
\draw[->] (0.667, 2.333) -- (3.0, 3.5);
\draw[->] (-3.0, 3.5) -- (3.0, 2.5);
\draw[->] (-0.923, -0.846) -- (3.0, -1.5);
\draw[->] (-5.5, -1.5) -- (3.0, -0.286);
\draw[->] (0.923, 2.846) -- (3.0, 3.143);
\draw[red] (3,1.3) -- (1.7,1.3) -- (1.7,.85) -- (1.9,.85);
\fill[red] (1.9,.85) circle (1pt);
\end{tikzpicture}
\caption{Scattering diagram and a broken line for $\mathbb{F}_3$.}
\label{fig:F3broken}
\end{figure}

\subsection{Theta functions} 

In \cite{GHS}, theta functions are defined as sums of ending monomials of broken lines. They generalize the notion of theta functions on abelian varieties by the Mumford construction, as explained in \cite{GHS}, \S6.

\begin{definition}
\label{def:theta}
Let $(X,D)$ be a smooth log Calabi-Yau pair with dual intersection complex $B$. Let $m$ be an asymptotic direction on $B$. Let $\mathfrak{B}_m(B)_P$ be the set of broken lines $\mathfrak{b}$ on $\mathscr{S}_\infty(B)$ with initial monomial $a_1z^{m_1}=z^m$ and endpoint $\mathfrak{b}(0)=P$. Define the \emph{theta function}
\[ \vartheta_m(X,D)_P = \sum_{\mathfrak{b}\in\mathfrak{B}_m(B)_P} a_{\mathfrak{b}} z^{m_{\mathfrak{b}}}. \]
\end{definition}

In our case, with smooth divisor $D$, the dual intersection complex $B$ has exactly one unbounded direction $m_{\textup{out}}$, so asymptotic directions on $B$ are just multiples of $m_{\textup{out}}$. We write $\vartheta_q(X,D)_P$ for $\vartheta_{q\cdot m_{\textup{out}}}(X,D)_P$ with $q\in \mathbb{N}$. By \cite{CPS}, Lemma 4.7, $\vartheta_q(X,D)_P$ is the same for all points $P$ within a chamber of the scattering diagram. Wall crossing to an adjacent chamber corresponds to a mutation (see Proposition \ref{prop:crossmut}). We show in Theorem \ref{thm:theta} that for every chamber $\vartheta_1(X,D)_P$ is a mirror potential for $(X,D)$.

\begin{proposition}[\cite{GHS}, Theorem 3.24, \cite{GSintrinsic}, Theorem 1.9]
\label{prop:multiplication}
Theta functions generate a commutative ring with unit $\vartheta_0$ by the multiplication rule
\[ \vartheta_p(X,D)_P \cdot \vartheta_q(X)_P = \sum_{r=0}^\infty \alpha_{p,q}^r \vartheta_r(X,D)_P \]
with structure constants
\[ \alpha_{p,q}^r = \sum_{\substack{(\mathfrak{b}_1,\mathfrak{b}_2)\in\mathfrak{B}_p(B)_P\times\mathfrak{B}_q(B)_P \\ m_{\mathfrak{b}_1}+m_{\mathfrak{b}_2}= \ r}} a_{\mathfrak{b}_1}a_{\mathfrak{b}_2}. \]
The theta functions $\vartheta_p(X,D)_P$ only depend on the chamber containing $P$ and the structure constants $\alpha_{p,q}^r$ do not depend on $P$.
\end{proposition}

\section{Tropical curves} 
\label{S:tropical}

In this section we recall the definition of tropical curves and tropical disks and how they relate to scattering diagrams and broken lines.

\begin{definition}
Let $B$ be an affine manifold with singularities, e.g. the dual intersection complex of a smooth log Calabi-Yau pair. Write $\text{Sing}(B)$ for the singular locus (for the structure of an affine manifold with singularities), and let $i : B_0 = B\setminus\text{Sing}(B) \hookrightarrow B$ be the inclusion of the regular locus. Let $\Lambda_B=i_\star\Lambda_{B_0}$ be the sheaf of integral affine tangent vectors. Over $B_0$, the stalks of $\Lambda_B$ are isomorphic to $\mathbb{Z}^{\dim B}$. Over $\text{Sing}(B)$, the stalks have lower rank, depending on the codimension of the affine singularity.
\end{definition}

\begin{definition}
\label{defi:tropcurve}
A \emph{tropical curve} on an affine manifold with singularities $B$ is a map $h : \Gamma \rightarrow B$ from a graph $\Gamma$, without bivalent vertices but possibly with some legs (non-compact edges with only one vertex), together with
\begin{enumerate}[(1)]
\item a non-negative integer $g_V$ (\textit{genus}) for each vertex $V$;
\item a non-negative integer $\ell_E$ (\textit{length}) for each compact edge $E$;
\item a weight vector $u_{(V,E)} \in \Lambda_{B,h(V)}$ for each flag $(V,E)$;
\end{enumerate}
such that $h$ respects the graph structure of $\Gamma$ and
\begin{enumerate}[(i)]
\item if $E$ is a compact edge with vertices $V_1$, $V_2$, then $h$ maps $E$ affine linearly to the line segment connecting $h(V_1)$ and $h(V_2)$, and $h(V_2)-h(V_1)=\ell_Eu_{(V_1,E)}$. In particular, $u_{(V_1,E)} = -u_{(V_2,E)}$. Here the affine structure on $\Gamma$ is given by the lengths $\ell_E$ of its edges;
\item if $E$ is a leg with vertex $V$, then $h$ maps $E$ affine linearly either to the ray $h(V)+\mathbb{R}_{\geq 0}u_{(V,E)}$ or to the line segment $[h(V),\delta)$ for $\delta$ an affine singularity of $B$ such that $\delta-h(V)\in\mathbb{R}_{>0} u_{(V,E)}$, i.e., $u_{(V,E)}$ points from $h(V)$ to $\delta$. Note that in the latter case the direction of $u_{(V,E)}$ is determined by $\Lambda_{B,\delta}\simeq\mathbb{Z}$.
\item at each vertex $V$ we have the \emph{balancing condition}
\[ \sum_{E\ni V}u_{(V,E)} = 0. \]
\end{enumerate}
We write the set of compact edges of $\Gamma$ as $E(\Gamma)$, the set of legs as $L(\Gamma)$, the set of legs mapped to a ray (\textit{unbounded legs}) as $L_\infty(\Gamma)$ and the set of legs mapped to an open line segment (\textit{bounded legs}) as $L_\Delta(\Gamma)$ (since such edges end at the singular locus $\Delta$ of $B$).
\end{definition}

\begin{remark}
All our tropical curves will be genus $0$, so that $g_V=0$ for all vertices and $\Gamma$ is a tree.
\end{remark}

\begin{definition}
\label{defi:mult}
Let $h : \Gamma \rightarrow B$ be a tropical curve. For a trivalent vertex $V\in V(\Gamma)$ define $m_V=\lvert u_{(V,E_1)}\wedge u_{(V,E_2)}\rvert=\lvert\text{det}(u_{(V,E_1)}|u_{(V,E_2)})\rvert$, where $E_1,E_2$ are any two edges adjacent to $V$. For a vertex $V\in V(\Gamma)$ of valency $\nu_V>3$ let $h_V$ be the one-vertex tropical curve describing $h$ locally at $V$ and let $h'_V$ be a deformation of $h_V$ to a trivalent tropical curve. It has $\nu_V-2$ vertices. Define $m_V=\prod_{V'\in V(h'_V)}m_{V'}$. For a bounded leg $L\in L_\Delta(\Gamma)$ define $m_L=(-1)^{w_L+1}/w_L^2$. Then define the \textit{multiplicity} of $h$ to be
\[ \text{Mult}(h) = \frac{1}{|\text{Aut}(h)|} \cdot \prod_V m_V \cdot \prod_{L\in L_\Delta(\Gamma)} m_L. \]
\end{definition}

\subsection{Tropical disks, scattering and broken lines} 

\begin{definition}
A \textit{tropical disk} $h^\circ : \Gamma \rightarrow B$ is a tropical curve with the choice of univalent vertex $V_\infty$, adjacent to a unique edge $E_\infty$, such that the balancing condition (Definition \ref{defi:tropcurve}, (iii)) only holds for vertices $V \neq V_\infty$. Define the multiplicity $m_{V_\infty}=1$, such that $\text{Mult}(h^\circ)$ is the multiplicity of the balanced tropical curve obtained by forgetting $V_\infty$.
\end{definition}

\begin{definition}
Let $\mathfrak{d}\in\mathscr{S}_\infty(B)$ be a ray and choose a point $P\in\textup{Int}(\mathfrak{d})$. Define $\mathfrak{H}^\circ_{\mathfrak{d},w}(B)$ to be the set of all tropical disks $h^\circ : \Gamma \rightarrow B$ with no unbounded legs and such that $h^\circ(V_\infty)=P$ and $u_{(V_\infty,E_\infty)}=-w\cdot m_{\mathfrak{d}}$.
\end{definition}

\begin{proposition}[\cite{Gra22}, Proposition 5.9]
\label{prop:scattering}
For a ray $\mathfrak{d}\in\mathscr{S}_\infty(B)$ we have
\[ \textup{log }f_{\mathfrak{d}} = \sum_{w=1}^\infty\sum_{h^\circ\in\mathfrak{H}^\circ_{\mathfrak{d},w}(B)} w \text{Mult}(h^\circ) t^w z^{wm_{\mathfrak{d}}}. \]
\end{proposition}

\begin{definition}
Let $\mathfrak{H}^\circ_{q}(B)_P$ be the set of tropical disks $h^\circ : \Gamma \rightarrow B$ with one unbounded leg $E_{\text{out}}$ and such that $h^\circ(V_\infty)=P$ and $u_{(V_{\text{out}},E_{\text{out}})}=qm_{\text{out}}$ and $u_{(V_\infty,E_\infty)}=-pm_{\textup{out}}$.
\end{definition}

\begin{lemma}
\label{lem:parallel}
Let $\mathfrak{b}$ be a broken line in $\mathfrak{B}_q^{(k)}(B)_P$. If $P$ lies in an unbounded chamber of $\mathscr{S}_k$, then $\bar{m}_{\mathfrak{b}}$ is parallel to $m_{\textup{out}}$.
\end{lemma}

\begin{proof}
This is \cite{GRZ}, Proposition 3.5, or \cite{Gra2}, Proposition 2.3.
\end{proof}

\subsubsection{Disk completion} %

\begin{proposition}
\label{prop:complete}
There is a surjective map with finite preimages
\[ \mu : \mathfrak{H}^\circ_q(B)_P \rightarrow \mathfrak{B}_q(B)_P \]
defined by taking the path from $V_\infty$ to $V_{\text{out}}$. We say a tropical disk in $\mu^{-1}(\mathfrak{b})$ is obtained from $\mathfrak{b}$ by disk completion. We have
\[ a_{\mathfrak{b}} = \sum_{h^\circ\in\mu^{-1}(\mathfrak{b})} \textup{Mult}(h^\circ). \]
\end{proposition}

\begin{proof}
This is \cite{CPS}, Lemma 6.4 and Proposition 6.15, with some more details given in \cite{Gra2}, Proposition 3.2 and Proposition 4.18.
\end{proof}

\subsubsection{Leg extension} %

If $h^\circ$ is a tropical disk with $u_{(V_\infty,E_\infty)}=-wm_{\textup{out}}$, then we can forget $V_\infty$ and extend $E_\infty$ to an unbounded leg $E_{\text{out}}$ with $u_{(V_{\text{out}},E_{\text{out}})}=wm_{\text{out}}$. Hence, the above statements for tropical disks can be translated to tropical curves as follows.

\begin{definition}
Let $\mathfrak{H}_w(B)$ be the set of tropical curves on $B$ with one unbounded leg $E_{\text{out}}$ of weight $w$. This implies $u_{(V_{\text{out}},E_{\text{out}})}=wm_{\text{out}}$. Further, let $\mathfrak{H}_{p,q}(B)_P$ be the set of tropical curves on $B$ having two unbounded legs, of weights $p$ and $q$, and such that the image of the unbounded leg of weight $p$ contains $P$.
\end{definition}

\begin{corollary}
We have
\[ \textup{log }\prod_{\mathfrak{d} : m_{\mathfrak{d}} = wm_{\text{out}}} f_{\mathfrak{d}} = \sum_{w=1}^\infty\sum_{h\in\mathfrak{H}_w(B)} w \text{Mult}(h) z^{(wm_{\mathfrak{d}},0)}. \]
\end{corollary}

\begin{corollary}
\label{cor:bh}
Let $P \in B$ be a point in an unbounded chamber of $\mathscr{S}_{p+q}(B)$. Then there is a surjective map with finite preimages
\[ \mu : \mathfrak{H}_{p,q}(B) \rightarrow \mathfrak{B}_{p,q}(B)_P, \]
such that
\[ a_{\mathfrak{b}} = \sum_{h\in\mu^{-1}(\mathfrak{b})} \textup{Mult}(h). \]
A preimage is given by disk completion and leg extension.
\end{corollary}

\subsection{Tropical curve classes} 
\label{S:class}

The curve class of a tropical curve is defined by its intersection with tropical cocycles defined by \emph{elementary tropical curves}. For more details see \cite{Gra2}, \S3, and \cite{GRZ}, \S4.2.

\begin{example}
Figure \ref{fig:F3corr} shows some tropical curves on $\mathbb{F}_3$ of class $2F+E$:
\end{example}

\begin{definition}
Write $\mathfrak{H}_w(B,\beta)$ and $\mathfrak{H}_{p,q}(B,\beta)_P$ for the set of tropical curves of class $\beta$ in $\mathfrak{H}_w(B)$ and $\mathfrak{H}_{p,q}(B)_P$, respectively.
\end{definition}

\begin{figure}[h!]
\centering
\begin{tikzpicture}[scale=1.5,rotate=90]
\clip (-2,-1) rectangle (3,3);
\draw[dashed] (0.0, 0.5) -- (-1.0, 1.0);
\draw[dashed] (0.5, 1.5) -- (-1.0, 1.0);
\draw[dashed] (0.5, 1.5) -- (0.0, 2.0);
\draw[dashed] (0.5, 2.5) -- (0.0, 2.0);
\draw[dashed] (0.5, 2.5) -- (-1.0, 3.0);
\draw[dashed] (-3.0, 3.5) -- (-1.0, 3.0);
\draw[dashed] (0.0, 0.5) -- (-1.0, 0.0);
\draw[dashed] (-1.0, -0.5) -- (-1.0, 0.0);
\draw[dashed] (-1.0, -0.5) -- (-3.0, -1.0);
\draw[dashed] (-5.5, -1.5) -- (-3.0, -1.0);
\draw[->] (-1.0, -0.5) -- (-3.0, -1.5);
\draw[-] (0.5, 1.5) -- (-0.333, 0.667);
\draw[->] (-0.333, 0.333) -- (-1.25, -1.5);
\draw[-] (0.0, 1.0) -- (-1, 1.0);
\draw[->] (0.5, 2.5) -- (-0.5, 3.5);
\draw[->] (0.0, 0.5) -- (0.0, -1.5);
\draw[-] (0.0, 0.5) -- (0.0, 1.333);
\draw[->] (0.333, 1.667) -- (1.25, 3.5);
\draw[->] (0.5, 2.5) -- (3.0, 0.0);
\draw[->] (-0.667, -0.333) -- (0.5, -1.5);
\draw[->] (2.0, 1.0) -- (3.0, 1.0);
\draw[->] (0.0, 0.0) -- (3.0, 0.0);
\draw[->] (0.0, -0.714) -- (3.0, -0.714);
\draw[->] (1.714, 2.714) -- (3.0, 2.714);
\draw[->] (1.0, 2.0) -- (3.0, 2.0);
\draw[->] (-2.0, -1.0) -- (3.0, -1.0);
\draw[->] (-0.0, 3.0) -- (3.0, 3.0);
\draw[->] (0.5, 1.5) -- (2.5, 3.5);
\draw[->] (-1.0, -0.5) -- (3.0, 1.5);
\draw[->] (0.667, 2.333) -- (3.0, 3.5);
\draw[->] (-3.0, 3.5) -- (3.0, 2.5);
\draw[->] (-0.923, -0.846) -- (3.0, -1.5);
\draw[->] (-5.5, -1.5) -- (3.0, -0.286);
\draw[->] (0.923, 2.846) -- (3.0, 3.143);
\fill[blue] (2.5,.8) circle (1pt);
\draw[blue] (3,.8) node[above]{$2F+E$} -- (2,.8) node[right]{$2$} -- (1.6,.8) -- (1.6,1.4) -- (3,1.4);
\draw[blue] (1.6,.8) -- (-1,-.5);
\draw[blue] (1.6,1.4) -- (.5,2.5);
\draw[red] (0,.5) -- (0,4/3);
\draw[red] (1/3,5/3) -- (2/3,7/3) -- (.5,2.5) node[above]{$2$};
\draw[red] (2/3,7/3) -- (2,7/3) node[right]{$3$} -- (3,7/3) node[above]{$2F+E$};
\end{tikzpicture}
\caption{Some tropical curves for $\mathbb{F}_3$ of class $2F+E$.}
\label{fig:F3corr}
\end{figure}

\part{Results} 

\section{Tropical correspondence} 

\subsection{Idea of proof} 
\label{S:idea}

The geometric idea behind tropical correspondence is as follows. 

Consider a surface $X$ or a log Calabi-Yau pair $(X,D)$ with some conditions on the genus $g$, curve class $\beta$, fixed point and tangency conditions such that the (virtual) number of curves satisfying these conditions is finite. Each condition translates to a condition for tropical curves on the dual intersection complex $B$. The genus of a tropical curve is $g=g_\Gamma+\sum_V g_V$, where $g_\Gamma$ is the genus (first Betti number) of the underlying graph. We only consider the case $g=0$, so that $\Gamma$ is a tree and all $g_V=0$. The curve class $\beta$ of a tropical curve can be read off from the intersection with tropical cocycles, as described in \S\ref{S:class}. Point conditions on $X$ simply translate to point conditions on $B$. Tangency conditions of intersection multiplicity $w$ with $D$ correspond to outgoing legs with weight $u_{(V,E)}=wm_{\text{out}}$. Psi class conditions $\psi^k$ correspond to vertices with overvalency $k$, see \S\ref{S:tropdesc}. We will not discuss lambda class conditions, which are used for higher genus statements, but they correspond to a $q$-refinement of tropical multiplicites, see \cite{Bou}.

Given conditions as above, such that the (virtual) count is finite, the set of tropical curves satisfying these conditions is finite as well. The union of all such tropical curves induces a refinement of the dual intersection complex by considering each edge of a tropical curve as a new edge in the polyhedral complex. The refinement corresponds to a logarithmic modification of the Artin fan of $X$, see \cite{AW}, \S2.6. This can be seen as a generalization of toric blow ups corresponding to subdivisions of the fan. The logarithmic modification in turn gives another toric degeneration $\mathcal{Y}$ of $X$, whose central fiber $Y_0$ has more components, one for each vertex of a tropical curve satisfying the given conditions. The refinement is chosen in such a way that curves on the reducible variety $Y_0$ that are degenerations of curves on the general fiber $X$ and meet the given conditions are torically transverse, which means that they meet toric divisors in $Y_0$ only if two components come together and are disjoint from toric strata of higher codimension (torus fixed points).

Now consider a particular tropical curve satisfying the given conditions, together with an orientation of its edges. A vertex $V$ of the tropical curve maps to a vertex of the refined polyhedral complex, corresponding to a curve component $C_V$ mapping to a component $Y_V$ of $Y_0$. The curve $C_V$ on the toric variety $Y_V$ is subject to the following conditions. The genus is given by $g_V=0$. The curve class $\beta_V$ is defined by the condition that its intersection with a toric divisor of $Y_V$ equals the weight of the corresponding edge. If an edge $E$ attached to $V$ is oriented such that it points towards $V$, this induces a fixed point condition on the corresponding toric divisor. In this way we ensure that $C_V$ meets the component corresponding to the other vertex of $E$, such that they can be glued together to give a connected curve on $Y_0$. There are only finitely many (usually just one) orientation of the tropical curve such that on each component $Y_V$ the (virtual) number of curves satisfying the conditions (together with the point conditions coming from the orientation) is finite. We sum over all such orientations, i.e., over all possible gluings of components $C_V$. In this way we reduce the problem of counting curves on $X$ to computing curves on the toric varieties $Y_V$ with prescribed conditions.

On each of the toric varieties $Y_V$ we have conditions as above, a genus $g_V$, a curve class $\beta_V$ described by the weighted edges attached to $V$, point conditions coming from the orientation and possibly psi class conditions corresponding to higher valency. Hence, we can repeat the above procedure. The point conditions coming from the orientation are sufficiently general, such that the tropical curves meeting these conditions will be trivalent. Again, each vertex of a tropical curve induces conditions on the corresponding toric variety. But since the vertex is trivalent, the curve class will be very simple, and the (virtual) number of curves meeting the conditions is easily seen to match the multiplicity of the vertex. Taking the product over all vertices, we see that the multiplicity of a tropical curve matches the (virtual) number of curves on $Y_0$ that satisfy the given conditions on the components $Y_V$. Summing over all tropical curves, we see that the (virtual) number of curves matching the prescribed conditions equals the number, counted with multiplicity, of tropical curves satisfying the corresponding conditions.

There is one subtlely here. To make the above arguments rigorous, we have to apply the degeneration formula for logarithmic Gromov-Witten invariants \cite{KLR}. But to do this, we have to work with a log smooth degeneration. This is obtained by successively blowing up irreducible components on the central fiber $Y_0$ of the degeneration $\mathcal{Y}$. A component $Y_V$ has a number of log singularities that equals the number of affine singularities on the edges attached to $V$. For each such log singularity, the blow up along $Y_V$ produces an exceptional $\mathbb{P}^1$ on the central fiber. The multiplicities $m_L$ of bounded legs ending in affine singularities account for components mapping onto or intersecting the exceptional locus of this log resolution.

\subsection{Non-nef curve classes} 
\label{S:non-nef}

Note that tropical curves have edges of positive weight, by which we mean that the weight vectors are $u_{(V,E)}=w_Em_{(V,E)}$ for some $w_E> 0$ and $m_{(V,E)}$ the primitive direction vector pointing from $V$ towards $E$. If $E$ connects $V$ to $V'$, then $m_{(V',E)}=-m_{(V,E)}$ and $u_{(V',E)}=-u_{(V,E)}$, so $u_{(V',E)}w_Em_{(V',E)}$ with the same weight $w_E> 0$. In particular, the unbounded legs have positive weight. This means that the corresponding curves on $X$ have positive intersection multiplicities with the divisor. In particular $\beta\cdot D\geq 0$ (here we have $\geq$ because there need not be any unbounded legs). 

If $D=D_1\cup\ldots\cup D_r$ is reducible, then $\beta\cdot D_i \geq 0 \ \forall i$. If $D=\partial X$ is the toric boundary then $\beta$ must be nef, that is, $\beta\cdot D'\geq 0$ for any divisor $D'$, since the components of $\partial X$ generate the divisor class group of $X$. If $D$ is a smooth anticanonical divisor, then the condition $\beta\cdot D\geq 0$ is weaker than $\beta$ being nef, hence tropical correspondence applies to more curve classes in this case. For example, with smooth divisor $D$ we can compute the number $R_{0,(1)}((\mathbb{F}_3,D),F+E)=R_{0,(1)}((\mathbb{F}_1,D),E)=1$.

\subsection{Tropical correspondence for relative and $2$-marked invariants} 

The ideas of \S\ref{S:idea} have been applied explicitly in \cite{Gra22} and \cite{Gra2}. The Fano condition was not used in these arguments, and working with the dual intersection complex is possible also for non-Fano varieties, as explained in \S\ref{S:fan}. Hence, the tropical correspondence theorems of \cite{Gra22} and \cite{Gra2} carry over to this setting, giving the following.

\begin{theorem}
\label{thm:tropcorr}
Let $X$ be a (not necessarily Fano) smooth projective surface and let $D$ be a smooth anticanonical divisor. Let $B$ be the dual intersection complex of a toric degeneration of the smooth log Calabi-Yau pair $(X,D)$. Let $\beta$ be a curve class with $\beta\cdot D\geq 0$. Then
\[ R_{0,(\beta\cdot D)}((X,D),\beta) = \sum_{h\in\mathfrak{H}_{0,(\beta\cdot D)}(B,\beta)} \text{Mult}(h), \]
and
\[ R_{0,(p,q)}((X,D),\beta) = \frac{1}{p}\sum_{h\in\mathfrak{H}_{0,(p,q)}(B,\beta)_\infty} \text{Mult}(h). \]
\end{theorem}

\begin{corollary}
\label{cor:theta}
At infinity we have
\[ \vartheta_q(X,D)_\infty = y^q + \sum_{p\geq 1}\sum_{\beta : \beta.E=p+q} p R_{0,(p,q)}((X,D),\beta) s^\beta t^{p+1}y^{-p}. \]
\end{corollary}

\subsection{Tropical correspondence for descendant invariants} 
\label{S:tropdesc}

Recall the descendant Gromov-Witten invariants from \S\ref{S:descGW},
\[ D_{0,1}(X,\beta) = \int_{[\mathcal{M}_{g,1}(X,\beta)]^{\text{vir}}} \psi^{\beta\cdot D-2}\textup{ev}^\star [\textup{pt}]. \]
Note that they are only defined for $\beta\cdot D\geq 2$.

\begin{definition}
\label{defi:psi}
Let $\mathfrak{H}^\psi_{0,1}(B,\beta)_P$ be the set of tropical curves of genus $g=0$ and class $\beta$ on $B$ which have a vertex $V$ of valency (= number of attached edges) $\beta\cdot D$ at a prescribed point $P$ in general position (meaning in the interior of a chamber of the scattering diagram).
\end{definition}

\begin{remark}
By \cite{MR}, Definition 2.8, the condition in Definition \ref{defi:psi} corresponds to insertion of the class $\psi^{\beta\cdot D-2}\text{ev}^\star[\text{pt}]$. In \cite{MR}, Definition 2.8, a modification of the vertex multiplicity was defined. But in our case, with a single psi class, this modification is trivial, so the multiplicity is just $m_V$ from Definition \ref{defi:mult}.
\end{remark}

\begin{theorem}
\label{thm:tropcorrdesc}
Let $X$ be a (not necessarily Fano) smooth projective surface and let $D$ be a smooth anticanonical divisor. Let $B$ be the dual intersection complex of a toric degeneration of the smooth log Calabi-Yau pair $(X,D)$. Let $\beta$ be a curve class with $\beta\cdot D\geq 2$. Then
\[ D_{0,1}(X,\beta) = \sum_{h\in\mathfrak{H}^\psi_{0,1}(B,\beta)_P} \text{Mult}(h), \]
independent of the point $P$.
\end{theorem}

\begin{proof}
The idea of proof is described in \S\ref{S:idea} and can be made rigorous as in \cite{Gra22}. All tropical curves in $\mathfrak{H}^\psi_{0,1}(B,\beta)_P$ induce a refined toric degeneration $\mathcal{Y}$ such that a tropical curve in $\mathfrak{H}^\psi_{0,1}(B,\beta)_P$ corresponds to a torically transverse curve on the central fiber $Y_0$. The higher valency vertex $V$, corresponds to curves on the toric variety $Y_V$ with curve class $\beta_V$ defined by the edges attached to $V$, and with psi class insertion of $\psi^{\beta\cdot D-2}\text{ev}^\star[\text{pt}]$. By \cite{MR}, Theorem 1.1, the multiplicity $m_V$ equals the number of such curves. So we use the results of \cite{MR} locally at the vertex $V$ and the ideas of \S\ref{S:idea} to globalize to the whole tropical curve. The results of \cite{MR} only hold for nef curve classes. But the curve class $\beta_V$ is defined by the positive intersection multiplicities with the toric divisors of $Y_V$, given by the weights of the attached edges, hence it is nef.
\end{proof}

\begin{remark}
If we instead consider the toric boundary $D=\partial X=D_1\cup\ldots\cup D_r$, then the dual intersection complex of $(X,D)$ is simply the fan of $X$, without any affine singularities. Then the statement of Theorem \ref{thm:tropcorrdesc} only holds for curve classes $\beta$ with $\beta\cdot D_i > 0 \ \forall i$. Indeed, if $\beta\cdot D_i \leq 0$ for some $i$, then $\mathfrak{H}^\psi_{0,1}(B,\beta)_P$ is empty. 
\end{remark}

\begin{example}
If $D=\partial X$ is the toric boundary, the dual intersection complex $B$ is simply given by the fan of $X$, without affine singularities. If $\beta$ is nef and $\beta\cdot D\geq 2$, there is only one one tropical curve of class $\beta$ that fulfills the condition $\psi^{\beta\cdot D-2}\text{ev}^\star[\text{pt}]$. It has one vertex of valency $\beta\cdot D$ and, for each ray $\rho$ of $\Sigma$, it has $\beta\cdot D_\rho$ legs in direction $\rho$. The vertex has multiplicity $m_V=1$, hence
\[ \text{Mult}(h) = \frac{1}{|\textup{Aut}(h)|}m_V = \frac{1}{\sum_\rho (\beta\cdot D_\rho)!}. \]
This equals $D_{0,1}(X,\beta)$ by Proposition \ref{prop:Giv}. If $\beta$ is not nef, then there is no tropical curve that fulfills the condition $\psi^{\beta\cdot D-2}\text{ev}^\star[\text{pt}]$, since it would have an edge of negative weight. However, if $D$ is a smooth anticanonical divisor and $\beta\cdot D\geq 2$, then there exists a corresponding tropical curve in the dual intersection complex $B$ of $(X,D)$.
\end{example}

\section{The theta function is a mirror potential} 
\label{S:desc}

\begin{proposition}
\label{prop:theta}
We have, for all $k>0$ and general points $P$,
\[ \const(\vartheta_1(X,D)_P^k) = \sum_{\beta:\beta\cdot D=k} \sum_{h\in\mathfrak{H}^\psi_{0,1}(B,\beta)_P} \text{Mult}(h) z^\beta t^k. \]
\end{proposition}

\begin{proof}
Let $h : \Gamma \rightarrow B$ be a tropical curve in $\mathfrak{H}^\psi_{0,1}(B,\beta)_P$. Let $V$ be the vertex with prescribed position. Since $g=0$, $\Gamma$ is a tree. Consider it as a rooted tree with root vertex $V$. Let $E$ be an edge attached to $V$ and follow its path to a leaf of $\Gamma$. Since $P$ is in general position, the path cannot end in an affine singularity of $B$. Otherwise, the path would be contained in the support (= union of rays) of the scattering diagram, and so would be $P$. So the leaf must be an unbounded leg. The same is true for all edges attached to $V$. By definition of $\mathfrak{H}^\psi_{0,1}(B,\beta)_P$, there are $\beta \cdot D$ such edges, hence $\beta \cdot D$ unbounded legs in $\Gamma$. By tropical correspondence, $\beta \cdot D$ also equals the sum of weights of unbounded legs. Hence, all unbounded legs must have weight $1$.

By the multiplication rule of theta functions (Proposition \ref{prop:multiplication}), $\const(\vartheta_1(q)^m)$ is a sum over $m$-tuples of broken lines that are balanced in the endpoint $P$. Given such a tuple $(\mathfrak{b}_1,\ldots\mathfrak{b}_m)$, the union of the broken lines $\mathfrak{b}_1,\ldots\mathfrak{b}_m$ can be viewed as tropical curve that has an $m$-valent vertex mapping to $P$. By Proposition \ref{prop:complete}, its multiplicity is given by $\prod_{i=1}^m a_{\mathfrak{b}_i}$, which equals term appearing in the mutliplication rule for $\const(\vartheta_1(q)^m)$.
\end{proof}

\begin{lemma}
\label{lem:homog}
$\const(\vartheta_1(X,D)_P^k)$ is homogeneous in $t$ of order $k$, hence $\const(t^{-1}\vartheta_1(X,D)_P^k)$ is independent of $t$.
\end{lemma}

\begin{proof}
This is because a term in $\vartheta_1(X,D)_\infty$ with $z^\beta$ has $t$-order $\beta\cdot D$ and $\const(\vartheta_1(X,D)_P^k)$ is independent of the chamber.
\end{proof}

\begin{theorem}
\label{thm:theta}
For every point $P$ inside a chamber in the scattering diagram, $t^{-1}\vartheta_1(X,D)_P$ is a mirror potential for $(X,D)$, in the sense of Definition \ref{defi:MS},
\[ \pi_{t^{-1}\vartheta_1(X,D)_P}(z) = G(z). \]
\end{theorem}

\begin{proof}
This follows from Theorem \ref{thm:tropcorrdesc}, Proposition \ref{prop:theta} and Lemma \ref{lem:homog}.
\end{proof}

As a consequence, we can define the mirror maps (see \S\ref{S:mirmap}) using the classical period of the corrected potential $\vartheta_1$, which can be computed combinatorially.

\begin{remark}
If $X$ is Fano, then there is a ``central'' chamber given by the spanning polytope $\Delta^\star$ of $\Sigma$ (where $X_\Sigma$ is a toric variety to which $X$ admits a $\mathbb{Q}$-Gorenstein degeneration). If $P$ is inside the central chamber, then $\vartheta_1(X,D)_P = W_\Sigma$.
\end{remark}

\section{The open mirror map counts $2$-marked invariants} 
\label{S:ms}

Recall the definition of mirror maps from \S\ref{S:mirmap}.

\begin{theorem}
\label{thm:mirmap}
Let $(X,D)$ be a smooth log Calabi-Yau pair with mirror dual potential $W$. Then, under the change of variables $Q_i=z_i(t/y)^{d_i}$, with $d_i=\beta_i\cdot D$, we have
\[ \vartheta_1(t,z,y)_\infty = yM_W(Q), \]
where $\vartheta_1(t,z,y)_\infty := \vartheta_1(X,D)_\infty$ is the theta function at infinity.
\end{theorem}

To prove Theorem \ref{thm:mirmap} we first we need two lemmas.

\begin{definition}
For a Laurent series $f(x)=\sum_{k\geq k_{\min}} f_kx^k$ write $[x^k]f=f_k$.
\end{definition}

\begin{lemma}[Lagrange inversion for Laurent series]
\label{lem:inversion}
Let $f(y)=y^{-1}+\sum_{k\geq 0}f_ky^k$ be a Laurent series with only simple poles. Then the composition inverse $g(z)$ satisfying $f(g(z))=z$ is a power series in $z^{-1}$ given by
\[ g(z) = \sum_{k> 0}\frac{z^{-k}}{k}[x^{-1}]f^k. \]
\end{lemma}

\begin{proof}
Look at the multiplicative inverse
\[ \frac{1}{f(y)}=\frac{y}{h(y)}, \quad h(y)=yf(y)=1+\sum_{k\geq 0}f_ky^{k+1}. \]
This is a power series with nonzero linear term, and by the Laurent inversion theorem we have $1/f(g(z))=z$ for
\[ g(z) = \sum_{k>0}\frac{z^k}{k}[x^{k-1}]h^k = \sum_{k>0}\frac{z^k}{k}[x^{-1}]f^k. \]
The equation $1/f(g(z))=z$ is equivalent to $f(g(z))=z^{-1}$ and to $f(g(z^{-1}))=z$. So the inverse of $f(y)$ is $g(z^{-1})$ as claimed.
\end{proof}

\begin{remark}
Note that the Laurent series $f(y)$ is not bijective, so there is no global inverse. The inverse is around $z=0$ for which $1/f(g(z))$ is finite.
\end{remark}

\begin{lemma}
\label{lem:3}
Let $f(x)=1+\sum_{k\geq 1}f_kx^k$ be a power series. Then
\[ \exp\left(\sum_{k>0}\frac{1}{k}[x^k]f^kz^k\right) = \sum_{k>0}\frac{1}{k}[x^{k-1}]f^kz^{k-1}. \]
\end{lemma}

\begin{proof}
This is proved in \cite{GRZZ}, Lemma A.12, using Bell polynomial identities.
\end{proof}

\begin{proof}[Proof of Theorem \ref{thm:mirmap}]
By Theorem \ref{thm:theta} we can take $W=t^{-1}\vartheta_1(X,D)$. Then the claimed equation is
\[ \vartheta_1(t,z,y)_\infty = yM_{t^{-1}\vartheta_1}(Q). \]
By Definition \ref{defi:mirmap}, the closed mirror map is $Q_i=z_iM_{t^{-1}\vartheta_1}(Q(z))^{d_i}$, so that the change of variables $Q_i=z_i(t/y)^{d_i}$ means $-t/y=M_{\vartheta_1}(Q)$, or
\[ y(t,z) = -tM_{t^{-1}\vartheta_1}(Q(z))^{-1} = -te^{-\sum_{k>0}\frac{(-1)^k}{k}[y^0]\vartheta_1^kt^{-k}}. \]
Note that this only depends on $t$ by the global $t$-factor. Moreover, $-t/y=M_{\vartheta_1}(Q)$ shows that the right hand side of the claimed equation is $y(t,z)M_{\vartheta_1}(Q(z))=-t$. Consider the Laurent series in $y$ over $\mathbb{C}[z_1,\ldots,z_r]$ given by $f(y):=t^{-1}\vartheta_1(t,t/y,z)$. Then the claimed equation is equivalent to $tf(t/y(t,z))=-t$, hence to $f(t/y(t,z))=-1$, which is
\[ f\left(-e^{\sum_{k>0}\frac{(-1)^k}{k}[y^0]f^k}\right) = -1. \]
This follows from the equation
\[ f\left(te^{\sum_{k>0}\frac{1}{k}[y^0]f^kt^{-k}}\right) = t \]
by mapping $t\mapsto -t$ and then setting $t=1$. By Lemma \ref{lem:inversion}, the function $g(t)$ satisfying $f(g(t))=t$ is given by
\[ g(t) = \sum_{k>0}\frac{1}{k}[y^{-1}]f^kt^{-k}. \]
Hence, the claimed equation follows from
\[ t\exp\left(\sum_{k>0}\frac{1}{k}[y^0]f^kt^k\right) = \sum_{k>0}\frac{1}{k}[y^{-1}]f^kt^k. \]
This follows from Lemma \ref{lem:3} applied to the power series $yf(y)$.
\end{proof}

\begin{corollary}
We have
\[ M_{t^{-1}\vartheta_1}(Q) = 1+\sum_{\beta \in NE(X)} (\beta\cdot D-1)R_{\beta\cdot D-1,1}(X,\beta)Q^\beta. \]
\end{corollary}

\begin{proof}
Combine Theorem \ref{thm:mirmap} and Corollary \ref{cor:theta}.
\end{proof}

\section{Mutations} 
\label{S:mutation}

Mutations of Laurent polynomials were introduced in \cite{FZ}\cite{ACGK} and related to mirror symmetry in \cite{KNP}\cite{ACH}\cite{CKPT}.

\begin{definition}
Let $N\simeq\mathbb{Z}^2$ be a lattice and let $w\in M$ be a primitive vector in the dual lattice. Then $w$ induces a grading of $\mathbb{C}[N]$. Let $a\in\mathbb{C}[w^\bot\cap N]$ be a Laurent polynomial in the zeroth piece of $\mathbb{C}[N]$, where $w^\bot\cap N=\{v \in N \ | \ w(v)=0\}$. The pair $(w,a)$ defines an automorphism of the Laurent polynomial ring $\mathbb{C}(N)$ by
\[ \mu_{w,a} : \mathbb{C}(N) \rightarrow \mathbb{C}(N), \quad x^v \mapsto x^va^{w(v)}. \]
Let $f\in\mathbb{C}[N]$. We say $f$ is \emph{mutable with respect to $(w,a)$} if $\text{mut}_{w,a}(f) \in \mathbb{C}[N]$, in which case we call $\text{mut}_{w,a}(f)$ a \emph{mutation} of $f$.
\end{definition}

\begin{remark}
The \emph{Laurent phenomenon} of \cite{ACGK} in our case means that in any chamber inside $\Delta^\star$ the potential $\vartheta_1(X,D)$ is a Laurent polynomial and mutable with respect to its boundary walls, so that it stays a Laurent polynomial after mutation.
\end{remark}

\begin{proposition}
\label{prop:crossmut}
The wall crossing morphisms $\theta_{\mathfrak{d}}$ are equal to mutations $\text{mut}_{w,a}$ with $w=n_{\mathfrak{d}}$ and $a=f_{\mathfrak{d}}$.
\end{proposition}

\begin{proof}
This follows from the definition.
\end{proof}

\begin{proposition}
\label{prop:laurent}
In any bounded chamber and away from the dense region, $\vartheta_1(X,D)$ is a Laurent polynomial.
\end{proposition}

\begin{proposition}
If $f$ and $g$ are mutation equivalent Laurent polynomials, then their classical periods (Definition \ref{defi:pi}) agree: $\pi_f=\pi_g$.
\end{proposition}

\begin{proof}
This follows from application of the change-of variables-formula to the period integral
\[ \pi_f(z) = \left(\frac{1}{2\pi i}\right)^n \int_{|x_1|=\ldots=|x_n|=1} \frac{1}{1-f}\frac{dx_1}{x_1}\cdots\frac{dx_n}{x_n}. \]
\end{proof}

\begin{corollary}
If two (not necessarily Fano) log Calabi-Yau pairs $(X,D)$ and $(Y,E)$ have mutation equivalent potentials $\vartheta_1(X,D)_P$ and $\vartheta_1(Y,E)_{P'}$ for some points $P$, $P'$ inside chambers of their respective scattering diagrams, then we have
\[ \pi_{\vartheta_1(X,D)_P} = \pi_{\vartheta_1(Y,E)_{P'}}, \]
hence
\[ G_X(z) = G_Y(z), \]
hence
\[ M_X(Q) = M_Y(Q), \]
hence , if $\beta$ mutates to $\beta'$,
\[ R_{0,(1,\beta\cdot D-1)}((X,D),\beta) = R_{0,(1,\beta'\cdot E-1)}((Y,E),\beta'). \]
Moreover, since all $2$-marked invariants with order $(p,q)$ are determined by the ones with orders $(1,p+q-1)$ (see e.g. \cite{Gra2}, Theorem 2), for all $p,q$ we have
\[ R_{0,(p,q)}((X,D),\beta) = R_{0,(p,q)}((Y,E),\beta'). \]
\end{corollary}

\begin{definition}
A lattice polytope $\Delta$ is \emph{Fano} if all vertices of $\Delta$ are primitive lattice points and the origin is contained in the interior of $\Delta$.
\end{definition}

\begin{proposition}[\cite{ACGK}, Proposition 2]
The mutation of a Fano polytope is again a Fano polytope.
\end{proposition}

\begin{conjecture}
\label{conj:model}
In any bounded chamber of the scattering diagram, the Newton polytope of $\vartheta_1(X,D)$ is a Fano polytope $\overline{\Delta}^\star$ and $X$ admits a $\mathbb{Q}$-Gorenstein degeneration to the toric variety $X_\Sigma$ defined by the spanning fan $\Sigma$ of the Newton polytope of $\overline{\Delta}^\star$. In particular, two (not necessarily Fano) varieties are $\mathbb{Q}$-deformation equivalent if and only if their theta functions $\vartheta_1$ are mutation equivalent.
\end{conjecture}

We will see in \S\ref{S:hirz} that Conjecture \ref{conj:model} holds for Hirzebruch surfaces.

\section{Comparison with local/open MS and PF equations} 

As explained in \S\ref{S:GWlocal} and \S\ref{S:GWopen} the total space of the canonical bundle $K_X$ is a Calabi-Yau threefold. From the Mori vectors that describe relations between the rays in the fan of $K_X$, one can construct a Picard-Fuchs type differential equation. If $X$ is Fano, then the period $G_X(z)=\pi_W(z)$ is the holomorphic part of a logarithmic solution to this differential equation. Surprisingly, this still holds in the semi-Fano case, although the combinatorics change slightly. We show for the example $\mathbb{F}_3$ that this doesn't hold in the non-semi-Fano case.

To construct the spanning polytope for the canonical bundle $K_{\YD}$, we take the cone over $\Dp$ by placing $\Dp$ in a hyperplane of height 1 above the origin in $\R^{n+1}$
\[ \text{Cone}(\Dp)=\left\{(c\nu,c) \,|\, \nu\in\Dp \,,\,c\in\R_{\ge0} \right\}\subset \R^n\times\R \]
Let $v_i\in\Dp$ be the vertices of the spanning polytope and set $\overline{v}_i=(v_i,1)\in\Dp\times\{1\}\subset\text{Cone}(\Dp)$.
Define the vectors $\{\ell^{(a)}\}_{a=1}^{h^{1,1}(\PD)}$ from the linear relations of the $\overline{v}_i$
\[ \sum_{i=0}^r \ell^{(a)}_i\overline{v}_i=0 \]
where $v_0=(0,\dots,0)\in\R^{n}$.
By including $v_0$ we are restricting to the anticanonical hypersurface $X\hookrightarrow\PD$, and the Calabi-Yau condition is equivalent to $\sum_i\ell^{(a)}_i=0$.
Then $\ell^{(a)}$ are the generators of the Mori cone, which is dual to the K\"{a}hler cone of $\PD$.
The \textit{Mori vectors} $\ell^{(a)}$ allow us to determine the Picard--Fuchs operators
\[ \L_a=\prod_{\ell_i^{(a)}>0}\prod_{j=0}^{\ell^{(a)}_i-1}\left(\sum_{b=1}^s\ell^{(b)}_i\theta_b-j\right)-z_a\prod_{\ell_i^{(a)}<0}\prod_{j=0}^{-\ell^{(a)}_i-1}\left(\sum_{b=1}^s\ell^{(b)}_i\theta_b-j\right) \]
where $\theta_a=z_a\partial_{z_a}$ is the logarithmic derivative with respect to $z_a$.

\begin{proposition}
For each $a$, define
\[ \pi_a(z) = \sum_\beta \frac{(-1)^{\beta\cdot D_i}(-(\beta\cdot D_i)-1)!}{\prod_{j\neq i} (\beta \cdot D_j)!}z^\beta \]
Then $\L_i(\log(z_a)+\pi_a(z)) = 0$ for all $i$.
\end{proposition}

\begin{example}[Closed mirror map for $K_{\F_3}$ from the PF system]
\label{eg:KF3ClosedMM}
For the Hirzebruch surfaces $\mathbb{F}_m$ with $m>2$, we find that the spanning polytope $\Dp_m$ is not reflexive. The vertices satisfy the linear relations
\begin{alignat*}{6}
    \ell^{(1)}=(\ &&-2;\ &&1,\ && 1,\ && 0,\ && 0)\\
    \ell^{(2)}=(\ && -(2-m);\ && -m,\ && 0,\ && 1,\ && 1)
\end{alignat*}
so for $m>2$ we have $\ell^{(2)}_0>0$.

    We use these to construct the Picard--Fuchs operators
    \begin{align*}
        \L_1&=\theta_1(\theta_1-3\theta_2)-z_1(-2\theta_1+\theta_2)(-2\theta_1+\theta_2-1)\\
        \L_2&=(-2\theta_1+\theta_2)\theta_2-z_2(\theta_1-3\theta_2)(\theta_1-3\theta_2-1)(\theta_1-3\theta_2-2)
    \end{align*}
    Let
    \begin{align*}
    	F_a(z_1,z_2)=\sum_{n_1,n_2\ge 0} c^{(a)}_{n_1,n_2}z_1^{n_1}z_2^{n_2}
    \end{align*}
    be power series in $z_1$ and $z_2$. 
    The power series terms $S_a$ of the closed mirror maps $t_a=\log z_a + S_a(z_1,z_2)$ will be written as linear combinations of the $F_a$.
    By imposing $\L_b F_a=0$, we get recurrence relations on the coefficients $c^{(a)}_{n_1,n_2}$.
    \begin{align*}
    	\L_1F_a=0\implies c^{(a)}_{n_1+1,n_2}&=\frac{(2n_1-n_2)(2n_1-n_2+1)}{(n_1+1)(n_1-3n_2+1)}c^{(a)}_{n_1,n_2}\\
    	\L_2F_a=0\implies c^{(a)}_{n_1,n_2+1}&=-\frac{(n_1-3n_2)(n_1-3n_2-1)(n_1-3n_2-2)}{(n_2+1)^2(2n_1-n_2-1)}c^{(a)}_{n_1,n_2}
    \end{align*}
    For the first solution, we solve the first equation with $n_2=0$
    \begin{align*}
    	c^{(1)}_{n_1+1,0}=\frac{2n_1(2n_1+1)}{(n_1+1)^2}c^{(1)}_{n_1,0}\implies c^{(1)}_{n_1,0}=\frac{(2n_1-1)!}{(n_1!)^2}
    \end{align*}
    With similar factorial recursion identities (coming from the properties of the Gamma function), one can substitute $c^{(1)}_{n_1,0}$ into the second equation to solve for $c^{(1)}_{n_1,n_2}$.
    \begin{align*}
    	c^{(1)}_{n_1,n_2}&=(-1)^{n_2}\frac{(2n_1-n_2-1)!}{(2n_1-1)!}\frac{n_1!}{(n_1-3n_2)!}\frac{1}{(n_2!)^2}c^{(1)}_{n_1,0}\nonumber\\
    	&=(-1)^{n_2}\frac{(2n_1-n_2-1)!}{n_1!(n_1-3n_2)!(n_2!)^2}.
    \end{align*}
    For the second solution, we first solve the second equation with $n_1=0$
    \begin{align*}
    	c^{(2)}_{0,n_2+1}=-\frac{3n_2(3n_2+1)(3n_2+2)}{(n_2+1)^3}c^{(2)}_{0,n_2}\implies c^{(2)}_{0,n_2}=(-1)^{n_2}\frac{(3n_2-1)!}{(n_2!)^3}.
    \end{align*}
    After factoring out a minus sign to apply more factorial identities, one can substitute $c^{(2)}_{0,n_2}$ into the first equation to solve for $c^{(2)}_{n_1,n_2}$.
    \begin{align*}
    	c^{(2)}_{n_1,n_2}&=(-1)^{n_1}\frac{(3n_2-n_1-1)!}{(3n_2-1)!}\frac{n_2!}{(n_2-2n_1)!}\frac{1}{n_1!}c^{(2)}_{0,n_2}\nonumber\\
    	&=(-1)^{n_1+n_2}\frac{(3n_2-n_1-1)!}{n_1!(n_2-2n_1)!(n_2!)^2}
    \end{align*}
    One can check that $c^{(1)}_{n_1,n_2}$ and $c^{(2)}_{n_1,n_2}$ solve the recurrence relations.
    Now to obtain the log solutions we note that $\L_a t_b=\L_a(\log z_b+S_b)=0$ and hence $\L_a S_b=-\L_a \log z_b$. We calculate
    \begin{align*}
    		\L_1 S_1&=2z_1  \ && \L_1 S_2=-z_1\\
    		\L_2 S_1&=2z_2  \ && \L_2 S_2=-6z_2
    \end{align*}
    	Looking at the lowest order terms of $F_1$ and $F_2$ we have
    \begin{align*}
    		c^{(1)}_{1,0}z_1&=z_1 \ && c^{(1)}_{0,1}z_2=0 \\
    		c^{(2)}_{1,0}z_1&=0 \ && c^{(2)}_{0,1}z_2=-2z_2.
    \end{align*}
    This fixes $S_1$ and $S_2$
    \begin{align*}
    	S_1=2F_1-F_2 \ && S_2=-F_1+3F_2
    \end{align*}
    where
    \begin{align*}
    	F_1(z_1,z_2)&=\sum_{\substack{n_1,n_2\ge 0\\ n_1\ge 3n_2}} (-1)^{n_2}\frac{\Gamma(2n_1-n_2)}{\Gamma(n_1)\Gamma(n_1-3n_2+1)\Gamma^2(n_2+1)}z_1^{n_1}z_2^{n_2}\\
    	F_2(z_1,z_2)&=\sum_{\substack{n_1,n_2\ge 0\\ n_1\le 3n_2}} (-1)^{n_1+n_2}\frac{\Gamma(3n_2-n_1)}{\Gamma(n_1+1)\Gamma(n_2-2n_1+1)\Gamma^2(n_2+1)}z_1^{n_1}z_2^{n_2}
    \end{align*}
    The $z_1$ parameter is mirror to the fiber class parameter
    \begin{align*}
    	q_1=z_1 e^{S_1}
    \end{align*}
    and the $z_2$ parameter is mirror to the exceptional class parameter
    \begin{align*}
    	q_2=z_2 e^{S_2}
    \end{align*}
\end{example}

\begin{example}[Open mirror map for $K_{\F_3}$ from the PF system]
	We include the additional Mori vector $\ell^{(0)}$ to obtain
	\begin{alignat*}{8}
    \ell^{(1)}=(\ &&-2,\ &&1,\ && 1,\ && 0,\ && 0;\ && 0,\ && 0)\\
    \ell^{(2)}=(\ && -(2-m),\ && -m,\ && 0,\ && 1,\ && 1;\ && 0,\ && 0)\\
    \ell^{(0)}=(\ && 1,\ && 0,\ && 0,\ && -1,\ && 0;\ && -1,\ && 1)
\end{alignat*}
Now the Picard--Fuchs operators are modified
\begin{align*}
        \L_1&=\theta_1(\theta_1-3\theta_2)-z_1(-2\theta_1+\theta_2+\theta_0)(-2\theta_1+\theta_2+\theta_0-1)\\
        \L_2&=(-2\theta_1+\theta_2-\theta_0)(\theta_2-\theta_0)-z_2(\theta_1-3\theta_2)(\theta_1-3\theta_2-1)(\theta_1-3\theta_2-2)\\
        \L_0&=(\theta_0-2\theta_1+\theta_2)\theta_0+z_0(\theta_2-\theta_0)\theta_0
    \end{align*}
    The log solutions $t_1,t_2$ still solve this open-closed Picard--Fuchs system since $S_1,S_2$ are only functions of the closed string moduli and hence $\theta_0 t_a=0$. 
    However, we have an additional solution $t_0=\log z_0 + S_0(z_1,z_2)$, and we solve for $S_0$ in terms of $F_1$ and $F_2$ the same way we solved for $S_1$ and $S_2$.
    Namely, we compute $\L_b S_0 = -\L_b \log z_0$ and choose the appropriate linear combination which reproduces this lowest order term. 
    \begin{equation*}
    \left.
    \begin{aligned}
    	\L_1 S_0 &= -\L_1\log z_0 = -z_1\\
    \L_2 S_0 &= -\L_2\log z_0 = 0
    \end{aligned}\right\}\implies S_0=-F_1
    \end{equation*}
    Hence we can compute the open mirror map
    \begin{align*}
    	t_0=z_0e^{S_0}
    \end{align*}
\end{example}

\begin{table}[h!]
\centering
\begin{tabular}{ ||c || c c c c c c c c c c||} 
 \hline
 $c_{k_1,k_2}$ & 0 & 1 & 2 & 3 & 4 & 5 & 6 & 7 & 8 & 9\\ [0.5ex] 
 \hline\hline

 0 & 1 & 1 & 0 & 0 & 0 &0&0&0&0&0\\ 
 \hline
 1 & 0 & $-2$ & $-2$ & $-4$ & $-6$ & $-8$ & $-10$ & $-12$ &$-14$ &$-16$  \\
 \hline
 2 & 0 & 5 & 8 & 9 & 20 & 56 & 162 & 418 & 948 & 3621   \\
 \hline
 3 & 0 & $-32$ & $-70$ &  $-96$ & $-140$ & $-300$ & $-768$ & $-2220$ & $-6756$ & $-20440$ \\
 \hline
\end{tabular}
\caption{The numbers obtained from the Picard-Fuchs system for $K_{\F_3}$. Here $k_F$ runs horizontally and $k_E$ runs vertically.}
\label{tab:F3invLM}
\end{table}

These numbers to not agree with the $2$-marked invariants below, obtained from the mirror maps defined by the corrected potential $\vartheta_1(\mathbb{F}_3)$, which are equal to the $2$-marked invariants of $\mathbb{F}_1$ after a class shift $E\mapsto E-F$. This show that the method with Picard-Fuchs equations does not work in the non-semi-Fano case.

\begin{table}[h!]
\centering
\begin{tabular}{ ||c || c c c c c c c c c c||} 
 \hline
 $c_{k_1,k_2}$ & 0 & 1 & 2 & 3 & 4 & 5 & 6 & 7 & 8 & 9\\ [0.5ex] 
 \hline\hline

 0 & 1 & 1 & 0 & 0 & 0 &0&0&0&0&0\\ 
 \hline
 1 & 0 & 0 & $-2$ & $-4$ & $-6$ & $-8$ & $-10$ & $-12$ &$-14$ &$-16$  \\
 \hline
 2 & 0 & 0 & 0 & 0 & 5 & 35 & 135 & 385 & 910 & 1890   \\
 \hline
 3 & 0 & 0 & 0 &  0 & 0 & 0 & $-32$ & $-400$ & $-2592$ & $-11760$ \\
 \hline
\end{tabular}
\caption{The $2$-marked Gromov-Witten invariants for $\F_3$. Here $k_F$ runs horizontally and $k_E$ runs vertically. These are the same as for $\F_1$ but with $q_1\mapsto q_1q_2$.}
\label{tab:F3inv}
\end{table}

\part{Examples} 

\section{Hirzebruch surfaces} 
\label{S:hirz}

Recall the Hirzebruch surfaces $\mathbb{F}_m$ from Example \ref{eg:Hirz}. The complex parameters $z_1$ and $z_2$ correspond to the class of a fiber $F$ and of the exceptional section $E$, respectively. Let $D$ be a smooth anticanonical divisor, i.e., of class $(m+2)F+2E$. We compute the potentials $\vartheta_1((\mathbb{F}_m,D))_P$ for points $P$ inside different chambers of the scattering diagram. Consider the fan $\Sigma$ of $\mathbb{F}_m$ generated by $(1,0)$, $(0,1)$, $(-1,0)$ and $(-m,-1)$. 

We omit the variables $z_i$, i.e. the curve classes. One can find the curve class corresponding to a broken line by shifting the endpoint to infinity and completing the broken lines to a tropical curve, as in Proposition \ref{prop:complete}. Then one can calculate the curve class of the tropical curve by intersection with tropical cocycles.

\subsection{Relating scattering diagrams of Hirzebruch surfaces} 

We know from Example \ref{eg:Hirz} that the Hirzebruch surfaces $\mathbb{F}_m$ for $m$ even (resp. odd) are all $\mathbb{Q}$-Gorenstein deformation equivalent. This implies that their theta functions are mutation equivalent, and that their scattering diagrams are related to each other, such that they can be seen as different initial chambers inside one and the same scattering diagram. Here we describe precisely how this relation works. Start with the scattering diagram of $\mathbb{F}_m$, for $m=3$ see Figure \ref{fig:relation}. Move the affine singularity whose initial rays have slope $1/(m-1)$ all the way to the right, applying the monodromy transformation of the affine singularity whose initial rays have slope $1/(m+1)$. This transformation is
\[ \begin{pmatrix}-m&(m+1)^2\\-1&m+2\end{pmatrix}\begin{pmatrix}m-1\\1\end{pmatrix}=\begin{pmatrix}-(m+3)\\-1\end{pmatrix}. \]
Hence, the initial rays with slope $1/(m-1)$ become the bottom rays with slope $1/(m+3)$ of the scattering diagram for $\mathbb{F}_{m+2}$. At the same time, move the affine singularity whose initial rays have slope $1/(m+1)$ to the left of the initial ray with slope $-1$.

\begin{figure}[h!]
\centering
\begin{tikzpicture}[scale=1]
\draw (1,0) -- (0,1) -- (-1,0) -- (-3,-1) -- cycle;
\coordinate[fill,cross,inner sep=2pt,rotate=45] (0) at (.5,.5);
\coordinate[fill,cross,inner sep=2pt,rotate=45] (1) at (-.5,.5);
\coordinate[fill,cross,inner sep=2pt,rotate=18.43] (2) at (-2,-.5);
\coordinate[fill,cross,inner sep=2pt,rotate=11.31] (3) at (-1,-.5);
\draw[dashed] (.5,2) -- (0) -- (2,.5);
\draw[dashed] (-.5,2) -- (1) -- (-4,.5);
\draw[dashed] (-4,-.5) -- (2) -- (-4,-.5-2/3);
\draw[dashed] (2,-.5) -- (3) -- (-4,-1.5);
\draw (-.5,1.5) -- (0) -- (1.5,-.5);
\draw (.5,1.5) -- (1) -- (-1.75,-.75);
\draw (-1,0) -- (2,0);
\fill[yellow,opacity=.2] (-1,0) -- (1,0) -- (-5/3,-2/3);
\draw[thick,blue] (.5,.75) -- (2) -- (-4,-1-1/2);
\draw[thick,black!40!green] (2,.25) -- (3) -- (-4,-1-1/4);
\draw[red,->] (-2,-.5) to[bend right=40] (.8,-.2);
\draw[red,->] (-1,-.5) to[bend right=40] (-1.7,-.5);
\draw[->] (2.5,0) -- (3,0);
\end{tikzpicture}
\begin{tikzpicture}[scale=1]
\draw (1,0) -- (0,1) -- (-1,0) -- (-5,-1) -- cycle;
\coordinate[fill,cross,inner sep=2pt,rotate=45] (0) at (.5,.5);
\coordinate[fill,cross,inner sep=2pt,rotate=45] (1) at (-.5,.5);
\coordinate[fill,cross,inner sep=2pt,rotate=14.04] (2) at (-3,-.5);
\coordinate[fill,cross,inner sep=2pt,rotate=9.46] (3) at (-2,-.5);
\draw (-.5,1.5) -- (0) -- (1.5,-.5);
\draw (.5,1.5) -- (1) -- (-1.5,-.5);
\fill[yellow,opacity=.2] (-1,0) -- (0,1) -- (.6,.4);
\draw[thick,black!40!green] (1,.5) -- (2) -- (-6,-1.25);
\draw[thick,blue] (2,-.5+2/3) -- (3) -- (-2-6/2-6/6,-.5-1/2-1/6);
\draw[thick,blue] (-.5,2) -- (.5,2);
\draw[dashed] (-.5,2.2) -- (-.5,.5) -- (-6,.5);
\draw[dashed] (.5,2.2) -- (.5,.5) -- (2,.5);
\draw[dashed] (-6,-.5) -- (-3,-.5) -- (-6,-1.1);
\draw[dashed] (2,-.5) -- (-2,-.5) -- (-6,-1.3);
\draw (0,-1.5);
\end{tikzpicture}
\caption{Scattering diagrams for $\mathbb{F}_3$ and $\mathbb{F}_5$ are related by a translation of rays. The CPS chamber (left) is related to the top chamber (right).}
\label{fig:relation}
\end{figure}

Under this identification, some chambers are related to each other, such that they carry the same theta functions. For example, the chamber below the vertex $(-1,0)$ is related to the chamber below the vertex $(1,0)$. We call the former CPS chamber, since it is considered in \cite{CPS} and the latter the top chamber, and write $\vartheta_1(\mathbb{F}_m)_{\text{CPS}}$ resp. $\vartheta_1(\mathbb{F}_m)_{\text{top}}$.

In particular, this shows that the theta functions for $\mathbb{F}_m$ and $\mathbb{F}_{m+2}$ are mutation equivalent.

This does not work for the relation between $\mathbb{F}_0$ and $\mathbb{F}_2$, but for all other Hirzebruch surfaces. In summary, we have the following.

\begin{proposition}
For $m\geq 1$ we have
\[ \vartheta_1(\mathbb{F}_m)_{\text{CPS}} = \vartheta_1(\mathbb{F}_m)_{\text{top}}. \]
\end{proposition}

\subsection{CPS chamber} 
\label{S:CPS}

Note that $\vartheta_1(\mathbb{F}_m)_{\text{CPS}}$ is a correction of the Hori-Vafa potential $W_\Sigma$, meaning there is some Laurent polynomial $W'$ such that
\[ \vartheta_1(\mathbb{F}_m)_{\text{CPS}} = W_\Sigma + W'. \]
In particular, $\vartheta_1((\mathbb{F}_m,D))_{\text{CPS}}$ is a Laurent polynomial, as follows from Proposition \ref{prop:laurent}. For $m=2,3,4,5$ the theta function is computed in Figure \ref{fig:CPS}. 

\begin{figure}[h!]
\centering
\begin{tikzpicture}[scale=1.6]
\draw (1,0) -- (0,1) -- (-1,0) -- (-2,-1) -- cycle;
\coordinate[fill,cross,inner sep=2pt,rotate=45] (0) at (.5,.5);
\coordinate[fill,cross,inner sep=2pt,rotate=45] (1) at (-.5,.5);
\coordinate[fill,cross,inner sep=2pt,rotate=18.43] (2) at (-1.5,-.5);
\coordinate[fill,cross,inner sep=2pt,rotate=11.31] (3) at (-.5,-.5);
\draw[dashed] (.5,2) -- (0) -- (2,.5);
\draw[dashed] (-.5,2) -- (1) -- (-3,.5);
\draw[dashed] (-3,-.5) -- (2) -- (-3,-.5-1.5/2);
\draw[dashed] (2,-.5) -- (3) -- (-3,-.5-2.5/2);
\draw (-.5,1.5) -- (0) -- (1.5,-.5);
\draw (.5,1.5) -- (1) -- (-2.5,-1.5);
\draw (.5,1.5) -- (2) -- (-2.5,-1.5);
\draw (2,1/3) -- (3) -- (-3,-1-1/3);
\foreach \i in {3,5,...,49}
{
\draw (0,1) -- ({1/\i},2);
\draw (0,1) -- ({-1/\i},2);
}
\fill (0,1) -- (1/50,2) -- (-1/50,2) -- (0,1);
\foreach \i in {9,15,...,147}
{
\draw (-2,-1) -- (-3,{-1-1/2+1/\i});
\draw (-2,-1) -- (-3,{-1-1/2-1/\i});
}
\fill (-2,-1) -- (-3,-1-1/2+1/150) -- (-3,-1-1/2-1/150) -- (-2,-1);
\foreach \i in {1,2,...,20}
{
\draw (1,0) -- (2,{1/10+1/(5+5*\i)});
\draw (1,0) -- ({2-1/(5*\i)},-1/2);
}
\fill (1,0) -- (2,1/9) -- (2,-1/2) -- (1,0);
\fill[red] (-.95,-.42) circle (1pt);
\draw[red] (-.95,-.42) -- (-3,-.42);
\draw[red] (-3,-.32) node[left]{\tiny$z_1x^{-1}$};
\draw[red] (-3,-.52) node[left]{\tiny$z_1z_2x^{-1}$};
\draw[red] (-.95,-.42) -- (2,-.42) node[right]{\tiny$x$};
\draw[red] (-.95,-.42) -- (-.95,.05) -- (-3,.05) node[left]{\tiny$y$};
\draw[red] (-.95,-.42) -- (-3,-1.445) node[left]{\tiny$z_1^2z_2x^{-2}y^{-1}$};
\end{tikzpicture}
\\ [-5mm]
\begin{tikzpicture}[scale=1.6]
\draw (1,0) -- (0,1) -- (-1,0) -- (-3,-1) -- cycle;
\coordinate[fill,cross,inner sep=2pt,rotate=45] (0) at (.5,.5);
\coordinate[fill,cross,inner sep=2pt,rotate=45] (1) at (-.5,.5);
\coordinate[fill,cross,inner sep=2pt,rotate=18.43] (2) at (-2,-.5);
\coordinate[fill,cross,inner sep=2pt,rotate=11.31] (3) at (-1,-.5);
\draw[dashed] (.5,2) -- (0) -- (2,.5);
\draw[dashed] (-.5,2) -- (1) -- (-4,.5);
\draw[dashed] (-4,-.5) -- (2) -- (-4,-.5-2/3);
\draw[dashed] (2,-.5) -- (3) -- (-4,-1.5);
\draw (-.5,1.5) -- (0) -- (1.5,-.5);
\draw (.5,1.5) -- (1) -- (-1.75,-.75);
\draw (.5,.75) -- (2) -- (-4,-1-1/2);
\draw (2,.25) -- (3) -- (-4,-1-1/4);
\draw (-1,0) -- (2,0);
\foreach \i in {3,5,...,49}
{
\draw (0,1) -- ({1/\i},2);
\draw (0,1) -- ({-1/\i},2);
}
\fill (0,1) -- (1/50,2) -- (-1/50,2) -- (0,1);
\foreach \i in {9,15,...,147}
{
\draw (-3,-1) -- (-4,{-1-1/3+1/\i});
\draw (-3,-1) -- (-4,{-1-1/3-1/\i});
}
\fill (-3,-1) -- (-4,-1-1/3+1/150) -- (-4,-1-1/3-1/150) -- (-3,-1);
\foreach \i in {1,2,...,20}
{
\draw (1,0) -- (2,{1/8+1/(5+5*\i)});
\draw (1,0) -- ({2-1/(5*\i)},-1/2);
}
\fill (1,0) -- (2,1/7) -- (2,-1/2) -- (1,0);
\fill[red] (-.95,-.42) circle (1pt);
\draw[red] (-.95,-.42) -- (-4,-.42);
\draw[red] (-4,-.5) node[left]{\tiny$x^{-1}$};
\draw[red] (-.95,-.42) -- (2,-.42) node[right]{\tiny$x$};
\draw[red] (-.95,-.42) -- (-.95,.05) -- (-4,.05) node[left]{\tiny$y$};
\draw[red] (-.95,-.42) -- (-4,-1.43667);
\draw[red] (-4,-1.45) node[left]{\tiny$x^{-3}y^{-1}$};
\draw[red] (-.95,-.42) -- (-1.42,-.42) -- (-4,-1.28);
\draw[red] (-4,-1.25) node[left]{\tiny$2x^{-2}$};
\draw[red] (-.95,-.42) -- (-1.185,-.185) -- (-1.76,-.38) -- (-4,-.38);
\draw[red] (-4,-.3) node[left]{\tiny$x^{-1}y$};
\end{tikzpicture}
\\ [-5mm]
\begin{tikzpicture}[scale=1.6]
\draw (1,0) -- (0,1) -- (-1,0) -- (-4,-1) -- cycle;
\coordinate[fill,cross,inner sep=2pt,rotate=45] (0) at (.5,.5);
\coordinate[fill,cross,inner sep=2pt,rotate=45] (1) at (-.5,.5);
\coordinate[fill,cross,inner sep=2pt,rotate=26.57] (2) at (-2.5,-.5);
\coordinate[fill,cross,inner sep=2pt,rotate=14.04] (3) at (-1.5,-.5);
\draw[dashed] (.5,2) -- (0) -- (2,.5);
\draw[dashed] (-.5,2) -- (1) -- (-5,.5);
\draw[dashed] (-5,-.5) -- (2) -- (-5,-1.125);
\draw[dashed] (2,-.5) -- (3) -- (-5,-1.375);
\draw (-.5,1.5) -- (0) -- (1.5,-.5);
\draw (.5,1.5) -- (1) -- (-1.5,-.5);
\draw (.5,.5) -- (2) -- (-5,-.5-2.5/3);
\draw (2,-.5+3.5/5) -- (3) -- (-5,-.5-3.5/5);
\draw (-1,0) -- (2,0);
\draw (-1,0) -- (-.5,-.5);
\draw (-1,0) -- (.5,-.5);
\draw (-1,0) -- (1.5,-.5);
\foreach \i in {7,9,...,999}
{
\draw (-1,0) -- (2,-2.5/\i);
}
\draw (-1,0) -- (2,.5);
\foreach \i in {7,9,...,999}
{
\draw (-1,0) -- (2,2.5/\i);
}
\foreach \i in {3,5,...,49}
{
\draw (0,1) -- ({1/\i},2);
\draw (0,1) -- ({-1/\i},2);
}
\fill (0,1) -- (1/50,2) -- (-1/50,2) -- (0,1);
\foreach \i in {9,15,...,147}
{
\draw (-4,-1) -- (-5,{-1-1/4+1/\i});
\draw (-4,-1) -- (-5,{-1-1/4-1/\i});
}
\fill (-4,-1) -- (-5,-1-1/4+1/150) -- (-5,-1-1/4-1/150) -- (-4,-1);
\foreach \i in {1,2,...,20}
{
\draw (1,0) -- (2,{1/8+1/(5+5*\i)});
\draw (1,0) -- ({2-1/(5*\i)},-1/2);
}
\fill (1,0) -- (2,1/7) -- (2,-1/2) -- (1,0);
\fill[red] (-.95,-.32) circle (1pt);
\draw[red] (-.95,-.32) -- (-5,-.32) node[left]{\tiny$x^{-1}$};
\draw[red] (-.95,-.32) -- (2,-.32) node[right]{\tiny$x$};
\draw[red] (-.95,-.32) -- (-.95,.05) -- (-5,.05) node[left]{\tiny$y$};
\draw[red] (-.95,-.32) -- (-5,-1.3325);
\draw[red] (-5,-1.4) node[left]{\tiny$x^{-4}y^{-1}$};
\draw[red] (-.95,-.32) -- (-1.32,-.32) -- (-5,-1.245);
\draw[red] (-5,-1.2) node[left]{\tiny$3x^{-3}$};
\draw[red] (-.95,-.32) -- (-1.135,-.135) -- (-2.5+1/5,-.5+1/15) -- (-5,-.5+1/15);
\draw[red] (-5,-.5) node[left]{\tiny$3x^{-2}y$};
\draw[red] (-.95,-.32) -- (-.95,-.05) -- (2,-.05) node[right]{\tiny$y$};
\draw[red] (-.95,-.32) -- (-1.96,-.32) -- (-.68,.32) -- (2,.32) node[right]{\tiny$x^{-1}$};
\end{tikzpicture}
\\ [-5mm]
\begin{tikzpicture}[scale=1.6]
\draw (1,0) -- (0,1) -- (-1,0) -- (-5,-1) -- cycle;
\coordinate[fill,cross,inner sep=2pt,rotate=45] (0) at (.5,.5);
\coordinate[fill,cross,inner sep=2pt,rotate=45] (1) at (-.5,.5);
\coordinate[fill,cross,inner sep=2pt,rotate=14.04] (2) at (-3,-.5);
\coordinate[fill,cross,inner sep=2pt,rotate=9.46] (3) at (-2,-.5);
\draw (1) -- (-1.5,-.5);
\draw (2) -- (1,.5);
\draw (-5,-1) -- (-6,-1-1/6);
\draw (-5,-1) -- (-6,-1-1/5);
\draw (-5,-1) -- (-6,-1-1/4);
\draw (-4,-.5) -- (-2,.5);
\draw (.5,1) -- (.5-1/6,2);
\draw (-.5,.5) -- (.5,1.5);
\draw (0,1) -- (0,2);
\draw (1.5,.5) -- (2,.5-1/6);
\draw (.5,.5) -- (-.5,1.5);
\draw (-1.5,.5) -- (-4.5,-.5);
\draw (-3,0) -- (-6,0);
\draw (.43,1.43) -- (.43,2);
\draw (-1,0) -- (1.75,.5);
\draw (-1,0) -- (1.9,.5);
\draw (-1,0) -- (1.923,.5);
\draw (-1,0) -- (1.9265,.5);
\draw (-1,0) -- (-.75,-.5);
\draw (-1,0) -- (-.6,-.5);
\draw (-1,0) -- (-.577,-.5);
\draw (-1,0) -- (-.574,-.5);
\fill (-1,0) -- (1.927,.5) -- (2,.5) -- (2,-.5) -- (-.573,-1/2) -- (-1,0);
\foreach \i in {3,5,...,49}
{
\draw (0,1) -- ({1/\i},2);
\draw (0,1) -- ({-1/\i},2);
}
\fill (0,1) -- (1/50,2) -- (-1/50,2) -- (0,1);
\foreach \i in {30,31,...,40}
{
\draw (-5,-1) -- (-6,{-1-1/6-1/\i});
\draw (-5,-1) -- (-6,{-1-1/4+1/\i});
}
\fill (-5,-1) -- (-6,-1-1/6-1/40) -- (-6,-1-1/4+1/40) -- (-5,-1);
\draw[blue] (-1.2,-.3) -- (2,-.3);
\draw[blue] (-1.2,-.3) -- (-6,-.3);
\draw[blue] (-1.2,-.3) -- (-1.2,-.2) -- (-6,-.2);
\draw[blue] (-1.2,-.3) -- (-6,-1.26);
\draw[red] (-1.2,-.3) -- (-6,-1.1);
\draw[red] (-3.6,-.5) -- (-1.8,.4) -- (-6,.4);
\draw[red] (-1.2,-.3) -- (-1.3,-.3) -- (-5.5,-1);
\draw[red] (-3.5,-.5) -- (-1.5,.5) -- (-6,.5);
\draw[red] (-1.2,-.3) -- (-1.2-4/50,-.3+1/50) -- (-5,-.9);
\draw[red] (-3.4,-.5) -- (-1.4,.5);
\draw[red] (-.5,1.4) -- (-.4,1.4) -- (-.4,2);
\draw[red] (-1.2,-.3) -- (-1.2-3/50,-.3+2/50) -- (-4.5,-.8);
\draw[red] (-3.3,-.5) -- (-1.3,.5);
\draw[red] (-.5,1.3) -- (-.3,1.3) -- (-.3,2);
\draw[red] (-1.2,-.3) -- (-1.2-2/50,-.3+3/50) -- (-4,-.7);
\draw[red] (-3.2,-.5) -- (-1.2,.5);
\draw[red] (-.5,1.2) -- (-.2,1.2) -- (-.2,2);
\draw[red] (-1.2,-.3) -- (-1.2-1/50,-.3+4/50) -- (-3.5,-.6);
\draw[red] (-3.1,-.5) -- (-1.1,.5);
\draw[red] (-.5,1.1) -- (-.1,1.1) -- (-.1,2);
\draw[green] (-1.2,-.3) -- (-1.3,-.3) -- (-6,-1.24);
\draw[green] (-1.2,-.3) -- (-1.2-3/40,-.3+1/40) -- (-6,-1.24+1/40);
\draw[green] (-1.2,-.3) -- (-1.2-2/40,-.3+2/40) -- (-6,-1.24+2/40);
\draw[green] (-1.2,-.3) -- (-1.2-1/40,-.3+3/40) -- (-6,-1.24+3/40);
\draw[orange] (-1.2,-.3) -- (-2.2,-.3) -- (-.6,.5);
\draw[orange] (-.5,.6) -- (.4,.6) -- (.4,2);
\draw[orange] (-1.2,-.3) -- (-1.4,-.1) -- (-.8,.2) -- (2,.2);
\draw[orange] (-1.2,-.3) -- (-1.2,-.04) -- (-1.2+2/6,-.04+1/6) -- (2,-.04+1/6);
\fill[brown] (-1.2,-.3) circle (1pt);
\draw[dashed] (-.5,2) -- (-.5,.5) -- (-6,.5);
\draw[dashed] (.5,2) -- (.5,.5) -- (2,.5);
\draw[dashed] (-6,-.5) -- (-3,-.5) -- (-6,-1.1);
\draw[dashed] (2,-.5) -- (-2,-.5) -- (-6,-1.3);
\end{tikzpicture}
\caption{$\vartheta_1(\mathbb{F}_m)_{\text{CPS}}$ in the CPS chamber for $m=2,3,4,5$.}
\label{fig:CPS}
\end{figure}

For $m=2$ there is one broken line with ending monomial $\frac{z_1z_2}{x}$ that breaks at two rays at the same time, so that it looks like it doesn't break. Precisely, we get
\begin{align*}
\vartheta_1(\mathbb{F}_0)_{\text{CPS}} &= t \cdot \left( x+y+\frac{z_1}{x}+\frac{z_2}{y} \right) \\
\vartheta_1(\mathbb{F}_1)_{\text{CPS}} &= t \cdot \left( x+y+\frac{z_1}{x}+\frac{z_1z_2}{xy} \right) \\
\vartheta_1(\mathbb{F}_2)_{\text{CPS}} &= t \cdot \left( x+y+\frac{z_1}{x}+\frac{z_1z_2}{x}\left(1+\frac{z_1}{xy}\right) \right) \\
\vartheta_1(\mathbb{F}_3)_{\text{CPS}} &= t \cdot \left( x+y+\frac{z_1}{x}+\frac{z_1z_2y}{x}\left(1+\frac{z_1}{xy}\right)^2 \right) \\
\vartheta_1(\mathbb{F}_4)_{\text{CPS}} &= t \cdot \left( x+y+\frac{z_1}{x}+z_1z_2y\left(1+\frac{z_1}{xy}\right)+\frac{z_1z_2y^2}{x}\left(1+\frac{z_1}{xy}\right)^3 \right)
\end{align*}
One can easily check that the classical periods of $\vartheta_1(\mathbb{F}_m)_0$ agree for $m$ even/odd, as expected. Note that $\mathbb{F}_0$ and $\mathbb{F}_1$ are Fano, so $\vartheta_1$ equals the Hori-Vafa potential $W_\Sigma$. For $\mathbb{F}_2$ and $\mathbb{F}_3$ the potential was computed in \cite{Aur09} using symplectic methods and in \cite{CPS} using broken lines. Note the different conventions. In \cite{CPS} they don't give classes and have the wrong $t$-order. Moreover, they have one broken line that actually isn't there. Instead, one broken line appears with coefficient $2$.

\subsection{Mutations between Hirzebruch potentials} 
\label{S:mutate}

The CPS chamber of $\mathbb{F}_m$ is related to the top chamber of $\mathbb{F}_m$, which is equal to the CPS chamber of $\mathbb{F}_{m-2}$, by three mutations (see \S\ref{S:mutation}), along rays of slope $-1$, $1/(m-1)$ and $-1$, respectively, see Figure \ref{fig:relation2}. 

\begin{figure}[h!]
\centering
\begin{tikzpicture}[scale=1.5]
\draw (1,0) -- (0,1) -- (-1,0) -- (-5,-1) -- cycle;
\coordinate[fill,cross,inner sep=2pt,rotate=45] (0) at (.5,.5);
\coordinate[fill,cross,inner sep=2pt,rotate=45] (1) at (-.5,.5);
\coordinate[fill,cross,inner sep=2pt,rotate=14.04] (2) at (-3,-.5);
\coordinate[fill,cross,inner sep=2pt,rotate=9.46] (3) at (-2,-.5);
\draw (-.5,1.5) -- (0) -- (1.5,-.5);
\draw (.5,1.5) -- (1) -- (-1.5,-.5);
\draw (1,.5) -- (2) -- (-6,-1.25);
\draw (2,-.5+2/3) -- (3) -- (-2-6/2-6/6,-.5-1/2-1/6);
\fill[yellow,opacity=.2] (-1,0) -- (0,1) -- (.6,.4);
\fill[orange,opacity=.2] (-1,0) -- (-11/13,-4/13) -- (-7/5,-2/5);
\draw[red,->] (-1.1,-.3) to[bend left=90] (-.6,.2);
\draw[dashed] (-.5,2.2) -- (-.5,.5) -- (-6,.5);
\draw[dashed] (.5,2.2) -- (.5,.5) -- (2,.5);
\draw[dashed] (-6,-.5) -- (-3,-.5) -- (-6,-1.1);
\draw[dashed] (2,-.5) -- (-2,-.5) -- (-6,-1.3);
\draw (0,-1.5);
\end{tikzpicture}
\caption{Mutations from the CPS chamber to the top chamber.}
\label{fig:relation2}
\end{figure}

In this way, starting from $\mathbb{F}_m$ for $m=0,1,2$ we can successively compute $\vartheta_1$ for any Hirzebruch surface in any chamber, e.g. in the CPS chamber where $\vartheta_1$ is a correction of $W_\Sigma$. In Figures \ref{fig:mut13} and \ref{fig:mut35} we show how $\vartheta_1(\mathbb{F}_1)_{\text{CPS}}$ mutates to $\vartheta_1(\mathbb{F}_3)_{\text{CPS}}$ and then to $\vartheta_1(\mathbb{F}_5)_{\text{CPS}}$.

\begin{figure}[h!]
\centering
\begin{tikzpicture}[scale=.6]
\draw (1,0) -- (0,1) -- (-1,0) -- (-1,-1) -- cycle;
\draw (0,0) circle (2pt);
\fill (1,0) node[below]{$1$} circle (2pt);
\fill (0,1) node[below]{$1$} circle (2pt);
\fill (-1,0) node[below]{$1$} circle (2pt);
\fill (-1,-1) node[below]{$1$} circle (2pt);
\draw[red] (-1,0) -- (0,1);
\draw[->] (2,0) -- (3,0);
\end{tikzpicture}
\begin{tikzpicture}[scale=.6]
\draw (1,0) -- (2,1) -- (-1,0) -- (-1,-1) -- cycle;
\draw (0,0) circle (2pt);
\fill (1,0) node[below]{$1$} circle (2pt);
\fill (2,1) node[below]{$1$} circle (2pt);
\fill (-1,0) node[below]{$1$} circle (2pt);
\fill (-1,-1) node[below]{$1$} circle (2pt);
\draw[red] (1,0) -- (-1,-1);
\draw[->] (2,0) -- (3,0);
\end{tikzpicture}
\begin{tikzpicture}[scale=.6]
\draw (1,0) -- (2,1) -- (-1,0) -- (-3,-1) -- cycle;
\draw (0,0) circle (2pt);
\fill (1,0) node[below]{$1$} circle (2pt);
\fill (2,1) node[below]{$1$} circle (2pt);
\fill (-1,0) node[below]{$1$} circle (2pt);
\fill (-3,-1) node[below]{$1$} circle (2pt);
\draw[red] (1,0) -- (2,1);
\draw[->] (2,0) -- (3,0);
\end{tikzpicture}
\begin{tikzpicture}[scale=.6]
\draw (1,0) -- (0,1) -- (-1,1) -- (-3,-1) -- cycle;
\draw (0,0) circle (2pt);
\fill (1,0) node[below]{$1$} circle (2pt);
\fill (0,1) node[below]{$1$} circle (2pt);
\fill (-1,0) node[below]{$1$} circle (2pt);
\fill (-3,-1) node[below]{$1$} circle (2pt);
\fill (-1,1) node[below]{$1$} circle (2pt);
\fill (-2,0) node[below]{$2$} circle (2pt);
\end{tikzpicture}
\caption{Mutation from $\vartheta_1(\mathbb{F}_1)_{\text{CPS}}$ to $\vartheta_1(\mathbb{F}_3)_{\text{CPS}}$.}
\label{fig:mut13}
\end{figure}

\begin{figure}[h!]
\centering
\begin{tikzpicture}[scale=.6]
\draw (1,0) -- (0,1) -- (-1,1) -- (-3,-1) -- cycle;
\draw (0,0) circle (2pt);
\fill (1,0) node[below]{$1$} circle (2pt);
\fill (0,1) node[below]{$1$} circle (2pt);
\fill (-1,0) node[below]{$1$} circle (2pt);
\fill (-3,-1) node[below]{$1$} circle (2pt);
\fill (-1,1) node[below]{$1$} circle (2pt);
\fill (-2,0) node[below]{$2$} circle (2pt);
\draw[red] (-3,-1) -- (-1,1);
\draw[->] (2,0) -- (3,0);
\end{tikzpicture}
\begin{tikzpicture}[scale=.6]
\draw (1,0) -- (2,1) -- (-1,0) -- (-3,-1) -- cycle;
\draw (0,0) circle (2pt);
\fill (1,0) node[below]{$1$} circle (2pt);
\fill (2,1) node[below]{$1$} circle (2pt);
\fill (-1,0) node[below]{$1$} circle (2pt);
\fill (-3,-1) node[below]{$1$} circle (2pt);
\draw[red] (1,0) -- (-3,-1);
\draw[->] (2,0) -- (3,0);
\end{tikzpicture}
\begin{tikzpicture}[scale=.6]
\draw (1,0) -- (2,1) -- (-6,-1) -- (-5,-1) -- cycle;
\draw (0,0) circle (2pt);
\fill (1,0) node[below]{$1$} circle (2pt);
\fill (2,1) node[below]{$1$} circle (2pt);
\fill (-1,0) node[below]{$1$} circle (2pt);
\fill (-2,0) node[below]{$2$} circle (2pt);
\fill (-5,-1) node[below]{$1$} circle (2pt);
\fill (-6,-1) node[below]{$1$} circle (2pt);
\draw[red] (1,0) -- (2,1);
\end{tikzpicture}
\\[3mm]
\begin{tikzpicture}[scale=.6]
\draw[->] (-8,2) -- (-7,2);
\draw (1,0) -- (-1,4) -- (-6,-1) -- (-5,-1) -- cycle;
\draw (0,0) circle (2pt);
\fill (1,0) node[below]{$1$} circle (2pt);
\fill (0,1) node[below]{$1$} circle (2pt);
\fill (0,2) node[below]{$2$} circle (2pt);
\fill (-1,0) node[below]{$1$} circle (2pt);
\fill (-1,1) node[below]{$4$} circle (2pt);
\fill (-1,2) circle (1pt);
\fill (-1,3) node[below]{$1$} circle (2pt);
\fill (-2,0) node[below]{$2$} circle (2pt);
\fill (-2,1) circle (1pt);
\fill (-2,2) node[below]{$4$} circle (2pt);
\fill (-3,0) circle (1pt);
\fill (-3,1) node[below]{$6$} circle (2pt);
\fill (-4,0) node[below]{$4$} circle (2pt);
\fill (-5,-1) node[below]{$1$} circle (2pt);
\fill (-1,4) node[below]{$1$} circle (2pt);
\fill (-2,3) node[below]{$5$} circle (2pt);
\fill (-3,2) node[below]{$10$} circle (2pt);
\fill (-4,1) node[below]{$10$} circle (2pt);
\fill (-5,0) node[below]{$5$} circle (2pt);
\fill (-6,-1) node[below]{$1$} circle (2pt);
\end{tikzpicture}
\caption{Mutation from $\vartheta_1(\mathbb{F}_3)_{\text{CPS}}$ to $\vartheta_1(\mathbb{F}_5)_{\text{CPS}}$.}
\label{fig:mut35}
\end{figure}

\section{Blow ups of the plane} 
\label{S:Bl}

We consider blow up of $\mathbb{P}^2$ in $k\leq 8$ general points, $\text{Bl}^k\mathbb{P}^2=dP_{9-k}$. These are del Pezzo surfaces, i.e. smooth and Fano, of anticanonical degree $(-K_X)^2=9-k$. If the blow ups are toric, i.e., the successive blow up in torus fixed points, we obtain a toric variety $X_\Sigma$ to which $\text{Bl}^k\mathbb{P}^2$ admits a $\mathbb{Q}$-Gorenstein degeneration. For $k=1,2,3$ the toric variety is always Fano, for $k=4,5,6$ it might be Fano or non-Fano, depending on the choice of points, and for $k>6$ it is always non-Fano.

\subsection{Four points} 

Consider $\text{Bl}^4\mathbb{P}^2=dP_5$. Figure \ref{fig:bl4model} shows two different toric varieties $X_{\Sigma_1}$ and $X_{\Sigma_2}$, given by non-general blow ups of $\mathbb{P}^2$, to which $\text{Bl}^4\mathbb{P}^2$ admits a $\mathbb{Q}$-Gorenstein degeneration. 

\begin{figure}[h!]
\centering
\begin{tikzpicture}[scale=1.4]
\draw[thick,->] (0,0) -- (1,0);
\draw[thick,->] (0,0) -- (0,1);
\draw[thick,->] (0,0) -- (-1,-1);
\draw[->] (0,0) -- (1,1);
\draw[->] (0,0) -- (-1,0);
\draw[->] (0,0) -- (0,-1);
\draw[->] (0,0) -- (1,-1);
\end{tikzpicture}
\hspace{2cm}
\begin{tikzpicture}[scale=1.4]
\draw[thick,->] (0,0) -- (1,0);
\draw[thick,->] (0,0) -- (0,1);
\draw[thick,->] (0,0) -- (-1,-1);
\draw[->] (0,0) -- (0,-1);
\draw[->] (0,0) -- (1,-1);
\draw[->] (0,0) -- (2,-1);
\draw[->] (0,0) -- (3,-1);
\end{tikzpicture}
\caption{The fans of a Fano (left) and a non-Fano (right) $\mathbb{Q}$-Gorenstein degeneration of $\text{Bl}^4\mathbb{P}^2=dP_5$. The fan of $\mathbb{P}^2$ is bold.}
\label{fig:bl4model}
\end{figure}

$X_{\Sigma_1}$ is Fano while $X_{\Sigma_2}$ is non-Fano, since its spanning polytope is non-convex. They induce different toric degenerations of $\text{Bl}^4\mathbb{P}^2$, hence different scattering diagrams and theta functions, as shown in Figure \ref{fig:bl4theta}. $X_{\Sigma_1}$ is Fano, hence $\vartheta_1(X_{\Sigma_1})_0=W_\Sigma$ inside the central chamber. Figure \ref{fig:bl4mut} shows how the theta functions for the different cases are related by mutation.

\begin{figure}[h!]
\centering
\begin{tikzpicture}[scale=1.1]
\draw (-1,0) -- (0,1) -- (1,1) -- (1,0) -- (1,-1) -- (0,-1) -- (-1,-1) -- cycle;
\coordinate[fill,cross,inner sep=2pt,rotate=45] (0) at (-.5,.5);
\coordinate[fill,cross,inner sep=2pt,rotate=0] (1) at (.5,1);
\coordinate[fill,cross,inner sep=2pt,rotate=0] (2) at (-1,-.5);
\coordinate[fill,cross,inner sep=2pt,rotate=0] (3) at (1,.5);
\coordinate[fill,cross,inner sep=2pt,rotate=0] (4) at (1,-.5);
\coordinate[fill,cross,inner sep=2pt,rotate=0] (5) at (-.5,-1);
\coordinate[fill,cross,inner sep=2pt,rotate=0] (6) at (.5,-1);
\draw[dashed] (-2,.5) -- (0) -- (-.5,2);
\draw[dashed] (.5,2) -- (1) -- (1.5,2);
\draw[dashed] (-2,-.5) -- (2) -- (-2,-1.5);
\draw[dashed] (2,1.5) -- (3) -- (2,.5);
\draw[dashed] (2,-.5) -- (4) -- (2,-1.5);
\draw[dashed] (-1.5,-2) -- (5) -- (-.5,-2);
\draw[dashed] (.5,-2) -- (6) -- (1.5,-2);
\draw (-1,0) -- (-2,0);
\draw (0,1) -- (0,2);
\draw (1,1) -- (2,2);
\draw (1,-1) -- (2,-2);
\draw (-1,-1) -- (-2,-2);
\fill[red] (.2,-.1) circle (1.5pt);
\draw[red] (.2,-.1) -- (2,-.1);
\draw[red] (.2,-.1) -- (-2,-.1);
\draw[red] (.2,-.1) -- (.2,2);
\draw[red] (.2,-.1) -- (.2,-2);
\draw[red] (.2,-.1) -- (-1.7,-2);
\draw[red] (.2,-.1) -- (2,1.7);
\draw[red] (.2,-.1) -- (2,-1.9);
\end{tikzpicture}
\hspace{1cm}
\begin{tikzpicture}[scale=1.5]
\draw (-1,-1) -- (0,1) -- (1,0) -- (3,-1) -- cycle;
\coordinate[fill,cross,inner sep=2pt,rotate=-26.57] (0) at (-.5,0);
\coordinate[fill,cross,inner sep=2pt,rotate=45] (1) at (.5,.5);
\coordinate[fill,cross,inner sep=2pt,rotate=-26.57] (2) at (2,-.5);
\coordinate[fill,cross,inner sep=2pt,rotate=0] (3) at (-.5,-1);
\coordinate[fill,cross,inner sep=2pt,rotate=0] (4) at (.5,-1);
\coordinate[fill,cross,inner sep=2pt,rotate=0] (5) at (1.5,-1);
\coordinate[fill,cross,inner sep=2pt,rotate=0] (6) at (2.5,-1);
\draw[dashed] (-.5,1.5) -- (0) -- (-1.5,-1);
\draw[dashed] (.5,1.5) -- (1) -- (4.5,.5);
\draw[dashed] (4.5,-.5) -- (2) -- (4.5,-.5-2.5/3);
\draw[dashed] (-1,-1.5) -- (3) -- (-.5,-1.5);
\draw[dashed] (.5,-1.5) -- (4) -- (1,-1.5);
\draw[dashed] (2,-1.5) -- (5) -- (2.5,-1.5);
\draw[dashed] (3.5,-1.5) -- (6) -- (4,-1.5);
\draw (0,1) -- (0,1.5);
\draw (-1,-1) -- (-1.5,-1.5);
\draw (3,-1) -- (4.5,-1.5);
\draw (1,0) -- (-.5,0);
\draw (1) -- (2.5,-1.5);
\draw (2) -- (-.5,.75);
\fill[red] (.2,-.2) circle (1pt);
\draw[red] (.2,-.2) -- (4.5,-.2);
\draw[red] (.2,-.2) -- (.2,1.5);
\draw[red] (.2,-.2) -- (.2-1.3,-.2-1.3);
\draw[red] (.2,-.2) -- (.2+0,-.2-1.3);
\draw[red] (.2,-.2) -- (.2+1.3,-.2-1.3);
\draw[red] (.2,-.2) -- (.2+2.6,-.2-1.3);
\draw[red] (.2,-.2) -- (.2+3.9,-.2-1.3);
\draw[red] (.2,-.2) -- (.4,0) -- (.4,1.5);
\draw[red] (.2,-.2) -- (1.2,-.2) -- (2.8,-1) -- (4.3,-1.5);
\draw[red] (.2,-.2) -- (1.2,-.2) -- (1.8,-.4) -- (4.5,-.4);
\end{tikzpicture}
\caption{Blow up in $4$ points.}
\label{fig:bl4theta}
\end{figure}

\begin{figure}[h!]
\centering
\begin{tikzpicture}[scale=.7]
\draw (1,0) -- (1,1) -- (0,1) -- (-1,0) -- (-1,-1) --(1,-1) -- cycle;
\draw (0,0) circle (2pt);
\fill (1,0) node[below]{$2$} circle (2pt);
\fill (1,1) node[below]{$1$} circle (2pt);
\fill (0,1) node[below]{$1$} circle (2pt);
\fill (-1,0) node[below]{$1$} circle (2pt);
\fill (-1,-1) node[below]{$1$} circle (2pt);
\fill (0,-1) node[below]{$2$} circle (2pt);
\fill (1,-1) node[below]{$1$} circle (2pt);
\draw[red] (0,1) -- (1,1);
\draw[->] (2,0) -- (3,0);
\end{tikzpicture}
\begin{tikzpicture}[scale=.7]
\draw (-1,-1) -- (-1,0) --(0,1) -- (2,-1) -- cycle;
\draw (0,0) circle (2pt);
\fill (-1,0) node[below]{$1$} circle (2pt);
\fill (0,1) node[below]{$1$} circle (2pt);
\fill (1,0) node[below]{$2$} circle (2pt);
\fill (-1,-1) node[below]{$1$} circle (2pt);
\fill (0,-1) node[below]{$3$} circle (2pt);
\fill (1,-1) node[below]{$3$} circle (2pt);
\fill (2,-1) node[below]{$1$} circle (2pt);
\draw[red] (-1,-1) -- (-1,0);
\draw[->] (2,0) -- (3,0);
\end{tikzpicture}
\begin{tikzpicture}[scale=.7]
\draw (-1,-1) -- (0,1) -- (2,1) -- (2,-1) -- cycle;
\draw (0,0) circle (2pt);
\fill (0,1) node[below]{$1$} circle (2pt);
\fill (1,1) node[below]{$2$} circle (2pt);
\fill (1,0) node[below]{$5$} circle (2pt);
\fill (1,-1) node[below]{$3$} circle (2pt);
\fill (2,1) node[below]{$1$} circle (2pt);
\fill (2,0) node[below]{$2$} circle (2pt);
\fill (2,-1) node[below]{$1$} circle (2pt);
\fill (-1,-1) node[below]{$1$} circle (2pt);
\fill (0,-1) node[below]{$3$} circle (2pt);
\draw[red] (1,1) -- (2,1);
\draw[->] (3,0) -- (4,0);
\end{tikzpicture}
\begin{tikzpicture}[scale=.7]
\draw (-1,-1) -- (0,1) -- (1,1) -- (3,-1) -- cycle;
\draw (0,0) circle (2pt);
\fill (0,1) node[below]{$1$} circle (2pt);
\fill (1,1) node[below]{$1$} circle (2pt);
\fill (1,0) node[below]{$5$} circle (2pt);
\fill (2,0) node[below]{$2$} circle (2pt);
\fill (-1,-1) node[below]{$1$} circle (2pt);
\fill (0,-1) node[below]{$4$} circle (2pt);
\fill (1,-1) node[below]{$6$} circle (2pt);
\fill (2,-1) node[below]{$4$} circle (2pt);
\fill (3,-1) node[below]{$1$} circle (2pt);
\end{tikzpicture}
\caption{Mutation from $\vartheta_1(X_{\Sigma_1})_0$ to $\vartheta_1(X_{\Sigma_2})_0$.}
\label{fig:bl4mut}
\end{figure}

\subsection{Seven points} 

Consider $\text{Bl}^7\mathbb{P}^2=dP_2$. This does not admit a $\mathbb{Q}$-Gorenstein degeneration to a Fano or semi-fano toric variety, hence lies outside the scope of previous works. 
Figure \ref{fig:bl7theta} shows $\vartheta_1(dP_2)$ in a chamber close to the origin and how it is computed. It is mutation equivalent to the maximally mutable Laurent polynomial for $dP_2$ found in \cite{KNP}, see \cite{ACH}, Figure 1, case $S_2^2$.

\begin{figure}[h!]
\centering
\begin{tikzpicture}[scale=1]
\draw (-1,-1) -- (3,-1) -- (-1,3) -- cycle;
\draw (0,0) circle (2pt);
\fill (-1,-1) node[below]{$1$} circle (2pt);
\fill (0,-1) node[below]{$4$} circle (2pt);
\fill (1,-1) node[below]{$6$} circle (2pt);
\fill (2,-1) node[below]{$4$} circle (2pt);
\fill (3,-1) node[below]{$1$} circle (2pt);
\fill (-1,0) node[below]{$4$} circle (2pt);
\fill (1,0) node[below]{$8$} circle (2pt);
\fill (2,0) node[below]{$4$} circle (2pt);
\fill (-1,1) node[below]{$6$} circle (2pt);
\fill (0,1) node[below]{$8$} circle (2pt);
\fill (1,1) node[below]{$6$} circle (2pt);
\fill (-1,2) node[below]{$4$} circle (2pt);
\fill (0,2) node[below]{$4$} circle (2pt);
\fill (-1,3) node[below]{$1$} circle (2pt);
\draw (0,-3);
\end{tikzpicture}
\hspace{1cm}
\begin{tikzpicture}[scale=1.5]
\draw (-1,-1) -- (-1,2) -- (0,1) -- (1,1) -- (1,0) -- (2,-1) -- cycle;
\coordinate[fill,cross,inner sep=2pt,rotate=45] (0) at (-.5,1.5);
\coordinate[fill,cross,inner sep=2pt,rotate=0] (1) at (.5,1);
\coordinate[fill,cross,inner sep=2pt,rotate=0] (2a) at (-1,.5);
\coordinate[fill,cross,inner sep=2pt,rotate=0] (2b) at (-1,1.5);
\coordinate[fill,cross,inner sep=2pt,rotate=0] (2) at (-1,-.5);
\coordinate[fill,cross,inner sep=2pt,rotate=0] (3) at (1,.5);
\coordinate[fill,cross,inner sep=2pt,rotate=0] (4) at (1.5,-.5);
\coordinate[fill,cross,inner sep=2pt,rotate=0] (5) at (-.5,-1);
\coordinate[fill,cross,inner sep=2pt,rotate=0] (6) at (.5,-1);
\coordinate[fill,cross,inner sep=2pt,rotate=0] (7) at (1.5,-1);
\draw[dashed] (-.5,3) -- (0) -- (-1.25,3);
\draw[dashed] (.5,3) -- (1) -- (2.5,3);
\draw[dashed] (-2,.5) -- (2a) -- (-2,1.5);
\draw[dashed] (-2,2.5) -- (2b) -- (-1.75,3);
\draw[dashed] (-2,-.5) -- (2) -- (-2,-1.5);
\draw[dashed] (3,2.5) -- (3) -- (3,.5);
\draw[dashed] (3,-.5) -- (4) -- (3,-1.25);
\draw[dashed] (-1.5,-2) -- (5) -- (-.5,-2);
\draw[dashed] (.5,-2) -- (6) -- (1.5,-2);
\draw[dashed] (2.5,-2) -- (7) -- (3,-1.75);
\draw (1,0) -- (-2,0);
\draw (0,1) -- (0,-2);
\draw (-1,-1) -- (-2,-2);
\draw (1,1) -- (3,3);
\draw (2,-1) -- (3,-1.5);
\draw (-1,2) -- (-1.5,3);
\draw (0,0) -- (-2,-2);
\draw (1,.5) -- (1,-1.5);
\draw (.5,1) -- (-1.5,1);
\draw (1,0) -- (0,1);
\fill[red] (.1,.15) circle (1pt);
\draw[red] (.1,.15) -- (3,.15);
\draw[red] (.1,.15) -- (.1,3);
\draw[red] (.1,.15) -- (-2,.15);
\draw[red] (.1,.15) -- (.1,-2);
\draw[red] (.1,.15) -- (-2,-1.95);
\draw[red] (.1,.15) -- (-2,2.25);
\draw[red] (.1,.15) -- (2.25,-2);
\draw[red] (.1,.15) -- (-1.325,3);
\draw[red] (.1,.15) -- (3,-1.275);
\draw[red] (.1,.15) -- (2.95,3);
\draw[red] (.1,.15) -- (.1,0) -- (-1.9,-2);
\draw[red] (.1,.15) -- (.25,0) -- (.25,-2);
\draw[red] (.1,.15) -- (.4,0) -- (2.4,-2);
\draw[red] (.1,.15) -- (0,.15) -- (-2,-1.85);
\draw[red] (.1,.15) -- (0,.25) -- (-2,.25);
\draw[red] (.1,.15) -- (0,.35) -- (-2,2.35);
\draw[red] (.1,.15) -- (.1,.9) -- (1,.9) -- (3,2.9);
\draw[red] (.1,.15) -- (.85,.15) -- (.85,1) -- (2.85,3);
\draw[red] (.1,.15) -- (.1,.9) -- (.2,1) -- (.2,3);
\draw[red] (.1,.15) -- (.85,.15) -- (1,.3) -- (3,.3);
\draw[red] (.1,.15) -- (.55,0) -- (1.45,-.45) -- (3,-.45);
\draw[red] (.1,.15) -- (0,.45) -- (-1.275,3);
\end{tikzpicture}
\caption{$\vartheta_1(dP_2)$ in a chamber close to the origin.}
\label{fig:bl7theta}
\end{figure}


\end{document}